\documentclass[a4paper, 11pt]{amsart}
\usepackage[margin=1in]{geometry}

\usepackage{graphicx}
\usepackage{algorithm}
\usepackage{algpseudocode}
\usepackage{amsmath}
\usepackage{amsthm}
\usepackage{amsfonts}
\usepackage{amssymb}
\usepackage{bbm}
\usepackage{bm}
\usepackage{mathrsfs}
\usepackage{mathtools}
\usepackage{thmtools}

\usepackage{fancyref}
\usepackage{hyperref}
\usepackage{listings}
\usepackage{nameref}
\usepackage[section]{placeins}
\usepackage{soul}
\usepackage[normalem]{ulem}
\hypersetup{colorlinks=true, linkcolor=black}
\usepackage{xcolor}

\newtheorem{thm}{Theorem}[section]
\newtheorem{lem}[thm]{Lemma}

\newtheorem{defi}[thm]{Definition}

\newtheorem{prop}[thm]{Proposition}

\newtheorem{rmk}[thm]{Remark}

\declaretheoremstyle[headfont=\normalfont]{normalhead}
\declaretheorem[style=normalhead]{remark}

\newcommand{\rd}{\mathrm{d}}
\newcommand{\N}{\mathbb{N}}

\newcommand{\R}{\mathbb{R}}

\newcommand{\defeq}{\vcentcolon=}
\newcommand{\eqdef}{=\vcentcolon}

\DeclareMathOperator{\supp}{supp}

\DeclareMathOperator*{\esssup}{ess\,sup}
\DeclareMathOperator*{\law}{Law}

\def\XXint#1#2#3{{\setbox0=\hbox{$#1{#2#3}{\int}$ }
\vcenter{\hbox{$#2#3$ }}\kern-.6\wd0}}

\graphicspath{ {./HW1/} }
\numberwithin{equation}{section}

\begin{document}

\title{The mean-field limit of sparse networks of integrate and fire neurons}
\author[P.--E.~Jabin]{Pierre-Emmanuel Jabin}
\address{\sc P.--E. Jabin.  Department of Mathematics and Huck Institutes, Pennsylvania State University, State College, PA 16801, USA}
\email{pejabin@psu.edu}
\thanks{P-E Jabin and D. Zhou were partially supported by NSF DMS Grants 2205694, 2219297.}

\author[D. Zhou]{Datong Zhou}
\address{\sc D. Zhou.  Department of Mathematics, Pennsylvania State University, State College, PA 16801, USA}
\email{dbz5086@psu.edu}

\date{}

\maketitle

\begin{abstract}

We study the mean-field limit of a model of biological neuron networks based on the so-called stochastic integrate-and-fire (IF) dynamics. Our approach allows to derive a continuous limit for the macroscopic behavior of the system, the 1-particle distribution, for a large number of neurons with no structural assumptions on the connection map outside of a generalized mean-field scaling. 
We propose a novel notion of observables that naturally extends the notion of marginals to systems with non-identical or non-exchangeable agents. Our 
new observables satisfy a complex approximate hierarchy, essentially a tree-indexed extension of the classical BBGKY hierarchy. We are able to pass to the limit in this hierarchy as the number of neurons increases through novel quantitative stability estimates in some adapted weak norm. 
While we require non-vanishing diffusion, this approach notably addresses the challenges of sparse interacting graphs/matrices and singular interactions from Poisson jumps, and requires no additional regularity on the initial distribution. 

\end{abstract}

\tableofcontents

\section{Introduction} \label{sec:introduction}
This article derives a continuous limit for the large-scale behavior of networks of neurons following a type of dynamics known as integrate-and-fire (IF). It is a natural example of \emph{multi-agent systems}, where each agent (neuron) could influence others and be influenced in return. However because each neuron has a priori different connections to other neurons, it is also an important example of non-exchangeable systems.

We focus on IF systems for large number of agents or neurons, typically $86 \times10^9$ in a human brain for example. This makes it quite challenging to study the original system, either numerically or analytically. 
Instead one can try to approach the large scale behavior of such multi-agent systems  through the concept of \emph{mean-field limit}. In classical exchangeable systems, the mean-field limit consists in replacing the exact 
 influence exerted on one particle by its expectation or mean. It is hence connected to the famous notion of propagation of chaos which allows the use of a law of large number to rigorously justify this approximation. However in non-exchangeable systems, the derivation of the mean-field limit also requires a way to capture the limit of the non-identical interactions between particles or agents.

 This article introduces a novel strategy based on a new concept of observables that are well chosen linear combinations of empirical laws of agents or neurons. This family solves a tree-indexed hierarchy, approximately for a finite number of neurons and exactly at the limit; a key feature of this hierarchy is that the connection weights between neurons does not appear explicitly anymore. As a consequence, the mean-field limit can be derived directly by passing to the limit in the hierarchy, bypassing a priori structural assumptions on the connection weights. In particular our result is entirely compatible with \emph{sparse} connection weights, as supported by experimental findings in neuroscience. However the IF-type dynamics involve jump processes in time, which inevitably introduce discontinuities. Therefore, at the technical level, a major contribution of this article is the development of well adapted weak norms that provide quantitative stability estimates.

\subsection{An IF neuron network with non-identical sparse connections} \label{sec:introduction_non_math}
We focus in this article on a type of stochastic integrate and fire models. In this model, neurons interact through ``spikes'' that represent short electrical pulses in the membrane potential, typically lasting 1-2 ms. A broad range of IF models adopt the following theoretical simplification that dates back to the earliest mathematical model of neuron \cite{La:07} as well as \cite{Hi:36, McPi:43}.

Spikes occur at distinct points in time, initiating what is typically referred to as a ``fire'' event.
For a network of IF neurons, at the exact time when the $i$-th neuron fires,
\begin{equation*}
\begin{aligned}
\text{for all } j \text{ connected to } i,\; X^j \text{ jumps by } w_{j,i},
\end{aligned}
\end{equation*}
where $X^j$ describes the membrane potential of the $j$-th neuron and $w_{j,i}$ represents the \emph{synaptic connection} from $i$ to $j$. The case of no synaptic connection is represented by $w_{j,i} = 0$. 

There exists a large variety of models with various rules to determine when a neuron is firing and what is the evolution of the membrane potential between spikes. In the seminal work \cite{La:07}, the firing of neuron $i$ is predicted at the time $X^i$ reaches a certain hard threshold value $X_F$. According to IF dynamics, at such a time point each $X^j$ jumps by $w_{j,i}$ and $X^i$ is reset to zero. However we consider instead in the present paper a notion of soft threshold where the firing of each neuron follows independent Poisson process with a rate that depends on the membrane potential.

When there is no firing, the ``pre-spike'' dynamics of membrane potential is usually given by a simple ODE or SDE,  which we may write in our case as
\begin{equation*}
\begin{aligned}
\rd X^i(t) = \mu(X^i(t))\, \rd t+\sigma(X^i(t))\, \rd t. 
\end{aligned}
\end{equation*}
As mentioned earlier, there exists a large variety of IF models in terms  of the equations for pre-spike dynamics and criteria for firing. From the point of view of the mathematical analysis developed in this paper, both the stochasticity in the SDE equation on $X^i(t)$ and the soft threshold are needed.

The non-linearity of pre-spike dynamics has been observed in modern experimental studies such as \cite{BaLeBePeGeRi:08a}, and stochasticity was noted in~\cite{GeMa:64, GeGo:66,Ki:72b}. Although biophysical models such as Hodgkin-Huxley \cite{HoHu:52} and FitzHugh-Nagumo \cite{Fi:61, NaArYo:62} are available for more accurately representing the shape of each spike, IF type dynamics are frequently preferred for their perceived precision when investigating multiple-neuron networks. Nevertheless the present is still a compromise between mathematical succinctness and biological plausibility.
Some extended mathematical models that aim to capture more complex neuronal phenomena have also been studied, for example, in \cite{CaPe:14,PaPeSa:10,PeSa:13}.
For a more extensive discussion of IF models in the context of neuroscience, we refer to \cite{Bu:06, GeKi:02, GeKiNaPa:14} and the references therein. For a more thorough exploration of the biological considerations, we direct interested readers to references \cite{GeKiNaPa:14, Sp:10}.

To complete the definition at end points, it is conventional to define $X^i(t)$ at a firing time as the value \emph{after} the jump or reset, making each $X^i(t)$ right continuous with left limit (c\`adl\`ag functions). 
This allows to give a precise mathematical definition of the dynamics.
%
Let $(X^i_t)_{i=1}^N$ be the $\R$-valued c\`adl\`ag processes representing the membrane potential changes of the $N$ neurons and let $w_N \defeq (w_{i,j;N})_{i,j =1}^N$ be the interaction matrix describing the synaptic connection between these neurons.
The IF-type dynamics of neurons are characterized by the following SDE in integral form holding for all $i \in \{1, \dots, N\}$:
\begin{equation} \label{eqn:IF_multi_agent}
\begin{aligned}
X^{i;N}_t = X^{i;N}_0 \;&+ \int_0^t \mu(X^{i;N}_{s^-}) \;\rd s + \int_0^t \sigma(X^{i;N}_{s^-}) \;\rd \bm{B}^i_s
\\
\;&+ \sum_{j \neq i} w_{i,j;N} \int_0^t \int_0^\infty \mathbbm{1}\{z \leq \nu(X^{j;N}_{s^-})\} \;\bm{N}^j (\rd z, \rd s)
\\
\;&- \int_0^t \int_0^\infty X^{i;N}_{s^-} \mathbbm{1} \{z \leq \nu(X^{i;N}_{s^-})\} \bm{N}^i (\rd z,\rd s),
\end{aligned}
\end{equation}
where
\begin{equation*}
\begin{aligned}
\;& \{\bm{N}^i\}_{i=1}^N \text{ are homogeneous spatial Poisson processes w.r.t. Lebesgue measure,}
\\
\;& \{\bm{B}^i\}_{i=1}^N \text{ are standard Wiener processes, and the $2N$ processes are independent.}
\end{aligned}
\end{equation*}
For the target neuron $i$, the term $\mu(X^{i;N}_{s^-}) \rd s$ summarizes its pre-spike dynamics and $\sigma(X^{i;N}_{s^-}) \rd \bm{B}^i_s$ adds a Brownian noise. 
It experiences a jump of $w_{i,j;N}$ when another neuron $j$ fires and is reset to zero when itself fires.
Neuron $i$ firing occurs with a likelihood depending on the membrane potential, which we denote by $\nu(X^{i;N}_{s^-})$ and we introduce the Poisson processes $\bm{N}^i (\rd z,\rd s)$.

For the simplified case that the connections between neurons are all identical, i.e. $w_{i,j;N} = 1/N$, $\forall \{i,j\} \in \{1,\dots,N\}$, the mean-field limit of \eqref{eqn:IF_multi_agent} or its variations can be expressed as a PDE about the (time-varying) density function $f(t,x)$, where $x \in \R$ represents the membrane potential.
We mention  \cite{RiThWa:12} that employs a PDE-based approach, and \cite{DeGaLoPr:15, DeInRuTa:15b, ErLoLo:21} that each offer a distinct probabilistic perspective. Though significantly different from \eqref{eqn:IF_multi_agent}, Hawkes processes give another type of popular models for biological neuron networks and their mean-field limit  has also been studied, as in \cite{Ch:17a, DeFoHo:16}. 
We also cite \cite{BaFaFaTo:12} for the study of large  biophysical models with Hodgkin-Huxley and FitzHugh-Nagumo equations for the neurons, together with \cite{PaThWa:10} which derives an IF model from biophysical models in a mean-field setting.
Even in the case of identical connections, we emphasize that some neuron models may contain singularities that lead to important mathematical challenges when deriving the mean-field limit.

While assuming identical connections is a significant simplification, the derived mean-field limits have nonetheless provided useful insights into our understanding of large biological neuron networks.
For some limiting models, the mean-field equations can for example exhibit blow-up in finite time, which may  represent some large-scale synchronization within the network, see for instance \cite{CaCaPe:11, CaGoGuSc:13, CaPeSaSm:15} from a PDE perspective, and \cite{DeInRuTa:15a} from a probability point of view. The issue of convergence to equilibrium in the mean-field limit is also an important question, for which we refer for example to 
\cite{FoLo:16} and \cite{DrVe:21}. Other studies, such as \cite{CoTaVe:20a, CoTaVe:20b, Co:20}, have explored the spectral conditions sufficient for the existence of periodic solutions near the invariant measure through a Hopf bifurcation.

\smallskip

Systems with non-identical connections remain less understood, despite their relevance to applications in neuroscience, as noted for instance in \cite{PhPaChVi:98}. This is also supported by recent progress in experimental biology that makes detailed connection graph for large neuron networks available \cite{HuHaFrTuTaWoNoDrDaPaHeRuJa:21}.
Mathematically, non-identical connections fundamentally alters the dynamics of coupled ODEs or SDEs like \eqref{eqn:IF_multi_agent}, rendering them \emph{non-exchangeable} and making many established tools for exchangeable systems lose their applicability.

Despite these challenges, there exists a wide range of results that are able to handle systems with certain types of non-identical connections, provided some structural assumptions are made. 
A first example assumes that connections follow the algebraic constraint $\sum_{j} w_{i,j;N} \equiv 1$ and that the initial data $(X^i_0)_{i=1}^N$ are i.i.d.; the same mean-field equation as in the exchangeable case is then obtained, see for instance~\cite{InTa:15}.
Another well known case is found when the connections smoothly depend on the physical location of each neuron: A typical assumption is that $w_{i,j;N} = W(y_{i;N}, y_{j;N})$, where $y_{i;N} \in \R^d$ denotes the spatial location of the $i$-th neuron and $W(\cdot,\cdot)$ is a smooth function. This case leads to some version of the well-known neural field equations, see \cite{Be:56, Gr:63a, Gr:63b, Gr:65, WiCo:72, Am:77}.
Within this type of assumptions on connections, the mean-field limit has also been investigated in \cite{ChDuLoOs:19} for a model based on the Hawkes processes.
Another well-known setting consists in taking random connections, typically corresponding to some classical random graph. This can of course be an attractive assumption when the connections remains mostly unknown. The mean-field limit has been rigorously derived with several types of random connections including the Erd\"os-R\'enyi type, as shown in \cite{GrLeMaChDeTa:19}. We also mention \cite{CoDiGi:20, Me:19, OlReSt:20} that obtain mean-field limits of other multi-agent systems, still with random connections.


It is also enlightening to draw a comparison with the wider spectrum of results on general non-exchangeable systems and not specifically IF models. 
Many approaches rely on graphon theory, such as \cite{KaMe:18} which derives the mean-field limit for the Kuramoto model (originally introduced in  \cite{Ku:75}) while subsequent explorations of the dynamics were performed in \cite{ChMe:19a, ChMe:19b}. Graphons are natural tools to try to describe the graph limit of connections $w_{i,j;N}$ without a priori knowledge of additional regularity. Unfortunately, the use of graphon  requires a dense scaling for the connections with typically $\max_{i, j} |w_{i,j;N}| \sim O(1/N)$. 
There are still some results on sparse graph connections. We mention \cite{LaRaWu:19} based on some concept of weak convergence on graphs, or \cite{GkKu:22, KuXu:22, GkKuXu:22, GkKuXu:23} which are based on extensions of graphons such as graph-op. While those results still require a priori knowledge of some additional convergence of $w_{i,j;N}$,  
the case of sparse connections without a priori regularity  has been recently studied in \cite{JaPoSo:21}.

We keep in the present article the same general assumptions on connections as \cite{JaPoSo:21} namely,
\begin{itemize}

\item The $w_{i,j;N}$ may  be completely different for every pair of neurons.

\item The $w_{i,j;N}$ can be positive or negative with corresponding excitation or inhibition between neurons, and are not symmetric.

\item The number of neurons is assumed to be very large $N \gg 1$. We recall in particular that the human brain for example contains approximately $8.6 \times 10^{10}$ neurons.

\item The $w_{i,j;N}$ satisfy the following scaling:
\begin{equation*}
\begin{aligned}
\textstyle \max \Big( \max_{i} \sum_{j} |w_{i,j;N}| , \max_{j} \sum_{i} |w_{i,j;N}| \Big) \sim O(1),
\quad \quad
\max_{i, j} |w_{i,j;N}| \ll 1.
\end{aligned}
\end{equation*}
This scaling allows each neuron $i$ to be connected to a large population of neurons $j$, while keeping the network sparsely connected. This again seems to fit with the  average of $7000$ synaptic connections per neuron in the human brain.


\end{itemize}
However, as explained later on, we introduce several new key ideas with respect to \cite{JaPoSo:21}, which allows for a broader set of assumptions on the initial data and  also makes dealing with jump processes easier. 
\subsection{The marginal laws and BBGKY hierarchy for exchangeable systems}
A classical way to address this mean-field limit of large SDE systems like~\eqref{eqn:IF_multi_agent} is to shift our focus from tracking trajectories to examining the joint law of various subsets of neurons.

For clarity, let us first mention some of the notations that we are using. We denote by $\mathcal{M}(\R^k)$ the space of signed Borel measures with bounded \emph{total variation norm} on $\R^k$.  $\mathcal{M}_+(\R^k)$ stands for the subset of non-negative measures. $\mathcal{P}(\R^k)$ stands for the subset of probability measures. When choosing a topology on $\mathcal{M}(\R^k)$, we will mostly use the classical notion of \emph{weak-*} convergence. Note that we will also have bounds on some exponential moments, so that together with those estimates, weak-* convergence will typically imply tight convergence.

We now introduce the classical concept of marginals for exchangeable systems, where we emphasize the following steps to highlight the difference with non-exchangeable systems,
\begin{itemize}
\item For any distinct indices $i_1,\dots,i_k \in \{1,\dots,N\}$, denote the marginal law of the agents $X^{i_1;N}_t,\dots, X^{i_{k};N}_t$ by
\begin{equation*}
\begin{aligned}
f_{N}^{i_1,\dots, i_k}(t,\cdot) \defeq \;& \law (X^{i_1;N}_t,\dots, X^{i_{k};N}_t) \in \mathcal{P}(\R^k).
\end{aligned}
\end{equation*}
\item Formally define $f_{N}^{i_1,\dots, i_k} \equiv 0$ if there are duplicated indices among $i_1,\dots,i_k$.
\item For the full joint law, adopt the simplified notation that
\begin{equation*}
\begin{aligned}
f_{N}(t,\cdot) \defeq \;& f_N^{1,\dots,N}(t,\cdot) = \law (X^{1;N}_t,\dots, X^{N;N}_t) \in \mathcal{P}(\R^N).
\end{aligned}
\end{equation*}
\end{itemize}
In the context of exchangeable system (identical connections, $w_{i,j;N} = w(N)$), it is straightforward that, if $(X^{1;N}_t,\dots, X^{N;N}_t)$ is a solution of system, then any permutation $(X^{i_1;N}_t,\dots, X^{i_N;N}_t)$ solves the same system as well.
This implies that the full joint law equation is symmetric, so it suffice to consider that marginals of the same order are identical, namely,
\begin{equation*}
\begin{aligned}
f_{N}^{i_1,\dots, i_k}(t,\cdot) = f_{N}^{j_1,\dots, j_k}(t,\cdot) \in \mathcal{P}(\R^k), 
\end{aligned}
\end{equation*}
if the indices $1 \leq i_1,\dots, i_k \leq $ are distinct and $1 \leq j_1,\dots, j_k \leq N$ are also distinct.

Given this property, it is natural to define the unique $k$-marginal
\begin{equation*}
\begin{aligned}
f_{N,k}(t,\cdot) \defeq f_{N}^{1,\dots,k}(t,\cdot) \in \mathcal{P}(\R^k).
\end{aligned}
\end{equation*}
The marginals are solutions to the famous BBGKY hierarchy of equations, in which the equation for each $f_{N,k}$ depends on itself and the next marginal $f_{N,k+1}$ recursively.

One of the key concepts to obtain the mean-field limit is the notion of (Kac's) chaos, which can be defined in various equivalent ways.
One possible definition involves the marginals which is the one we use in this article: We have chaos iff the $k$-marginals of random variables $(X^{1;N},\dots, X^{N;N})$ converge weak-* to the tensorization of a certain one-particle distribution $f \in \mathcal{P}(\R)$ as $N \to \infty$, namely, 
\begin{equation*}
\begin{aligned}
f_{N,k} \overset{\ast}{\rightharpoonup} f^{\otimes k} \in \mathcal{P}(\R^k), \quad f^{\otimes k}(z_1,\dots,z_k) \defeq \prod_{m=1}^k f(z_m),
\quad \mbox{for all fixed } k \in \N.
\end{aligned}
\end{equation*}
At least for smooth enough dynamics, it is possible to show that chaos on the initial data implies chaos at every later time, which is the famous propagation of chaos. Among the various strategies for proving propagation of chaos and for obtaining the Vlasov equation as a mean-field limit, we highlight the following one given its similarities with the approach we will follow:
\begin{itemize}
\item Pass to the limit in the BBGKY hierarchy to the Vlasov hierarchy as $N \to \infty$, which yields $f_{N,k}(t,\cdot) \overset{\ast}{\rightharpoonup} f_{\infty,k}(t,\cdot) \in \mathcal{P}(\R^k)$ where $f_{\infty,k}(t,\cdot)$ represents a solution to the Vlasov hierarchy with initial data in tensorized form, namely $f_{\infty, k}(0, \cdot) = f_0^{\otimes k}$.

\item Notice that if the one-particle distribution $f(t,\cdot)$ solves the Vlasov equation with initial data $f_0$, then the $k$-marginals in tensorized form $f^{\otimes k}(t,\cdot)$ are a solution to the Vlasov hierarchy with the same initial data $f_{\infty, k}(0, \cdot) = f_0^{\otimes k}$.

\item Prove the uniqueness of the solution of the Vlasov hierarchy, which allows one to conclude that at all time $t \geq 0$, $f_{\infty,k}(t,\cdot) = f^{\otimes k}(t,\cdot)$. 
\end{itemize}
A variation of this argument involves directly obtaining stability estimates between the BBGKY hierarchy and the Vlasov hierarchy, yet quantifies the deviation of the $N$-particle SDE system to the Vlasov equation on the level of marginal laws. In general deriving the mean-field limit can be challenging, especially when the interaction between particles is singular or when there is no diffusion in the dynamics. Not surprisingly, the above approach usually requires smoothness on the dynamics: from analytic in \cite{Sp:91} to Lipschitz in \cite{GoMoRi:13}. However recent results such as \cite{La:23} and \cite{BrJaSo:22} have shown how to take advantage of non-vanishing diffusion to handle interactions through a kernel merely in respectively only $L^\infty$ (more precisely some exponential Orlicz space) and only $L^p$ for $p > 1$.

We hope to implement a similar strategy for non-exchangeable systems, such as our \eqref{eqn:IF_multi_agent}. 
However, given a solution $(X^{1;N}_t,\dots, X^{N;N}_t)$ of \eqref{eqn:IF_multi_agent}, a permutation $(X^{i_1;N}_t, \dots, X^{i_N;N}_t)$ is not in general also a solution since $w_{1,2;N}$ is in general not equal to $w_{i_1,i_2;N}$ for example. Consequently, the concept of $k$-marginals does not actually exist and, instead, we have to consider the more complicated situation where for a fixed $k$, each marginal law $f_{N}^{i_1,\dots, i_k}$ might differ.

It is, however, not even the most significant obstacle. The more intricate issue lies in the fact that any direct generalization of the BBGKY hierarchy would depend explicitly on the coefficients $w_{i,j}$. Hence, passing to the limit in the hierarchy would require passing to the limit in some appropriate sense in the coefficients $w_{i,j;N}$. If the $w_{i,j;N}$ are of order $O(1/N)$, one can potentially apply the graphon theory \cite{LoSz:06} to achieve this, as has been done for the Kuramoto model in \cite{KaMe:18}.
Unfortunately, we are considering potentially sparse networks without any a priori smoothness and we have no idea how to generalize graphon theory in that case.
\subsection{The novel notion of observables for non-exchangeable systems}
A main contribution of the paper is to introduce a novel concept of \emph{observables} in Definition~\ref{defi:observables}, which not only incorporates into the marginal laws $f_{N}^{i_1,\dots, i_k}$ but also takes into account the effect of connectivity $w_N = (w_{i,j})_{i,j=1}^N$ in \eqref{eqn:IF_multi_agent}.

Those observables satisfy an approximate hierarchy that extends in some sense the BBGKY hierarchy but which does not involve any explicit dependence on the connection weights. This new hierarchy hence offers a promising framework for obtaining the mean-field limit, as it will be enough to pass to the limit in a countable family of observables and equations.

Its structure however remains more complex. The main idea behind the definition of the new observables, is to track all possible interactions between any finite number of neurons.  In the exchangeable case, it does not matter in which order these interactions take place, so that our observables would reduce to the marginals and only depend on the total number of neurons under consideration. But in the non-exchangeable case such as here, it is necessary to keep track of which neuron is interacting with which.
To achieve this, we use tree graphs to index our observables, and establish a natural correspondence between adding a leaf on a node of the tree and interacting with a particular agent among the $k$ selected ones.

\begin{defi} \label{defi:tree_graphs} Define $\mathcal{T}$ as a set of directed labeled graphs (trees) constructed recursively in the following manner
  \begin{itemize}
    \item Denoting by $|T|$ the total number of vertices in $T$, index the vertices in $T$ from $1,\dots, |T|$.
\item The graph of a single node (indexed by $1$) belongs to $\mathcal{T}$.

\item All other elements of $\mathcal{T}$ are constructed recursively: For any $T \in \mathcal{T}$ and any $1 \leq m \leq |T|$, the graph $T+m$ belongs to $\mathcal{T}$, where $T+m$ is obtained by adding a leaf to vertex $\#m$ namely  by adding a node indexed by $|T| + 1$ and adding $(m, |T| + 1)$ as an edge to~$T$.
\end{itemize}

\end{defi}
The family $\mathcal{T}$ corresponds to all trees up to isomorphisms but it is
 equipped with a natural orientation. The root of the tree is always labeled $1$, and  $(l,m) \in \mathcal{E}(T)$ if there exists an edge connecting $l$ and $m$ and if $l$ is closer to the root than $m$. 
This family enables us to define our observables.
\begin{defi} \label{defi:observables}
Consider any connectivity matrix $w_N = (w_{i,j})_{i,j =1}^N$ and a collection of random processes $(X^{1;N}_t,\dots, X^{N;N}_t)$. We define the observable $\tau_N (T,w_N,f_N)(t,\cdot) \in \mathcal{M}(\R^{|T|})$, $T \in \mathcal{T}$ as the weighted sum of marginals
\begin{equation} \label{eqn:hierarchy}
\begin{aligned}
\tau_N (T,w_N,f_N)(t,\rd z) \defeq \;& \frac{1}{N} \sum_{i_1,\dots, i_{|T|} = 1}^N w_{N,T}(i_1,\dots, i_{|T|}) f_{N}^{i_1,\dots, i_{|T|}}(t, \rd z_1,\dots,\rd z_{|T|})
\end{aligned}
\end{equation}
where the weight of each marginal is given by
\begin{equation*}
\begin{aligned}
w_{N,T}(i_1,\dots, i_{|T|}) \defeq \prod_{(l,m) \in \mathcal{E}(T)} w_{i_l, i_m; N} \in \R.
\end{aligned}
\end{equation*}
We also define the absolute observable $|\tau_N| (T,w_N,f_N)(t,\cdot) \in \mathcal{M}_+(\R^{|T|})$, $T \in \mathcal{T}$,  as
\begin{equation*} \label{eqn:absolute_hierarchy}
\begin{aligned}
|\tau_N| (T,w_N,f_N)(t,\rd z) \defeq \frac{1}{N} \sum_{i_1,\dots, i_{|T|} = 1}^N \big| w_{N,T}(i_1,\dots, i_{|T|}) \big| f_{N}^{i_1,\dots, i_{|T|}}(t,\rd z_1,\dots,\rd z_{|T|}).
\end{aligned}
\end{equation*}

\end{defi}
As we can see, if $T_1, T_2 \in \mathcal{T}$ are isomorphic as tree graphs, the corresponding observables are also identical up to permutation. In this sense, we can say our observables are indexed by trees. It will be apparent later that the weights are chosen in a natural way so that, in the evolution of observable $T$, the observable $T+m$ accounts for the interaction with the $m$-th agent among the $|T|$ selected ones.

There does not appear to be an immediate interpretation for most observables, with the obvious exception of the first one. If we take as $T=T_1$ the first trivial tree with only vertex, then the observable is the $1$-particle distribution which is just the average of all marginals of order $1$,
\[
\tau_N (T_1,w_N,f_N)(t,\rd z_1)=\frac{1}{N} \,\sum_{i=1}^N f_N^i(t,\rd z_1).
\]
Hence obtaining the limit of the observables directly provides the limit of the $1$-particle distribution.

We also emphasize that, in contrast to the marginals, our observables are not probability measures. They are neither necessarily normalized to a total mass of $1$, nor guaranteed to be non-negative.
But the scaling of $w_N$ still ensures the total variation of any observable is at most $O(1)$,
\begin{lem}
  For any $T\in \mathcal{T}$, we have that
\[
\big\| |\tau_N| (T,w_N,f_N)(t,\cdot) \big\|_{\mathcal{M}(\R^{|T|})}\leq \big( {\textstyle \max_{i} \sum_{j=1}^N |w_{i,j;N}|} \big)^{|T| - 1}.
\]\label{lemboundtauN}
\end{lem}
\begin{proof}
Recall that any marginal law has total mass $1$ by definition, thus,
\begin{equation*}
\begin{aligned}
\big\| |\tau_N| (T,w_N,f_N)(t,\cdot) \big\|_{\mathcal{M}(\R^{|T|})} \leq \frac{1}{N} \sum_{i_1,\dots, i_{|T|} = 1}^N \big| w_{N,T}(i_1,\dots, i_{|T|}) \big|.
\end{aligned}
\end{equation*}
If $|T| = 1$, the right hand side equals to $1$ trivially, concluding the proof.

When $|T| \geq 2$, we can assume $T = T' + m$ and argue recursively
\begin{equation*}
\begin{aligned}
\;& \frac{1}{N} \sum_{i_1,\dots, i_{|T|} = 1}^N \big| w_{N,T}(i_1,\dots, i_{|T|}) \big| = \frac{1}{N} \sum_{i_1,\dots, i_{|T|} = 1}^N \big| w_{N,T'}(i_1,\dots, i_{|T|-1}) \big| |w_{i_m,i_{|T|};N}|
\\
\leq \;& \bigg( \frac{1}{N} \sum_{i_1,\dots, i_{|T|-1} = 1}^N \big| w_{N,T'}(i_1,\dots, i_{|T|-1}) \big| \bigg) {\textstyle \max_{i} \sum_{j=1}^N |w_{i,j;N}|}. 
\end{aligned}
\end{equation*}
\end{proof}

\begin{rmk}
  While in Definition~\ref{defi:observables} the laws and observables are only assumed to be measures, and hence are denoted by $f(\rd z)$, we may adopt the abuse of notation $f(z)$ in latter discussions.
Many of the forthcoming equations, such as the Vlasov equation \eqref{eqn:Vlasov_transport}, are indeed classically written on densities. 
\end{rmk}

Given the non-exchangeability of the system \eqref{eqn:IF_multi_agent}, the limiting behavior as $N \to \infty$ cannot be approximated by just a function $f(t,x)$, with $x \in \R$. Following the idea in \cite{KaMe:18} and \cite{JaPoSo:21}, we introduce the so-called extended density $f(t,\xi,x)$ instead, where the additional variable $\xi \in [0,1]$ accounts for the non-exchangeable indices $i \in \{1, \dots, N\}$ in the mean-field limit. 
The non-identical interactions in the limit is described by a kernel we denote by $w(\xi,\zeta)$, $(\xi,\zeta) \in [0,1]^2$, and the Vlasov equation corresponding to \eqref{eqn:IF_multi_agent} is given by
\begin{equation} \label{eqn:Vlasov_transport}
\begin{aligned}
\partial_t f(t,\xi,x) + \partial_x \Big(\mu^*_{f}(t,\xi,x) f(t,\xi,x) \Big) - \frac{\sigma^2}{2} \partial_{xx} f(t,\xi,x)
\\
+ \nu(x) f(t,\xi,x) - \delta_0(x) J_f(t,\xi) = 0,
\end{aligned}
\end{equation}
where the mean firing rate and the mean-field drift are defined as
\begin{equation} \label{eqn:Vlasov_mean_field}
\begin{aligned}
J_f(t,\xi) \defeq \int_{\R} \nu(x) f(t,\xi,x) \;\rd x,
\quad \quad
\mu^*_{f}(t,\xi,x) \defeq \mu(x) + \int_0^1 w(\xi, \zeta) J_f(t,\zeta) \;\rd \zeta.
\end{aligned}
\end{equation}
In our context, $w(\xi,\zeta)$ should be the limit object of the sparsely connected $w_N \defeq (w_{i,j;N})_{i,j =1}^N$ that we have described in Section~\ref{sec:introduction_non_math}. As a consequence, we are forced to consider singular kernels $w(\xi,\zeta)$ and the only property we can inherit from $w_N$ is the $O(1)$ scaling of 
\begin{equation*}
\begin{aligned}
\textstyle \max \Big( \max_{i} \sum_{j} |w_{i,j;N}| , \max_{j} \sum_{i} |w_{i,j;N}| \Big) = \max \big( \|w_N\|_{\ell^\infty \to \ell^\infty} , \|w_N\|_{\ell^1 \to \ell^1} \big).
\end{aligned}
\end{equation*}
To extend this norm for $N \times N$ connectivity matrices to the kernel on $(\xi,\zeta) \in [0,1]^2$, we define the Banach space $L^\infty_\xi([0,1], \mathcal{M}_\zeta[0,1])$ as the topological dual of the (strong) Bochner space $L^1_\xi([0,1], C_\zeta[0,1])$. 
Since $\mathcal{M}_{\xi,\zeta}([0,1]^2)$ is the topological dual of $C_{\xi,\zeta}([0,1]^2)$ and
the canonical embedding
\begin{equation*}
\begin{aligned}
C_{\xi,\zeta}([0,1]^2) \to L^1_\xi([0,1], C_\zeta[0,1])
\end{aligned}
\end{equation*}
is continuous with dense image, one can consider
\begin{equation*}
\begin{aligned}
L^\infty_\xi([0,1], \mathcal{M}_\zeta[0,1]) \subset \mathcal{M}_{\xi,\zeta}([0,1]^2).
\end{aligned}
\end{equation*}
This leads to the main Banach space $\mathcal{W}$ for the kernels 
\begin{equation*}
\begin{aligned}
\mathcal{W} \defeq \{w \in \mathcal{M}([0,1]^2) : w(\xi, \rd \zeta) \in L^\infty_\xi([0,1], \mathcal{M}_\zeta[0,1]), \;  w(\rd \xi, \zeta) \in L^\infty_\zeta([0,1], \mathcal{M}_\xi[0,1]) \}.
\end{aligned}
\end{equation*}
We note that we deal later in the article with a priori estimate of $f(t,\xi,x)$ and we use for those the usual strong Bochner spaces $L^\infty([0,t_*] \times [0,1]; \mathcal{M}(\R))$.

The proper definition of the kernel space  $\mathcal{W}$ allows us to correctly define the conjectured limiting observables from the extended density.
\begin{defi} \label{defi:limiting_observables}

Consider a connectivity kernel $w(\xi,\zeta)\in \mathcal{W}$, $(\xi,\zeta) \in [0,1]^2$ and some extended density $f \in L^\infty([0,t_*] \times [0,1]; \mathcal{M}_+(\R))$. Define the observables $\tau_\infty (T,w,f)(t,\cdot) \in \mathcal{M}(\R^{|T|})$, $T \in \mathcal{T}$, as

\begin{equation} \label{eqn:limiting_observables}
\begin{aligned}
\tau_\infty(T,w,f)(t,z) \defeq \int_{[0,1]^{|T|}} w_T(\xi_1,\dots,\xi_{|T|}) {\textstyle \prod_{m=1}^{|T|} f(t,\xi_m,z_m)} \;\rd \xi_1,\dots, \rd \xi_{|T|},
\end{aligned}
\end{equation}
where
\begin{equation*}
\begin{aligned}
w_{T}(\xi_1,\dots, \xi_{|T|}) \defeq \prod_{(l,m) \in \mathcal{E}(T)} w(\xi_l, \xi_m).
\end{aligned}
\end{equation*}

\end{defi}

It is easy to check the validity of integrals in \eqref{eqn:Vlasov_mean_field} and \eqref{eqn:limiting_observables} if the kernel $w(\xi,\zeta)$ is smooth or  when $w \in L^\infty$. 
%
At the present, it may not be clear yet why the integrations with respect to $\xi \in [0,1]$ involved in \eqref{eqn:Vlasov_mean_field} and \eqref{eqn:limiting_observables} make sense when we only have $w \in \mathcal{W}$.
We prove in Section~\ref{sec:observables} that it is possible to extend the bounds in Lemma~\ref{lemboundtauN} through a density argument.
We note that a definition akin to $\tau_\infty$ along with a similar argument on integrability has been addressed in \cite{JaPoSo:21}.


\subsection{Main result} \label{subsec:main_result}
Our main result states that the large scale dynamics of \eqref{eqn:IF_multi_agent} described in terms of observables $\tau_N (T,w_N,f_N)$ can be indeed approximated by the mean-field limit, provided the initial observables $\tau_N(t=0)$ are approximated by the initial~$\tau_\infty(t=0)$.
\begin{thm} \label{thm:stable_qualitative}

Assume that $\mu, \nu \in W^{1,\infty}$ and $\sigma > 0$. For a sequence of $N \to \infty$,
let $(X^{i;N}_t)_{i=1}^N$ be solutions of the non-exchangeable SDE system \eqref{eqn:IF_multi_agent} with connectivity matrices $w_N \defeq (w_{i,j;N})_{i,j =1}^N$.
In addition, let $f \in L^\infty([0,t_*] \times [0,1]; \mathcal{M}_+(\R))$ be a solution of the Vlasov equation \eqref{eqn:Vlasov_transport}-\eqref{eqn:Vlasov_mean_field} with connectivity kernel $w \in \mathcal{W}$.
Assume that the following holds:
\begin{itemize}
\item The connectivity matrices are uniformly bounded: For some $C_{\mathcal{W}} > 0$,
\begin{equation} \label{eqn:assumption_1}
\begin{aligned}
\sup_N \textstyle \; \max \Big( \max_{i} \sum_{j} |w_{i,j;N}| , \max_{j} \sum_{i} |w_{i,j;N}| \Big) \leq C_{\mathcal{W}}.
\end{aligned} 
\end{equation}
\item The interaction of each pair of agents vanishes: 
\begin{equation} \label{eqn:assumption_2}
\begin{aligned}
\max_{1\leq i,j \leq N} |w_{i,j;N}| \to 0 \;\text{ as }\; N \to \infty.
\end{aligned} 
\end{equation}
\item The hierarchy of observables and the extended density are initially bounded by an exponential scale: There exists some $a > 0$, $M_a > 0$, such that,
\begin{equation} \label{eqn:assumption_3}
\begin{aligned}
\sup_N \; \int_{\R^{|T|}} { \textstyle \exp\big( a \sum_{m=1}^{|T|} |z_m|\big)} |\tau_N|(T,w_N,f_N)(0,z) \;\rd z \leq \;& M_a^{|T|}, \quad \forall T \in \mathcal{T},
\\
\esssup_{\xi \in [0,1]} \int_{\R} { \textstyle \exp\big( a |x|\big)} f(0, \xi, x) \;\rd x \leq \;& M_a.
\end{aligned} 
\end{equation}
\item The hierarchy of observables initially converges in weak-* topology:
\begin{equation} \label{eqn:assumption_4}
\begin{aligned}
\tau_N(T,w_N,f_N)(0,\cdot) \overset{\ast}{\rightharpoonup} \tau_\infty(T,w,f)(0,\cdot) \in \mathcal{M}(\R^{|T|}) \;\text{ as }\; N \to \infty, \quad \forall T \in \mathcal{T}
\end{aligned} 
\end{equation}
\end{itemize}
Then, the hierarchy of observables converges at any time, in weak-* topology:
\begin{equation} \label{eqn:main_convergence}
\begin{aligned}
\tau_N(T,w_N,f_N)(t,\cdot) \overset{\ast}{\rightharpoonup} \tau_\infty(T,w,f)(t,\cdot) \in \mathcal{M}(\R^{|T|}) \;\text{ as }\; N \to \infty, \quad \forall t \in [0,t_*], \; T \in \mathcal{T}.
\end{aligned} 
\end{equation}
\end{thm}
While we state Theorem~\ref{thm:stable_qualitative} in terms of the observables $\tau_N$ from non-exchangeable systems converging to the limiting observables $\tau_\infty$ in weak-* topology, our approach is inherently quantitative. We state, in the next section, a  precise and quantitative version of Theorem~\ref{thm:stable_qualitative}, namely Theorem~\ref{thm:stable_power_norm}.

We recall that the first observable immediately correspond to the $1$-particle distribution so that Theorem~\ref{thm:stable_qualitative} provides the limit of this $1$-particle distribution. It would in fact be possible to derive the limit of other well-known statistical objects, the $2$-particle distribution and correlations for example. To do that, we would build another family of new observables starting from the $2$-particle distribution in addition  to the $1$-particle distribution. This would also require stronger assumptions with the initial convergence on both families instead of only \eqref{eqn:assumption_4}. However we did not want to further add to our approach or our statements and confine ourselves to the limit of the $1$-particle distribution.

%
%

The only non-straightforward assumption in Theorem~\ref{thm:stable_qualitative} is \eqref{eqn:assumption_4} about  whether the $\tau_\infty(T)(0,\cdot)$, $\forall T \in \mathcal{T}$ come from a pair of extended density $f(0,x,\xi)$ and $w \in \mathcal{W}$ as defined in Definition~\ref{defi:limiting_observables}. It would be possible to formulate a version of Theorem~\ref{thm:stable_qualitative} without this assumption. The sequence of initial data $\tau_N(T,w_N,f_N)(0,\cdot)$ is obviously precompact as $N \to \infty$, so that we could extract a converging sub-sequence. The proof of Theorem~\ref{thm:stable_qualitative} would then imply that the limiting $\tau_\infty$ are exact solutions to a limiting, tree-indexed hierarchy. However, without \eqref{eqn:assumption_4}, we cannot identify the limiting $\tau_\infty$ as being obtained through some solution $f(t,x,\xi)$ to the limiting Vlasov equation.

It is fortunately straightforward to show that \eqref{eqn:assumption_4} directly follows when the initial $X^{i,N}_0=X^{i,N}(t=0)$ are independent.
When the initial data $(X^{1;N}_0,\dots, X^{N;N}_0)$ are independent random variables with $f_{N,0}^i = \law (X^{i;N}_0)$ for all $1 \leq i \leq N$, the marginal laws are of form $f_{N,0}^{i_1,\dots, i_k} = \prod_{m=1}^k f_{N,0}^{i_m}$ for $1 \leq i_1,\dots, i_k \leq N$ that are distinct. 
We can then define a graphon-like kernel and the extended density as
\begin{equation} \label{eqn:extended_density_kernel}
\begin{aligned}
\tilde w_N(\xi,\zeta) = \;& \sum_{i,j=1}^N N w_{i,j;N} \mathbbm{1}_{[\frac{i-1}{N},\frac{i}{N})}(\xi) \mathbbm{1}_{[\frac{j-1}{N},\frac{j}{N})}(\zeta),
\\
\tilde f_N(x,\xi) = \;& \sum_{i=1}^N f_{N}^{i}(x) \mathbbm{1}_{[\frac{i-1}{N},\frac{i}{N})}(\xi).
\end{aligned}
\end{equation}
It becomes straightforward to show that the initial observables $\tau_N (T,w_N,f_N,t=0)$ are approximated by $\tau_\infty(T,\tilde w_N, \tilde f_N,t=0)$ up to an error of $O(\max_{1\leq i,j \leq N} |w_{i,j;N}|)$. We in particular state the following proposition, whose proof is postponed to Section~\ref{sec:observables}.
\begin{prop} \label{prop:independent_compactness}

For a sequence of $N \to \infty$, consider $(X^{1;N},\dots, X^{N;N})$ as independent random variables and $w_N = (w_{i,j})_{i,j =1}^N \in \R^{N \times N}$.
Denote the marginal laws as $f_{N}^i = \law (X^{i;N})$ for each $N$ and $1 \leq i \leq N$. Further, let $\tilde w_N$, $\tilde f_N$ be the kernel and extended density as defined in \eqref{eqn:extended_density_kernel}.
Assume that the following holds:
\begin{itemize}
\item The connectivity matrices are uniformly bounded: For some $C_{\mathcal{W}} > 0$,
\begin{equation} \label{eqn:assumption_1_compact}
\begin{aligned}
\sup_N \; \textstyle \max \Big( \max_{i} \sum_{j} |w_{i,j;N}| , \max_{j} \sum_{i} |w_{i,j;N}| \Big) \leq C_{\mathcal{W}}.
\end{aligned} 
\end{equation}
\item The interaction of each pair of agents vanishes: 
\begin{equation} \label{eqn:assumption_2_compact}
\begin{aligned}
\bar w_N \defeq \max_{1\leq i,j \leq N} |w_{i,j;N}| \to 0, \;\text{ as }\; N \to \infty.
\end{aligned} 
\end{equation}
\item The laws are bounded by an exponential scale: There exists some $a > 0$, $M_a > 0$, such that,
\begin{equation} \label{eqn:assumption_3_compact}
\begin{aligned}
\sup_N \; \max_{1 \leq i \leq N} \int_{\R} {\textstyle \exp(a |z|) f_{N}^{i}(z)} \;\rd z
\leq \;& M_a.
\end{aligned} 
\end{equation}
\end{itemize}
Then the difference between observables $\tau_N (T,w_N,f_N)$ and their approximations $\tau_\infty(T,\tilde w_N, \tilde f_N)$, as formulated by \eqref{eqn:limiting_observables} and \eqref{eqn:extended_density_kernel}, is quantified by
\begin{equation} \label{eqn:approximate_independence}
\begin{aligned}
\;& \int_{\R^{|T|}} { \textstyle \exp\big( a \sum_{m=1}^{|T|} |z_m|\big)} |\tau_\infty(T,\tilde w_N, \tilde f_N)(z) - \tau_N(T,w_N, f_N)(z)| \;\rd z
\\
\leq \;& \max_{1\leq i,j \leq N} |w_{i,j;N}| \max \Big( \textstyle \max_{i} \sum_{j} |w_{i,j;N}| , \max_{j} \sum_{i} |w_{i,j;N}| \Big)^{|T|-2} |T|^2 M_a^{|T|}.
\end{aligned} 
\end{equation}
Moreover, by extracting a subsequence (which we still index by $N$ for simplicity), there exists a pair of kernel $w \in \mathcal{W}$ and extended density $f \in L^\infty([0,1]; \mathcal{M}_+(\R))$, such that the hierarchy of approximate observables $\tau_\infty(T,\tilde w_N, \tilde f_N)$ converges weak-* to the limit hierarchy $\tau_\infty(T,w,f)$:
\begin{equation} \label{eqn:assumption_4_compact}
\begin{aligned}
\tau_\infty(T,\tilde w_N, \tilde f_N) \overset{\ast}{\rightharpoonup} \tau_\infty(T,w,f) \in \mathcal{M}(\R^{|T|}) \;\text{ as }\; N \to \infty, \quad \forall T \in \mathcal{T}.
\end{aligned} 
\end{equation}
In addition, such extended density $f$ satisfies the bound
\begin{equation*}
\begin{aligned}
\esssup_{\xi \in [0,1]} \int_{\R} { \textstyle \exp\big( a |x|\big)} f (\xi, x) \;\rd x \leq \;& M_a.
\end{aligned}
\end{equation*}

\end{prop}
When combined with Proposition~\ref{prop:independent_compactness}, Theorem~\ref{thm:stable_qualitative} yields the mean-field limit for independent initial $X_0^{i,N}$ with only some appropriate moments bounds and no other structural assumptions on the $w_{i,j,N}$. However we do emphasize that for non-exchangeable systems, the convergence of observables can in general be much less demanding than independence. It is a very different situation from exchangeable systems where chaos (or approximate independence) is essentially equivalent to the asymptotic tensorization of the marginal.

But for our present models, counterexamples are easy to construct. We can for instance separate the index $i=1\dots N$ into two distinct subset $I_1$ and $I_2$. We then take $w_{i,j,N}=0$ if $i\in I_1$ and $j\in I_2$ or $j\in I_1$ and $i\in I_2$. In that case there are no interactions between neurons in $I_1$ and neurons in $I_2$.  We can then easily satisfy Assumption~\eqref{eqn:assumption_4} by having the $X_0^{i,N}$ independent within each subset $I_1$ and $I_2$ but with as much correlation as desired between the subsets. This example can obviously be generalized to any arbitrary fixed number of subsets and it is possible to construct even more intricate examples.  But this already shows that the optimal assumptions on the initial $X_0^{i,N}$ have to depend intrinsically on the structure of the connections in non-exchangeable cases. In that regard, we conjecture that Assumption~\eqref{eqn:assumption_4} is both necessary and sufficient to have the convergence of the $1$-particle distribution.

\smallskip

Theorem~\ref{thm:stable_qualitative} is the first rigorous result to obtain the mean-field limit for networks of neurons interacting through integrate and fire models. The approach through an extended hierarchy solved by observables has very few comparisons in the literature, having only been used previously in~\cite{JaPoSo:21}. In comparison with the previous~\cite{JaPoSo:21} however, we put forward several new key ideas with notably
\begin{itemize}
\item We introduce the observables directly at the level of the marginals. Instead the notion of observables in~\cite{JaPoSo:21} was only valid for almost independent variables, which required first the propagation of independence. There are hence several advantages to our new definition, first as per the discussion above about independence but also by providing a much immediate notion of the statistical distribution in the system.
\item We develop a new approach for the quantitative estimates on the hierarchy, based on weak norms. This is again in contrast to~\cite{JaPoSo:21} which was using strong $L^2$ norms. This is a critical point because the jumps in integrate and fire models lead to discontinuities so that we cannot have convergence in the hierarchy for our system for any strong norm. On the other hand, the use of weak norms forces a different method in the analysis as propagating weak norms necessarily creates intricate commutator estimates. An important technical contribution of the present paper is to introduce the ``right'' weak norms and a novel approach to handle those commutators.  
  \end{itemize}
There are however many remaining open questions. First of all, the statistical approach followed here does not seem to allow to obtain the limit of any individual trajectory. This is again in contrast with classical exchangeable systems where obtaining the limit of the 1-particle distribution allows to have the limit of typical (in some sense) trajectories. Another important question is whether it is possible to connect the additional variable $\xi$ to some properties of individual neurons, which could lead to classifying neurons in terms of their role in the dynamics. We mention as final example of open problem, the issue of including learning in the models.  
In the setting of \eqref{eqn:IF_multi_agent}, learning can be simply incorporated in the model by considering time-dependent synaptic weights $w_{i,j;N}(t)$ together with some equation prescribing the evolution of those weights. This has been recognized to be a critical mechanism as early as the famous Hebb rule in~\cite{He:49}. But it is unclear how to model this kind of learning appropriately while keeping sparse connections and a mean-field scaling, or whether the present approach would remain valid for such models.  The mean-field limit has been derived \cite{PeSaWa:17, ToSa:20} for neuron networks incorporating learning mechanisms, and also in \cite{AyDu:21} for an opinion dynamics model. But those results impose the strong algebraic constraint that $w_{i_1,j;N} = w_{i_2,j;N}$, $\forall i_1,i_2 \neq j$.

The rest of the paper is structured as follows. In Section~\ref{sec:PDE}, we present our approach of directly obtaining stability estimates, starting from the extended BBGKY hierarchy from non-exchangeable system \eqref{eqn:IF_multi_agent}, the corresponding Vlasov hierarchy, and their a priori estimates. The main stability result, as a quantitative version of \eqref{eqn:main_convergence}, is stated as Theorem~\ref{thm:stable_power_norm}.

The subsequent sections are about rigorously proving the results in Section~\ref{sec:PDE}.
We  discuss in Section~\ref{sec:negative_Sobolev} the properties of the weak norms denoted as $H^{-1\otimes k}_\eta$ that we use throughout the quantitative estimates.
In Section~\ref{sec:observables}, we revisit the limiting observables $\tau_\infty(T,w,f)$, $T \in \mathcal{T}$, to show that they are well-defined.
Finally, with the preliminaries done in Section~\ref{sec:negative_Sobolev} and \ref{sec:observables}, Section~\ref{sec:technical_proofs} is devoted to the proofs of the main results of Section~\ref{sec:PDE}, including Theorem~\ref{thm:stable_power_norm}.
%
\section{Quantitative stability estimates} \label{sec:PDE}
\subsection{A tensorized negative Sobolev norm} \label{subsec:tensorized_norm}
This subsection is dedicated to the introduction of $H^{-1\otimes k}_\eta$-norm along with its basic properties. While it is straightforward, the specific choice of this norm plays a key role in our later estimates as it leads to good commutator estimates. Introducing the mollification kernel
\begin{equation*} 
\begin{aligned}
K(x) \defeq \frac{1}{\pi} \int_0^\infty \exp(-|x| \cosh(\xi)) \rd \xi,
\end{aligned}
\end{equation*}
we may define the $H^{-1\otimes k}_\eta$-norm as follows.
\begin{defi} \label{defi:weak_Hilbert}
For any function $F$ defined on $\R$, denote its tensorization to $\R^k$ by
\begin{equation*}
\begin{aligned}
F^{\otimes k}(z_1,\dots,z_{k}) \defeq \prod_{m = 1}^{k} F (z_m), \quad \forall (z_1,\dots,z_k) \in \R^k.
\end{aligned}
\end{equation*}
We then define 
\begin{equation*}
\begin{aligned}
\|g\|_{H^{-1\otimes k}} \defeq \| K^{\otimes k} \star g \|_{L^2(\R^{k})},
\end{aligned}
\end{equation*}
and for any weight function $\eta$ on $\R$, 
\begin{equation*}
\begin{aligned}
\|g\|_{H^{-1\otimes k}_\eta} \defeq \| K^{\otimes k} \star ( g \eta^{\otimes k} ) \|_{L^2(\R^{k})}.
\end{aligned}
\end{equation*}
\end{defi}
The introduction of the weight $\eta$ is motivated by the need for some control on the decay of the solutions at infinity since we work on the whole $\R$. We simply choose some $\alpha>0$ and define  
\begin{equation*} 
\begin{aligned}
\eta(x) = \eta_\alpha(x) \defeq C_\alpha \exp \Big( \sqrt{1 + \alpha^2 x^2} \Big), \quad C_\alpha = \int_\R \exp \Big(- \sqrt{1 + \alpha^2 x^2} \Big) \;\rd x.
\end{aligned}
\end{equation*}
Our definition of $H^{-1\otimes k}_\eta$ leads to a topology that is equivalent to the classical weak-* topology of $\mathcal{M}(\R^k)$.
\begin{lem} \label{lem:topology_negative_Sobolev}
Consider any $a > 0$, $C > 0$, $0 < \alpha < a$ (which determines $\eta = \eta_\alpha$) and any sequence
\begin{equation*}
\begin{aligned}
\{g_n\}_{n = 1}^\infty \subset \bigg\{g \in \mathcal{M}(\R^k) : \int_{\R^k} { \textstyle \exp\big( a \sum_{m=1}^{k} |z_m|\big)} |g|(z) \;\rd z \leq  C\bigg\}.
\end{aligned}
\end{equation*}
Then the following are equivalent
\begin{itemize}
\item $g_n \overset{\ast}{\rightharpoonup} g_\infty$ under the weak-* topology of $\mathcal{M}(\R^k)$. 
\item $\| g_n - g_\infty \|_{H^{-1\otimes k}_\eta} \to 0$.
\end{itemize}
\end{lem}
The proof of Lemma~\ref{lem:topology_negative_Sobolev} is postponed to Section~\ref{sec:negative_Sobolev}, where we also conduct a deeper examination of the relationship between the $H^{-1\otimes k}_\eta$ norm  and classical negative Sobolev norms. The use of weak distances such as Wasserstein distances is classical in the derivation of the mean-field limit, in particular when looking at the notion of empirical measures.

However our observables are bounded functions at any $t > 0$, for which we can even prove bounds, and a main motivation for the use of weak norms stems from the singularity introduced by the Poisson jump processes. The usefulness of negative-Sobolev norms in that context has been highlighted in works such as \cite{RiThWa:12}. We also mention~\cite{FlPrZa:19} which considers a somewhat relaxed IF model with connections depending on the spatial structure of neurons. However, instead of studying the 1-particle distribution, we use tensorized $H^{-1\otimes k}_\eta$-norms to investigate the joint law $f_{N}^{i_1,\dots, i_k}$ and the observables, which seems to be a novel approach in this context.
%
\subsection{From the original SDE system to the extended BBGKY hierarchy} \label{subsec:SDE_to_BBGKY}
%
We show in this subsection that the observables, as defined in Definition~\ref{defi:observables}, satisfy an extended BBGKY hierarchy.

We first recall the Liouville or forward Kolmogorov equation that is satisfied by the full joint law $f_N$ of solutions to the SDE \eqref{eqn:IF_multi_agent},
\begin{equation} \label{eqn:IF_Liouville_PDE}
\begin{aligned}
\partial_t f_N(t,x) + \;& \sum_{i = 1}^{N} \bigg[ \partial_{x_i}(\mu(x_i) f_N(t,x)) - \frac{\sigma^2}{2} \partial_{x_i}^2 f_N(t,x)
\\
\;& + \nu(x_i) f_N(t,x) - \delta_0(x_i) \bigg( \int_{\R} \nu(y_i) f_N(t,y - (w_N)_{\cdot, i}^{\top}) \;\rd y_i \bigg)\bigg|_{\forall j \neq i,\, y_j = x_j} \bigg] = 0,
\\
(w_N)_{\cdot, i}^{\top} = \;& \big( w_{1, i; N}, \dots w_{N, i; N} \big) \in \R^{N}, \quad \forall 1 \leq i \leq N,
\end{aligned}
\end{equation}
where $\delta_0$ is the Dirac delta function at origin. The ``spike vector'' $(w_N)_{\cdot, i}^{\top}$ corresponds to the $i$-th column of connectivity matrix $w_N$ that account for the jumps when the $i$-th neuron fires.

From the Kolmogorov equation, we may derive equations on each observable.
\begin{prop} \label{prop:hierarchy_of_equations}
Assume that $\mu,\nu \in W^{1,\infty}$ and $\sigma > 0$. 
Let $w_N \defeq (w_{i,j;N})_{i,j =1}^N$ be the connectivity matrix and $(X^{1;N}_0,\dots, X^{N;N}_0)$ be the initial data with $g_{N} = \law (X^{1;N}_0,\dots, X^{N;N}_0)$.

Then, there exists a unique solution $(X^{1;N}_t,\dots, X^{N;N}_t)$ solving SDE~\eqref{eqn:IF_multi_agent} for all $t \geq 0$, whose law
\begin{equation*}
\begin{aligned}
f_N(t,\cdot) = \law(X^{1;N}_t,\dots, X^{N;N}_t) 
\end{aligned}
\end{equation*}
is the unique distributional solution of Liouville equation \eqref{eqn:IF_Liouville_PDE} with initial data $g_{N}$.
In addition, the observables
\begin{equation*}
\begin{aligned}
\tau_N(T) = \tau_N(T,w_N,f_N), \quad \forall T \in \mathcal{T}
\end{aligned}
\end{equation*}
 solve the extended version of BBGKY hierarchy with remainder terms: For all $T \in \mathcal{T}$,
\begin{equation} \label{eqn:hierarchy_equation}
\begin{aligned}
& \partial_t \tau_N (T)(t,z)
\\
&\ =  \sum_{m = 1}^{|T|} \Bigg\{ \bigg[ - \partial_{z_m}(\mu(z_m) \tau_N (T)(t,z)) + \frac{\sigma^2}{2} \partial_{z_m}^2 \tau_N (T)(t,z)
\\
&\quad - \nu(z_m) \tau_N (T)(t,z) + \delta_0(z_m) \bigg( \int_{\R} \nu(u_m) \Big( \tau_N (T)(t,u) + \mathscr{R}_{N,T,m} (t,u) \Big) \;\rd u_m \bigg)\bigg|_{\forall n \neq m,\, u_n = z_n} \bigg]
\\
&\quad - \partial_{z_m} \bigg[ \int_{\R} \nu(z_{|T|+1}) \Big( \tau_N (T+m)(t,z) + \mathscr{\tilde R}_{N,T+m,|T|+1} (t,z) \Big) \;\rd z_{|T|+1} \bigg] \Bigg\},
\end{aligned}
\end{equation}
where the remainder terms are given by
\begin{equation} \label{eqn:hierarchy_equation_remainder}
\begin{aligned}
\mathscr{R}_{N,T,m} (t,z) \defeq \;& \frac{1}{N} \sum_{i_1,\dots, i_{|T|} = 1}^N w_{N,T}(i_1,\dots, i_{|T|}) \Big( f_{N}^{i_1,\dots, i_{|T|}}(t,z - w_{N;i_m}^{i_1,\dots, i_{|T|}}) - f_{N}^{i_1,\dots, i_{|T|}}(t,z) \Big),
\\
\mathscr{\tilde R}_{N,T,m} (t,z) \defeq \;& \int_0^1 \frac{1}{N} \sum_{i_1,\dots, i_{|T|} = 1}^N w_{N,T}(i_1,\dots, i_{|T|}) \Big( f_{N}^{i_1,\dots, i_{|T|}}(t,z - r w_{N;i_m}^{i_1,\dots, i_{|T|}}) - f_{N}^{i_1,\dots, i_{|T|}}(t,z) \Big) \;\rd r,
\end{aligned}
\end{equation}
and the $w_{N;j}^{i_1,\dots, i_{k}}$ are defined as the restriction of the ``spike vector'' $(w_N)_{\cdot, j}^{\top}$ to the marginal space, namely
\begin{equation*}
\begin{aligned}
w_{N;j}^{i_1,\dots, i_{k}} \defeq \;& \big( w_{i_n, j; N} \big)_{n=1}^{k} = \big( w_{i_1, j; N}, \dots w_{i_{k}, j; N} \big) \in \R^{k}.
\end{aligned}
\end{equation*}
\end{prop}
The proof of the proposition will be done in Section~\ref{sec:hierarchy}.
Unlike the standard BBGKY hierarchy that usually gives a closed equation involving $f_{N,k}$ and the next marginal $f_{N,k+1}$, the hierarchy of equations derived here is only approximate as the remainder terms do not only depend on our observables. 
Thus, an essential part of our approach is to prove that 
as the strength of pairwise interaction $\max_{1\leq i,j \leq N} |w_{i,j;N}|$ goes to $0$ (which is assumption \eqref{eqn:assumption_2} in Theorem~\ref{thm:stable_qualitative}), 
those remainder terms $\mathscr{R}$ and $\mathscr{\tilde R}$ vanish in the $H^{-1 \otimes k}_\eta$ sense. As we mentioned earlier, it is a main motivation of choosing $H^{-1 \otimes k}_\eta$ as its specific form.
This result is precisely formulated in Proposition~\ref{prop:translation_estimate_weighted} in the next subsection.

We also note that the presence of the remainder terms $\mathscr{R}$ and $\mathscr{\tilde R}$ is not only a consequence of the Poisson jump process. Consider the more classical first-order dynamics
\begin{equation*}
\begin{aligned}
X^{i;N}_t = X^{i;N}_0 &+ \int_0^t \mu(X^{i;N}_{s}) \;\rd s + \int_0^t \sigma(X^{i;N}_{s}) \;\rd \bm{B}^i_s
\\
&+ \sum_{j \neq i} w_{i,j;N} \int_0^t \nu(X^{i;N}_{s}, X^{j;N}_{s}) \; \rd s.
\end{aligned}
\end{equation*}
Depending on the specific form of $\nu(\cdot,\cdot)$, the term $\mathscr{\tilde R}_{N,T+m,|T|+1}$ may vanish, but the term $\mathscr{R}_{N,T,m}$ is always present.
More than the specific form of the dynamics, the remainders reflect the more essential difficulty that interaction between the first $k$ neurons $i_1,\dots, i_k$ can not be fully described by the observables as defined in Definition~\ref{defi:observables}. 

This is also one of the crucial distinctions that separates the method in this article from~\cite{JaPoSo:21}. The observables in \cite{JaPoSo:21} are similar to the limiting observables $\tau_\infty$ in this article, but are constructed from the solutions of the Mckean-Vlasov SDE where the interaction felt by one agent $X^{i;N}$ is determined not by the exact $X^{j;N}$, but the $\law(X^{j;N})$.
This leads to a simplified hierarchy without remainders. 
On the other hand, in this article all the observables are constructed directly from the solution of \eqref{eqn:IF_multi_agent}, hence the extended, approximate BBGKY hierarchy \eqref{eqn:hierarchy_equation} reflects the dynamics of the original non-exchangeable system.

We conclude the subsection with a priori estimates of the absolute observables $|\tau_N|$ 
whose proof is also postponed to Section~\ref{sec:hierarchy}.
\begin{prop} \label{prop:Hilbert_a_priori_bound}
Let $N \geq 1$, $t_* > 0$ and $\alpha > 0$ (which determines $\eta = \eta_\alpha$). 
Assume that the connectivity matrix $w_N \defeq (w_{i,j;N})_{i,j =1}^N$ and joint law $f_N \in L^\infty([0,t_*]; \mathcal{M}_+(\R^{N}))$ solves the Kolmogorov equation~\eqref{eqn:IF_Liouville_PDE} in the sense of distributions. For any $T \in \mathcal{T}$, assume that at $t = 0$,
\begin{equation*} \label{eqn:assumption_3_alter}
\begin{aligned}
\||\tau_N|(T)(0,\cdot) \eta^{\otimes |T|}\|_{\mathcal{M}(\R^{|T|})} \leq C_\eta(T) < \infty.
\end{aligned} 
\end{equation*}
Then there exists $A_\eta > 0$ only depending on $\alpha$, $\|\mu\|_{W^{1,\infty}}$, $\|\nu\|_{W^{1,\infty}}$, $\sigma$ and 
\begin{equation*}
\begin{aligned}
\textstyle \max\left(\max_{i} \sum_{j} |w_{i,j;N}|,\ \max_{j} \sum_{i} |w_{i,j;N}|\right),
\end{aligned} 
\end{equation*}
such that,
\begin{equation*}
\begin{aligned}
\||\tau_N|(T)(t,\cdot) \eta^{\otimes |T|}\|_{\mathcal{M}(\R^{|T|})} \leq C_\eta(T) \big( \exp(A_\eta t_*) \big)^{|T|}, \quad \forall t \in [0,t_*],
\end{aligned}
\end{equation*}
and
\begin{equation} \label{eqn:Hilbert_energy_bound_scale}
\begin{aligned}
\| |\tau_N| (T) (t,\cdot) \|_{H^{-1\otimes |T|}_\eta} \leq C_\eta(T) \big(\|K\|_{L^2(\R)} \exp(A_\eta t_*) \big)^{|T|}, \quad \forall t \in [0,t_*].
\end{aligned} 
\end{equation}
\end{prop}
Let us emphasize again that this proposition is about $|\tau_N|$ the absolute observables, which are non-negative measures obtained by linear combinations of laws $f_{N}^{i_1,\dots, i_{k}}$, $1 \leq i_1,\dots, i_{k} \leq N$. 
We do not expect a straightforward extension to the $\tau_N$ as the potential cancellations of positive and negative terms in the dynamics makes the problem much less tractable.
\subsection{From the limiting Vlasov equation to the limiting hierarchy} \label{subsec:Vlasov_to_Vlasov}
The following proposition states that the limiting observables $\tau_\infty$ defined from the limiting Vlasov equation \eqref{eqn:Vlasov_transport}-\eqref{eqn:Vlasov_mean_field} satisfy the limiting hierarchy \eqref{eqn:limit_hierarchy_equation}, which is similar to the BBGKY hierarchy \eqref{eqn:hierarchy_equation} in Proposition~\ref{prop:hierarchy_of_equations} but without the remainder terms $\mathscr{R}$ and $\mathscr{\tilde R}$. In that sense the limiting hierarchy provides closed recursive relations of the family $\tau_\infty(T)$, $\forall T \in \mathcal{T}$. In particular the quantitative estimates proved later would imply the uniqueness of solutions to the hierarchy for a given choice of initial data.
\begin{prop} \label{prop:hierarchy_of_equations_limit}
Assume that $\mu,\nu \in W^{1,\infty}$ and $\sigma > 0$. 
Then for any $t_* > 0$, $\alpha > 0$ (which determines $\eta = \eta_\alpha$), any connectivity kernel $w  \in \mathcal{W}$ and any initial extended density $g \in L^\infty([0,1]; H^{-1}_\eta \cap \mathcal{M}_+(\R))$, there exists a unique
\begin{equation*}
\begin{aligned}
f \in L^\infty([0,t_*] \times [0,1]; H^{-1}_\eta \cap \mathcal{M}_+(\R))
\end{aligned}
\end{equation*}
solving Vlasov equation \eqref{eqn:Vlasov_transport}-\eqref{eqn:Vlasov_mean_field} in the sense of distributions.
Furthermore, the observables $\tau_\infty(T) = \tau_\infty(T,w,f)$, $\forall T \in \mathcal{T}$ are bounded by
\begin{equation} \label{eqn:limit_observable_bound}
\begin{aligned}
\big\| \tau_\infty(T,w,f)(t,\cdot) \big\|_{H^{-1 \otimes |T|}_\eta} \leq &\; \|w\|_{\mathcal{W}}^{|T| - 1} \|f\|_{L^\infty_{t,\xi} (H^{-1}_\eta)_x}^{|T|}, \quad \forall t \in [0,t_*], \; T \in \mathcal{T},
\end{aligned} 
\end{equation}
and solve the following non-exchangeable extended version of the Vlasov hierarchy: For all $T \in \mathcal{T}$,
\begin{equation} \label{eqn:limit_hierarchy_equation}
\begin{aligned}
& \partial_t \tau_\infty (T)(t,z)
\\
&\quad = \sum_{m = 1}^{|T|} \Bigg\{ \bigg[ - \partial_{z_m}(\mu(z_m) \tau_\infty (T)(t,z)) + \frac{\sigma^2}{2} \partial_{z_m}^2 \tau_\infty (T)(t,z)
\\
&\quad - \nu(z_m) \tau_\infty (T)(t,z) + \delta_0(z_m) \bigg( \int_{\R} \nu(u_m) \tau_\infty (T)(t,u) \;\rd u_m \bigg)\bigg|_{\forall n \neq m,\, u_n = z_n} \bigg]
\\
&\quad - \partial_{z_m} \bigg[ \int_{\R} \nu(z_{|T|+1}) \tau_\infty (T+m)(t,z) \;\rd z_{|T|+1} \bigg] \Bigg\}.
\end{aligned}
\end{equation}
\end{prop}
The proof of the proposition is again done in Section~\ref{sec:hierarchy}.

\subsection{Quantitative stability estimates between the hierarchies} \label{subsec:main_quantitative}
We are now ready to state the main quantitative result in this paper, which compares the observables $\tau_N(T,w_N,f_N)$ satisfying the approximate hierarchy \eqref{eqn:hierarchy_equation}-\eqref{eqn:hierarchy_equation_remainder} to $\tau_\infty(T)$ satisfying the limiting hierarchy \eqref{eqn:limit_hierarchy_equation}.
The proof of the theorem and the exact derivation of constants $C_1,\;C_2$ in the estimate are performed in Section~\ref{subsec:quantitative}.
%
%
\begin{thm} \label{thm:stable_power_norm}
Assume that $\mu, \nu \in W^{1,\infty}$, $\sigma > 0$ and $N \geq 1$. Let $w_N \defeq (w_{i,j;N})_{i,j =1}^N \in \R^{N \times N}$ be a connectivity matrix and $f_{N}^{i_1,\dots, i_k}$, $\forall \{i_1,\dots, i_k\} \subset \{1 \dots N\}$ be marginal laws, from which the hierarchy of observables $\tau_N(T,w_N,f_N)$ and the absolute observables $|\tau_N|(T,w_N,f_N)$ are defined and satisfy \eqref{eqn:hierarchy_equation}-\eqref{eqn:hierarchy_equation_remainder} in distributional sense.
Denote the strength of pairwise interaction as
\begin{equation*}
\begin{aligned}
\bar w_N \defeq \max_{1\leq i,j \leq N} |w_{i,j;N}|.
\end{aligned} 
\end{equation*}
In addition, let $\tau_\infty(T) \in L^\infty([0,t_*]; \mathcal{M}(\R^{|T|}))$, $\forall T \in \mathcal{T}$ satisfy \eqref{eqn:limit_hierarchy_equation} in distributional sense.

For some choice of $\lambda>0$ and $\alpha > 0$ (which determines $\eta = \eta_\alpha$), assume that there exists $n\in \N$ s.t.
%
\begin{equation*}
\begin{aligned}
\bar{\varepsilon} \defeq C_1 \big[ \exp \big( (2 + 2\alpha) n \bar w_N \big) - 1 \big] + (1/4)^n<1,
\end{aligned} 
\end{equation*}
where $C_1$ is a constant depending only on $\|\mu\|_{W^{1,\infty}},\|\nu\|_{W^{1,\infty}},\sigma$ and the scaling factor~$\lambda > 0$.
Then the following estimate holds: for any tree $T_* \in \mathcal{T}$, 
\begin{equation} \label{eqn:stable_power_norm}
\begin{aligned}
&\sup_{t\leq t_*} \; (\lambda / 8)^{|T_*|} \| \tau_N(T_*,w_N,f_N)(t,\cdot) - \tau_\infty(T_*)(t,\cdot) \|_{H^{-1\otimes |T_*|}_\eta}^2\\
&\quad \leq C_2\, C_{\lambda;\eta}^2\,
\left\{ \max\left( \bar{\varepsilon},\ \max_{|T| \leq \max(n,\ |T_*|)} \frac{(\lambda / 8)^{|T|}}{C_{\lambda;\eta}^2\,}\, \| \tau_N(T,w_N,f_N)(0,\cdot) - \tau_\infty(T)(0,\cdot) \|_{H^{-1\otimes |T|}_\eta}^2 \right) \right\}^{1/C_2},
\end{aligned}
\end{equation}
where $C_2$ depends only on $t_*$, $\|\mu\|_{W^{1,\infty}},\|\nu\|_{W^{1,\infty}},\sigma$ and $\lambda > 0$, and where $C_{\lambda;\eta}$ depends on the following a priori estimate
\begin{equation}\label{eqn:hierarchy_boundedness_1}
\sup_{t\leq t_*} \ \max_{|T| \leq \max(n,\ |T_*|)} \lambda^{\frac{|T|}{2}}\, \left(\||\tau_N|(T,w_N,f_N)(t,\cdot)\|_{H^{-1\otimes |T|}_\eta} 
+\|\tau_\infty(T)(t,\cdot)\|_{H^{-1\otimes |T|}_\eta}\right) \leq C_{\lambda;\eta}. 
\end{equation}

\end{thm}

\begin{remark}
The values of $\lambda$, $\alpha$, and $n$ must be chosen carefully for this result to be useful. The scaling factors $\lambda$ and $\alpha$ need to be selected so that the various norms in the theorem are finite, to fit with the existing a priori estimates. Also, we need to have $n$ s.t. $\bar \varepsilon$ is small enough, which would typically lead to taking $n\sim \frac{|\log \bar w_N|}{\bar w_N}$. However $n$ also enters in the definition of $C_{\lambda:\eta}$ in an implicit way as a larger value of $n$ forces to take the $\max$ over more trees $T$. Hence the actual optimal value of $n$ is not so easy to determine unless \eqref{eqn:hierarchy_boundedness_1} is a priori given where the maximum is replaced by the supremum over all trees $T \in \mathcal{T}$.
\end{remark}

Stability and uniqueness estimates on the kind of generalized hierarchy that we are dealing with here are notoriously difficult, with only limited results available.  As we mentioned before there are obvious similarities between our approach and the hierarchy derived in \cite{JaPoSo:21} or the strong estimates on the classical BBGKY hierarchy in \cite{BrJaSo:22} (leading for example to the mean-field limit to the Vlasov-Fokker-Planck-Poisson equation). We also mention results around the wave kinetic equation in \cite{DeHa:21,DeHa:23}. 

A major difference in Theorem~\ref{thm:stable_power_norm} is that the observables $\tau_N$ do not solve an exact hierarchy and the remainder terms only vanish in some weak norms. As we briefly explained earlier, this forces the use of the $H^{-1\otimes |T|}_\eta$ norm to both control the remainders and to have appropriate commutator estimates, which is the main technical innovation in the paper.

We also emphasize that the general method used to derive stability estimates relies on recursive inequalities, which often leads to a blow-up in finite time. Those do not occur here because we can derive a priori estimates, namely \eqref{eqn:hierarchy_boundedness_1} from  Proposition~\ref{prop:Hilbert_a_priori_bound} and Proposition~\ref{prop:hierarchy_of_equations_limit}, that are strong enough with respect to the weak norms that we are using. 
%
%
\subsection{Proving Theorem~\ref{thm:stable_qualitative} from our quantitative estimates}
We conclude this subsection by explaining how Theorem~\ref{thm:stable_qualitative} follows from all the estimates presented here.
\begin{proof} [Proof of Theorem~\ref{thm:stable_qualitative}]
The first step  is to make sure that we can apply Theorem~\ref{thm:stable_power_norm} from the assumptions~\eqref{eqn:assumption_1}-\eqref{eqn:assumption_4} in Theorem~\ref{thm:stable_qualitative}. More precisely, we tend to show that \eqref{eqn:hierarchy_boundedness_1} in Theorem~\ref{thm:stable_power_norm}  hold for some well chosen $\lambda > 0$ and $C_{\lambda;\eta} > 0$, and the maximum over $|T| \leq \max(n,\ |T_*|)$ can actually replaced by the supremum over all trees $T \in \mathcal{T}$.

Recall that for any $k \geq 1$,
\begin{equation*}
\begin{aligned}
\eta_a^{\otimes k}(z_1,\dots,z_k) = C_a^k \exp \Big( {\textstyle \sum_{m=1}^k } \sqrt{1 + a^2 z_m^2} \Big),
\end{aligned}
\end{equation*}
hence
\begin{equation*}
\begin{aligned}
\exp \Big( {\textstyle \sum_{m=1}^k } a |z_m| \Big) \leq \eta_a^{\otimes k}(z_1,\dots,z_k) \leq (C_a \exp(1))^k \exp \Big( {\textstyle \sum_{m=1}^k } a |z_m| \Big).
\end{aligned}
\end{equation*}
Thus, from assumption~\eqref{eqn:assumption_3} in Theorem~\ref{thm:stable_qualitative}, the following two inequalities about the initial data can immediately be derived,
\begin{equation*}
\begin{aligned}
\sup_N \; \||\tau_N|(T)(0,\cdot) \eta_a^{\otimes |T|}\|_{\mathcal{M}(\R^{|T|})} \leq \;& \Big( M_a C_a \exp(1) \Big)^{|T|}, \quad \forall T \in \mathcal{T},
\\
\| f(0,\cdot,\cdot) \|_{L^\infty_\xi (H^{-1}_{\eta_a})_x} \leq \;& \|K\|_{L^2(\R)} M_a C_a \exp(1).
\end{aligned}
\end{equation*}
Now, applying Proposition~\ref{prop:Hilbert_a_priori_bound} and Proposition~\ref{prop:hierarchy_of_equations_limit} to the two initial bounds, we obtain the exponential moment bound
\begin{equation*}
\begin{aligned}
&\sup_N \; \int_{\R^{|T|}} { \textstyle \exp\big( a \sum_{m=1}^{|T|} |z_m|\big)} |\tau_N|(T)(t,\rd z)
\\
& \quad \leq \||\tau_N|(T)(t,\cdot) \eta_a^{\otimes |T|}\|_{\mathcal{M}(\R^{|T|})} \leq \Big( M_a C_a \exp(1) \exp(A_\eta t_*) \Big)^{|T|}, \quad \forall t \in [0,t_*], \; T \in \mathcal{T},
\end{aligned}
\end{equation*}
and a priori energy bounds 
\begin{equation*}
\begin{aligned}
&\||\tau_N|(T)(t,\cdot) \big\|_{H^{-1 \otimes |T|}_{\eta_a}} \leq \Big( \|K\|_{L^2(\R)} M_a C_a \exp(1) \exp(A_\eta t_*) \Big)^{|T|}, \quad \forall t \in [0,t_*], \; T \in \mathcal{T}, 
\\
&\big\| \tau_\infty(T,w,f)(t,\cdot) \big\|_{H^{-1 \otimes |T|}_{\eta_a}} \leq  \|w\|_{\mathcal{W}}^{|T| - 1} \|f\|_{L^\infty_{t,\xi} (H^{-1}_{\eta_a})_x}^{|T|}, \quad \forall t \in [0,t_*], \; T \in \mathcal{T},
\end{aligned}
\end{equation*}
where the coefficient $A_\eta$ inside the exponent now only depend on $a$, $\|\mu\|_{W^{1,\infty}}$, $\|\nu\|_{W^{1,\infty}}$, $\sigma$ and
\begin{equation*}
\begin{aligned}
\textstyle \max\left(\max_{i} \sum_{j} |w_{i,j;N}|,\ \max_{j} \sum_{i} |w_{i,j;N}|\right).
\end{aligned}
\end{equation*}
This guarantees \eqref{eqn:hierarchy_boundedness_1} where the maximum over $|T| \leq \max(n,\ |T_*|)$ is replaced by the supremum over all trees $T \in \mathcal{T}$, with $\lambda$, $C_{\lambda;\eta}$ chosen as
\begin{equation*}
\begin{aligned}
\lambda = \min \bigg( \Big( \|K\|_{L^2(\R)} M_a C_a \exp(1) \exp(A_\eta t_*) \Big)^{-2} , \Big( \max \big(\|w\|_{\mathcal{W}} , 1 \big) \|f\|_{L^\infty_{t,\xi} (H^{-1}_\eta)_x} \Big)^{-2} \bigg),
 \quad \quad C_{\lambda;\eta} = 1.
\end{aligned}
\end{equation*}
Hence the assumptions of Theorem~\ref{thm:stable_power_norm} are satisfied and we apply it along the following point.
\begin{itemize}
\item Using \eqref{eqn:assumption_1} and \eqref{eqn:assumption_3}, we choose the coefficients $\alpha \in (0,a)$, $\lambda > 0$, $C_{\lambda;\eta} > 0$ in \eqref{eqn:hierarchy_boundedness_1} independent of $N$, and the supremum in \eqref{eqn:hierarchy_boundedness_1} is taken over all possible $T \in \mathcal{T}$. This implies in particular a uniform bound on exponential moments with coefficient $a > 0$ so that Lemma~\ref{lem:topology_negative_Sobolev} applies.
\item Fix $T_* \in \mathcal{T}$. For any $\varepsilon > 0$, choose sufficiently large $n$, such that
\begin{equation*}
\begin{aligned}
(\lambda/8)^{-\frac{|T_*|}{2}} \sqrt{C_2}\, C_{\lambda;\eta}\,
\big[2 (1/4)^n \big]^{1/2 C_2} \leq \varepsilon.
\end{aligned}
\end{equation*}
\item By \eqref{eqn:assumption_2}, we choose sufficiently large $N_1$, such that for all $N \geq N_1$ the corresponding 
\begin{equation*}
\begin{aligned}
\bar w_N \defeq \max_{1\leq i,j \leq N} |w_{i,j;N}|
\end{aligned}
\end{equation*}
is sufficiently small such that
\begin{equation*}
\begin{aligned}
C_1 \big[ \exp \big( (2 + 2\alpha) n \bar w_N \big) - 1 \big] \leq (1/4)^n.
\end{aligned}
\end{equation*}
\item Notice that there are only finitely many $T \in \mathcal{T}$ satisfying $|T| \leq \max(n, \ |T_*|)$.
By \eqref{eqn:assumption_4} on the weak-* convergence of initial data, and by Lemma~\ref{lem:topology_negative_Sobolev}, 
choose a sufficiently large $N_2 \geq 1$ such that for all $N \geq N_2$,
\begin{equation*}
\begin{aligned}
\max_{|T| \leq \max(n,\ |T_*|)} (\lambda / 8)^{|T|} \| \tau_N(T,w_N,f_N)(0,\cdot) - \tau_\infty(T)(0,\cdot) \|_{H^{-1\otimes |T|}_\eta}^2 / (4 C_{\lambda;\eta}^2) \leq 2(1/4)^n.
\end{aligned}
\end{equation*}
\item In summarize, for any $T_* \in \mathcal{T}$ and any $\varepsilon>0$, by taking $N \geq \max(N_1,\ N_2)$ according to our previous discussion and applying Theorem~\ref{thm:stable_power_norm}, we obtain that 
\begin{equation} \label{eqn:epsilon_delta_hierarchy_convergence}
\begin{aligned}
\sup_{t\in [0,\ t_*]} \|\tau_N(T_*,w_N,f_N)(t,\cdot) - \tau_\infty(T_*)(t,\cdot)\|_{H^{-1\otimes |T_*|}_\eta} \leq \varepsilon. 
\end{aligned}
\end{equation}
\item Invoking again Lemma~\ref{lem:topology_negative_Sobolev}, we finally deduce that
\[
\lim_{N \to \infty} \tau_N(T,w_N,f_N)(t,\cdot) = \tau_\infty(T)(t,\cdot), \quad \forall T \in \mathcal{T}
\]
in the weak-* topology of $L^\infty([0,\ t_*],\ \mathcal{M})$.
\end{itemize}

\end{proof}
%
\section{The weak norm and the exponential moments} \label{sec:negative_Sobolev}
\subsection{Basic properties} \label{subsec:negative_Sobolev}
We first revisit our definition of the kernel $K$, and introduce another kernel, denoted as $\Lambda$, as follows,
\begin{equation*}
\begin{aligned}
K(x) \defeq \frac{1}{\pi} \int_0^\infty \exp(-|x| \cosh(\xi)) \rd \xi, \quad \Lambda(x) \defeq \frac{1}{2} \exp(- |x|), \quad \forall x \in \R.
\end{aligned}
\end{equation*}
For $x > 0$, the kernel $K$ is, in fact, the zero-th order modified Bessel function of second type. From the known properties of Bessel functions, $K$ is a non-negative, radially-decreasing $L^2$ function, and satisfies
\begin{equation*}
\begin{aligned}
K \star K = \Lambda, \quad \widehat{K}(\xi) = \int_{\R} K(x) \exp(-2\pi i x \xi) \;\rd x = \frac{1}{\sqrt{1 + 4 \pi^2 \xi^2}}.
\end{aligned}
\end{equation*}
It is easy to extend the identity $K \star K = \Lambda$ to the tensorized kernels $K^{\otimes k} \star K^{\otimes k}= \Lambda^{\otimes k}$, which yields the following equivalent formalism of $H^{-1 \otimes k}$ by Fourier analysis:
\begin{equation*}
\begin{aligned}
\|f\|_{H^{-1 \otimes k}}^2 = \;& \int_{z \in \R^{k}}  \big[ K^{\otimes k} \star f (z) \big]^2 \;\rd z = \int_{z \in \R^{k}}  f(z) \big[ \Lambda^{\otimes k} \star f (z) \big] \;\rd z
\\
= \;&\int_{\xi \in \R^k} \bigg( \prod_{m=1}^{k} \frac{1}{1 + 4 \pi^2 \xi_m^2} \bigg) \; \hat f(\xi) \hat f(\xi) \;\rd \xi.
\end{aligned}
\end{equation*}

In one dimension, it is straightforward that our notion of $H^{-1 \otimes k}$-norm for $k = 1$ is equivalent to the negative Sobolev norm of $H^{-1}(\R)$, i.e.
\begin{equation*}
\begin{aligned}
\|f\|_{H^{-1 \otimes 1}} = \|f\|_{H^{-1}(\R)},
\end{aligned}
\end{equation*}
provided we define $H^s(\R)$ as
\begin{equation*}
\begin{aligned}
\|g\|_{H^s(\R)}^2 \defeq \int_{\R} \big( 1 + 4 \pi^2 \xi^2 \big)^{s} \big| \hat g(\xi) \big|^2 \;\rd \xi,
\end{aligned}
\end{equation*}
for any $s \in \R$.

This also gives us the duality formula
\begin{equation*}
\begin{aligned}
\|f\|_{H^{-s}(\R)} = \sup_{\|g\|_{H^s(\R)} \leq 1} \bigg| \int_{\R} f(x)g(x) \;\rd x \bigg|,
\end{aligned}
\end{equation*}
and the inequality from Leibniz rule for $s=1$,
\begin{equation*}
\begin{aligned}
\|\nu f\|_{H^{-1}(\R)} = \sup_{\|g\|_{H^1} \leq 1} \bigg| \int_{\R} g(x)\nu(x)f(x) \;\rd x \bigg| \leq \sup_{\|g\|_{H^1} \leq 1} \|g\nu\|_{H^1} \|f\|_{H^{-1}} \leq 2 \|\nu\|_{W^{1,\infty}} \|f\|_{H^{-1}(\R)}.
\end{aligned}
\end{equation*}
\subsection{Tensorization properties}
%
In higher dimensions, our notion of $H^{-1 \otimes k}$-norm is the tensorization of $H^{-1}(\R)$-norm to $\R^k$:
\begin{lem} \label{lem:stability_under_tensorization}
For any weight function $\eta: \R \to \R_+$, 
one has
\begin{equation*}
\begin{aligned}
\|f^{\otimes k}\|_{H^{-1 \otimes k}} = \big( \|f\|_{H^{-1}(\R)} \big)^k,
\quad
\|f^{\otimes k}\|_{H^{-1 \otimes k}_\eta} = \big( \|f\|_{H^{-1}_\eta(\R)} \big)^k.
\end{aligned}
\end{equation*}
\end{lem}

\begin{proof}
One has that
\begin{equation*}
\begin{aligned}
\|f^{\otimes k}\|_{H^{-1 \otimes k}_\eta}^2 = \;& \int_{z \in \R^{k}}  \big[ K^{\otimes k} \star (f^{\otimes k} \eta^{\otimes k}) (z) \big]^2 \;\rd z = \prod_{m = 1}^k \int_{z_m \in \R}  \big[ K \star (f \eta) (z_m) \big]^2 \;\rd z_m
\\
= \;& \big( \|f\|_{H^{-1}_\eta(\R)} \big)^{2k}.
\end{aligned}
\end{equation*}
The unweighted case of $H^{-1 \otimes k}$ is naturally included by choosing $\eta \equiv 1$.
\end{proof}

It is important to emphasize however that the tensorized $H^{-1 \otimes k}$-norm is weaker than the standard $H^{-1}(\R^k)$-norm since in Fourier
\begin{equation*}
\begin{aligned}
\prod_{m=1}^{k} \frac{1}{1 + 4 \pi^2 \xi_m^2} \ll  \frac{1}{1 + 4 \pi^2 \sum_{m=1}^{k} \xi_m^2}.
\end{aligned}
\end{equation*}
This shows that the energy distributed along the diagonals of the Fourier domain  have a much less contribution to the tensorized $H^{-1 \otimes k}$-norm than to the $H^{-1}(\R^k)$-norm.

Similarly, while it is possible to include $\mathcal{M}(\R^k)$ into the standard $H^{-s}(\R^k)$, the order $s > 0$ in such Sobolev inequalities depends on the dimension $k$, namely $s > {k}/{2}$. On the other hand, the following lemma holds for our notion of $H^{-1 \otimes k}$-norm,
\begin{lem} \label{lem:convolutional_inequality}
Consider $g \in \mathcal{M}(\R^k)$ and any weight function $\eta\in L^1(\R,  \R_+)$ such that $\eta^{\otimes k}$ is integrable against $g$, then
\begin{equation} \label{eqn:convolutional_inequality}
\begin{aligned}
\|g\|_{H^{-1\otimes k}} \defeq \| K^{\otimes k} \star g \|_{L^2(\R^k)} \leq \;& \|K\|_{L^2(\R)}^k \|g\|_{\mathcal{M}(\R^k)},
\\
\|g\|_{H^{-1\otimes k}_\eta} \defeq \|K^{\otimes k} \star (g \eta^{\otimes k})\|_{L^2(\R^k)} \leq \;& \|K\|_{L^2(\R)}^k \| g \eta^{\otimes k}\|_{\mathcal{M}(\R^k)}.
\end{aligned}
\end{equation}
\end{lem}
\begin{proof}
The proof is a simple application of convolutional inequality.
\end{proof}
Hence $\mathcal{M}(\R^k)$ is naturally included in $H^{-1 \otimes k}$, and can also be included into $H^{-1 \otimes k}_{\eta}$, provided that the measure has the right moment bound.

The next lemma extends the inequality from Leibniz rule to any dimension.
\begin{lem} \label{lem:commutator_inequality}

Consider $\nu_m$ of form
\begin{equation*}
\begin{aligned}
\nu_m = 1 \otimes \dots \otimes \nu \otimes \dots \otimes 1,
\end{aligned}
\end{equation*}
where $\nu \in W^{1,\infty}(\R)$ appears in the $m$-th coordinate, i.e. $\nu_m(z) = \nu(z_m)$. Then for any $f \in \mathcal{M}(\R^k) \cap H^{-1 \otimes k}$, the following inequality holds
\begin{equation*}
\begin{aligned}
\|\nu_m f\|_{H^{-1 \otimes k}} \leq 2 \|\nu\|_{W^{1,\infty}(\R)} \|f\|_{H^{-1 \otimes k}},
\end{aligned}
\end{equation*}
while for $f \in \mathcal{M}(\R^k) \cap H^{-1 \otimes k}_\eta$, we have the corresponding
\begin{equation*}
\begin{aligned}
\|\nu_m f\|_{H^{-1 \otimes k}_\eta} \leq 2 \|\nu\|_{W^{1,\infty}(\R)} \|f\|_{H^{-1 \otimes k}_\eta}.
\end{aligned}
\end{equation*}

\end{lem}

\begin{proof}
Let us first discuss the unweighted inequality and WLOG consider $\nu_k$ that is non-constant in the $k$-th dimension. Let us introduce the Fourier transform on the first $k-1$ dimensions
\begin{equation*}
\begin{aligned}
\mathcal{F}^{\otimes k-1} \otimes I: \R^{k-1} \times \R \to \R^{k-1} \times \R.
\end{aligned}
\end{equation*}
It is easy to verify that
\begin{equation*}
\begin{aligned}
\;& \big(\mathcal{F}^{\otimes k-1} \otimes I\big) \big( K^{\otimes k} \star (\nu_{k} f) \big) (\xi_1,\dots,\xi_{k-1}, z_{k})
\\
= \;& \bigg( \prod_{m=1}^{k-1} \frac{1}{\sqrt{1 + 4 \pi^2 \xi_m^2}} \bigg) \; \big( K \star_{k} (\nu_{k} \mathcal{F}^{\otimes k-1} f) \big) (\xi_1,\dots,\xi_{k-1}, z_{k}).
\end{aligned}
\end{equation*}
By Plancherel identity,
\begin{equation*}
\begin{aligned}
\|\nu_{k} f\|_{H^{-1 \otimes k}}^2 = \;& \int \bigg| \bigg( \prod_{m=1}^{k-1} \frac{1}{\sqrt{1 + 4 \pi^2 \xi_m^2}} \bigg) \; \big( K \star_{k} (\nu_{k} \mathcal{F}^{\otimes k-1} f) \big) (\xi_1,\dots,\xi_{k-1}, z_{k}) \bigg|^2 \;\rd \xi_1, \dots, \xi_{k-1} \rd z_{k}
\\
= \;& \int \bigg( \prod_{m=1}^{k-1} \frac{1}{1 + 4 \pi^2 \xi_m^2} \bigg) \Big\| \big( \nu_{k} \mathcal{F}^{\otimes k-1} f \big) (\xi_1,\dots,\xi_{k-1}, \cdot) \Big\|_{H^{-1}(\R)}^2 \rd \xi_1, \dots, \xi_{k-1}.
\end{aligned}
\end{equation*}
Since $\nu \in W^{1,\infty}(\R)$,
\begin{equation*}
\begin{aligned}
\Big\| \big( \nu_{k} \mathcal{F}^{\otimes k-1} f \big) (\xi_1,\dots,\xi_{k-1}, \cdot) \Big\|_{H^{-1}(\R)} \leq 2\|\nu\|_{W^{1,\infty}(\R)} \Big\| \mathcal{F}^{\otimes k-1} f (\xi_1,\dots,\xi_{k-1}, \cdot) \Big\|_{H^{-1}(\R)}.
\end{aligned}
\end{equation*}
Hence
\begin{equation*}
\begin{aligned}
\|\nu_{k} f\|_{H^{-1 \otimes k}}^2 
\leq \;& 4\|\nu\|_{W^{1,\infty}(\R)}^2 \int \bigg( \prod_{m=1}^{k-1} \frac{1}{1 + 4 \pi^2 \xi_m^2} \bigg) \Big\| \mathcal{F}^{\otimes k-1} f (\xi_1,\dots,\xi_{k-1}, \cdot) \Big\|_{H^{-1}(\R)}^2 \rd \xi_1, \dots, \xi_{k-1}
\\
= \;& 4\|\nu\|_{W^{1,\infty}(\R)}^2 \|f\|_{H^{-1 \otimes k}}^2,
\end{aligned}
\end{equation*}
which completes the proof of unweighted inequality. Finally, for the weighted inequality, we can apply the unweighted inequality to obtain
\begin{equation*}
\begin{aligned}
\|\nu_m f\|_{H^{-1 \otimes k}_\eta} = \|\nu_m f \eta^{\otimes k}\|_{H^{-1 \otimes k}} \leq 2 \|\nu\|_{W^{1,\infty}} \|f \eta^{\otimes k}\|_{H^{-1 \otimes k}} = 2 \|\nu\|_{W^{1,\infty}} \|f\|_{H^{-1 \otimes k}_\eta}.
\end{aligned}
\end{equation*}
\end{proof}
%
%
\subsection{The weak-* topology on measures}
Now, we proceed to the proof of Lemma~\ref{lem:topology_negative_Sobolev}, restated here.
\begin{lem}
Consider any $a > 0$, $C_a > 0$, $0 < \alpha < a$ (which determines $\eta = \eta_\alpha$) and any sequence
\begin{equation} \label{eqn:exponential_tightness}
\begin{aligned}
\{g_n\}_{n = 1}^\infty \subset \bigg\{g \in \mathcal{M}(\R^k) : \int_{\R^k} { \textstyle \exp\big( a \sum_{m=1}^{k} |z_m|\big)} |g|(\rd z) \leq C_a\bigg\}.
\end{aligned}
\end{equation}
Then the following are equivalent:
\begin{itemize}
\item $g_n \overset{\ast}{\rightharpoonup} g_\infty$ under the weak-* topology of $\mathcal{M}(\R^k)$.
\item $\| g_n - g_\infty \|_{H^{-1\otimes k}_\eta} \to 0$.
\end{itemize}
\end{lem}

\begin{proof} [Proof of Lemma~\ref{lem:topology_negative_Sobolev}]

A sequence $\{g_n\}_{n = 1}^\infty$ satisfying \eqref{eqn:exponential_tightness} is uniformly tight and bounded in total variation norm. By Prokhorov's theorem, $\{g_n\}_{n = 1}^\infty$ is sequentially precompact in the weak-* topology. Assuming now that $\| g_n - g_\infty \|_{H^{-1\otimes k}_\eta} \to 0$, the definition of $H^{-1\otimes k}_\eta$ directly implies that $(g_n-g_\infty) \, \eta^{\otimes k}$ converges to $0$ in the sense of distribution. Since $\eta = \eta_\alpha$ is smooth, bounded from below and from above on any compact, it further yields that $g_n$ converges to $g_\infty$, still in the sense of distributions. Hence we immediately have that $g_n \overset{\ast}{\rightharpoonup} g_\infty$ under the weak-* topology of $\mathcal{M}(\R^k)$.
  
%
Assuming now only that $g_n \overset{\ast}{\rightharpoonup} g_\infty$ under the weak-* topology of $\mathcal{M}(\R^k)$. First recall that
\begin{equation*}
\begin{aligned}
\eta^{\otimes k}(z_1,\dots,z_k) = C_\alpha^k \exp \Big( {\textstyle \sum_{m=1}^k } \sqrt{1 + \alpha^2 z_m^2} \Big) \leq (C_\alpha \exp(1))^k \exp \Big( {\textstyle \sum_{m=1}^k } \alpha |z_m| \Big).
\end{aligned}
\end{equation*}
The kernel $\Lambda^{\otimes k}$ is Lipschitz. Hence the convolution $\Lambda^{\otimes k} \star (g_n \eta^{\otimes k})$ is also Lipschitz, by
\begin{equation*}
\begin{aligned}
  \| \Lambda^{\otimes k} \star (g_n \eta^{\otimes k}) \|_{W^{1,\infty}} \leq \| \Lambda^{\otimes k} \|_{W^{1,\infty}}\, \|g_n \eta^{\otimes k}\|_{\mathcal{M}}.
\end{aligned}
\end{equation*}
By the exponential moment bound \eqref{eqn:exponential_tightness}, we have
\begin{equation*}
\begin{aligned}
\|g_n \eta^{\otimes k}\|_{\mathcal{M}} = \;& (C_\alpha \exp(1))^k \int_{\R^k} { \textstyle \exp\big( - (a-\alpha) \sum_{m=1}^{k} |z_m|\big)} { \textstyle \exp\big( a \sum_{m=1}^{k} |z_m|\big)} |g_n|(\rd z)
\\
\leq \;& (C_\alpha \exp(1))^k C_a.
\end{aligned}
\end{equation*}
This implies that $g_n\,\eta^{\otimes k}$ is precompact and hence converges to $g_\infty \,\eta^{\otimes k}$, so that
\begin{equation*}
\begin{aligned}
\Lambda^{\otimes k} \star (g_n\, \eta^{\otimes k}) \to \phi=\Lambda^{\otimes k} \star (g_\infty\, \eta^{\otimes k}) \in C(\R^k) \text{ uniformly on all compact subset of } \R^k.
\end{aligned}
\end{equation*}
Let $\rho \in C_c(\R)$ such that $0 \leq \rho \leq 1$, $\rho([-1,1]) \equiv 1$, $\supp \rho \subset [-2,2]$ and denote $\rho_R(x) = \rho(x/R)$. Then
\begin{equation*}
\begin{aligned}
\|g_n\|_{H^{-1 \otimes k}_\eta}^2 = \;& \int_{z \in \R^{k}}  (g_n \eta^{\otimes k})(z) \big[ \Lambda^{\otimes k} \star (g_n \eta^{\otimes k})(z) \big] \;\rd z
\\
\leq \;& \int_{z \in \R^{k}}  (g_n \eta^{\otimes k})(z) (\phi \rho_R^{\otimes k})(z) \;\rd z
\\
\;& + \int_{z \in \R^{k}}  (g_n\, \eta^{\otimes k})(z) ((\Lambda^{\otimes k} \star (g_n \eta^{\otimes k}) - \phi) \rho_R^{\otimes k})(z) \;\rd z
\\
\;& + \int_{z \in \R^{k}}  (g_n \eta^{\otimes k})(z) ((\Lambda^{\otimes k} \star (g_n \eta^{\otimes k})) (1 - \rho_R^{\otimes k}))(z) \;\rd z
\\
\eqdef \;& L_1 + L_2 + L_3.
\end{aligned}
\end{equation*}
We note that $\phi\, \rho_R^{\otimes k}$ is continuous and compactly supported so that, for a fixed $R$, $L_1$ converges to $0$ from the weak-* convergence of $g_n$. $L_2$ also directly converges to $0$ for a fixed $R$ from the uniform convergence of $\Lambda^{\otimes k} \star (g_n \eta^{\otimes k})$ to $\phi$ on compact sets.

Finally, for any $\varepsilon > 0$, choose sufficiently large $R > 0$ such that
\begin{equation*}
\begin{aligned}
\big[ C_\alpha { \textstyle \exp\big( 1 - (a-\alpha) R\big)} \big]^k \leq \frac{\varepsilon / 6}{ \|\Lambda^{\otimes k}\|_{L^\infty} (C_\alpha \exp(1))^k C_a^2}.
\end{aligned}
\end{equation*}
Then
\begin{equation*}
\begin{aligned}
L_3 \leq \;& \|\Lambda^{\otimes k}\|_{L^\infty} (C_\alpha \exp(1))^k C_a \int_{z \in \R^{k}}  |g_n \eta^{\otimes k}|(z) (1 - \rho_R^{\otimes k})(z) \;\rd z 
\\
\leq \;& \|\Lambda^{\otimes k}\|_{L^\infty}\, (C_\alpha \exp(1))^k\, C_a\, \big[ C_\alpha\, { \textstyle \exp\big( 1 - (a-\alpha) R\big)} \big]^k\, C_a \,\leq \varepsilon / 6.
\end{aligned}
\end{equation*}
This shows that $L_3$ converges to $0$ as $R \to \infty$ uniformly in $n$, which concludes.
\end{proof}
\subsection{Bounding the remainder terms}
As a first example of application of our weak norms, we can derive a quantified weak convergence of the remainder terms $\mathscr{R}$ and $\mathscr{\tilde R}$ in \eqref{eqn:hierarchy_equation_remainder}.
%
$L^p$ norms are too sensitive to the pointwise density of the distribution, which makes it difficult to quantify vanishing translations.
The following lemma shows how such translations are smoothen when mollified by $\Lambda^{\otimes k}$, making the behavior of $\mathscr{R}$ and $\mathscr{\tilde R}$ milder in the $H^{-1 \otimes k}$ sense and laying the ground for our future commutator estimates.
\begin{lem} \label{lem:translation_estimate}
For any non-negative measure $f \in \mathcal{M}_+(\R^k)$ and vector $w \in \R^{k}$, the following pointwise estimate holds 
\begin{equation*}
\begin{aligned}
\big| (\Lambda^{\otimes k} \star f)(z - w) - (\Lambda^{\otimes k} \star f)(z) \big|
\leq \big[ \exp \big( {\|w\|_{\ell^1}} \big) - 1 \big] (\Lambda^{\otimes k} \star f) (z), \quad \forall z \in \R^{k}.
\end{aligned}
\end{equation*}
\end{lem}
\begin{proof}
It is straightforward that
\begin{equation*}
\begin{aligned}
\Big| (\Lambda^{\otimes k} \star f)(z - w) - (\Lambda^{\otimes k} \star f)(z) \Big|
\leq \int_{\R^{k}} \Big| \Lambda^{\otimes k} (z - w - y) - \Lambda^{\otimes k} (z - y) \Big| f(y) \;\rd y.
\end{aligned}
\end{equation*}
From the formula,
\begin{equation*}
\begin{aligned}
\Lambda^{\otimes k} (z) = \frac{1}{2} \exp\Big( - \sum_{m=1}^k |z_m|\Big),
\end{aligned}
\end{equation*}
we have that
\begin{equation*}
\begin{aligned}
\Big| \Lambda^{\otimes k} (z - w - y) - \Lambda^{\otimes k} (z - y) \Big| \leq \big[ \exp \big( {\|w\|_{\ell^1}} \big) - 1 \big] \Lambda^{\otimes k} (z - y).
\end{aligned}
\end{equation*}
We conclude the lemma by multiplying both sides by $f(y)$ and integrate by $y$.
\end{proof}
The following proposition summarizes the estimates of $\mathscr{R}$ and $\mathscr{\tilde R}$ terms.
\begin{prop} \label{prop:translation_estimate_weighted} 
Consider any $\alpha > 0$ (which determines $\eta = \eta_\alpha$), any connectivity matrix $w_N \in \R^{N \times N}$ and any joint law $f_N \in \mathcal{M}_+(\R^{N})$. Let $\mathscr{R}_{N,T,m}$ and $\mathscr{\tilde R}_{N,T,m}$ be the remainder terms as in \eqref{eqn:hierarchy_equation_remainder} and let $|\tau_N|(T) = |\tau_N|(T,w_N,f_N)$ as in Definition~\ref{defi:observables} (where the variable $t$ shall be neglected).
Then the following estimate holds:
\begin{equation*}
\begin{aligned}
\max\left(\| \mathscr{R}_{N,T,m} \|_{H^{-1 \otimes |T|}_\eta}^2,\ \| \mathscr{\tilde R}_{N,T,m} \|_{H^{-1 \otimes |T|}_\eta}^2\right)
\leq \big[ \exp \big( (2 + 2\alpha) c(w_N,|T|) \big) - 1 \big] \||\tau_N|(T)\|_{H^{-1 \otimes |T|}_\eta}^2,
\end{aligned}
\end{equation*}
where
\begin{equation*}
\begin{aligned}
c(w_N,|T|) \defeq \min\left(|T| \big(\max_{i,j}|w_{i,j;N}| \big),\ \max\Big(\max_j \sum_i |w_{i,j;N}|, \max_i \sum_j |w_{i,j;N}|\Big)\right).
\end{aligned}
\end{equation*}
\end{prop}
Notice that the right hand side of the inequality is the ``absolute''   observables $|\tau_N|$ instead of $\tau_N$ as non-negativity plays a role in the proof.
The constant $\alpha > 0$ takes the effect of weight $\eta = \eta_\alpha$ into account.
\begin{proof} [Proof of Proposition~\ref{prop:translation_estimate_weighted}]
Once we obtain the bound of $\mathscr{R}_{N,T,m}$, we can derive the same bound of $\mathscr{\tilde R}_{N,T,m}$ by Minkowski inequality.
Hence, let us only consider $\mathscr{R}_{N,T,m}$.
For simplicity, we also omit $t$ variable in the proof.

By definition,
\begin{equation*}
\begin{aligned}
\| \mathscr{R}_{N,T,m} \|_{H^{-1 \otimes |T|}_\eta}^2 = \;& \int_{\R^{|T|}} \big[ \big(\mathscr{R}_{N,T,m}\eta^{\otimes n}\big) (z) \big] \big[ \Lambda^{\otimes n} \star \big( \mathscr{R}_{N,T,m}\eta^{\otimes n} \big) (z) \big] \;\rd z
\\
\leq \;& \int_{\R^{|T|}} \big| \big(\mathscr{R}_{N,T,m}\eta^{\otimes n}\big) (z) \big| \big| \Lambda^{\otimes n} \star \big( \mathscr{R}_{N,T,m}\eta^{\otimes n} \big) (z) \big| \;\rd z.
\end{aligned}
\end{equation*}
We recall the notation
\[
w_{N;j}^{i_1,\dots, i_{|T|}}=(w_{i_l,j;N})_{l=1}^{|T|},
\]
so that
\begin{equation*}
\begin{aligned}
\|w_{N;j}^{i_1,\dots, i_{|T|}}\|_{\ell^1} &\leq \min \left(|T| \big(\max_{i,j}|w_{i,j;N}| \big),\ \max\Big(\max_j \sum_i |w_{i,j;N}| , \max_i \sum_j |w_{i,j;N}|\Big)\right) \\
&= c(w,|T|).
\end{aligned}
\end{equation*}
By Lemma~\ref{lem:translation_estimate}, since the marginals are non-negative,
\begin{equation*}
\begin{aligned}
\;& \big| \Lambda^{\otimes n} \star \big( \mathscr{R}_{N,T,m}\eta^{\otimes n} \big) (z) \big|
\\
\leq \;& \frac{1}{N} \sum_{i_1,\dots, i_{|T|} = 1}^N \big| w_{N,T}(i_1,\dots, i_{|T|}) \big| \bigg| \Lambda^{\otimes n} \star \bigg( \big( f_{N}^{i_1,\dots, i_{|T|}}(\cdot - w_{N;i_m}^{i_1,\dots, i_{|T|}}) - f_{N}^{i_1,\dots, i_{|T|}}(\cdot) \big)\eta^{\otimes n} \bigg) (z) \bigg|
\\
\leq \;& \frac{1}{N} \sum_{i_1,\dots, i_{|T|} = 1}^N \big| w_{N,T}(i_1,\dots, i_{|T|}) \big| \big[ \exp \big( (1 + \alpha) c(w,|T|) \big) - 1 \big] \big[ \Lambda^{\otimes n} \star \big( f_{N}^{i_1,\dots, i_{|T|}} \eta^{\otimes n} \big) (z) \big]
\\
= \;& \big[ \exp \big( (1 + \alpha) c(w,|T|) \big) - 1 \big] \big[ \Lambda^{\otimes n} \star \big( |\tau_N|(T) \eta^{\otimes n} \big) (z) \big].
\end{aligned}
\end{equation*}

Then
\begin{equation*}
\begin{aligned}
\| \mathscr{R}_{N,T,m} \|_{H^{-1 \otimes |T|}_\eta}^2
\leq \;& \big[ \exp \big( (1 + \alpha) c(w,|T|) \big) - 1 \big] \int_{\R^{|T|}} \big| \big(\mathscr{R}_{N,T,m}\eta^{\otimes n}\big) (z) \big| \big[ \Lambda^{\otimes n} \star \big( |\tau_N|(T) \eta^{\otimes n} \big) (z) \big] \;\rd z
\\
\leq \;& \big[ \exp \big( (1 + \alpha) c(w,|T|) \big) - 1 \big] \int_{\R^{|T|}} \big[ \Lambda^{\otimes n} \star \big|\mathscr{R}_{N,T,m}\eta^{\otimes n} \big| (z) \big] \big[ \big( |\tau_N|(T) \eta^{\otimes n} \big) (z) \big] \;\rd z.\\
\end{aligned}
\end{equation*}
We hence need to bound also $\Lambda^{\otimes n} \star \big|\mathscr{R}_{N,T,m}\eta^{\otimes n} \big|$ with the absolute value inside but 
\begin{equation*}
\begin{aligned}
\;& \big( \Lambda^{\otimes n} \star \big|\mathscr{R}_{N,T,m}\eta^{\otimes n} \big| \big)(z)
\\
= \;& \frac{1}{N} \sum_{i_1,\dots, i_{|T|} = 1}^N \big| w_{N,T}(i_1,\dots, i_{|T|}) \big| \bigg[ \Lambda^{\otimes n} \star \bigg( \big( f_{N}^{i_1,\dots, i_{|T|}}(\cdot - w_{N;i_m}^{i_1,\dots, i_{|T|}}) + f_{N}^{i_1,\dots, i_{|T|}}(\cdot) \big)\eta^{\otimes n} \bigg) (z) \bigg]
\\
= \;& 2\,\Lambda^{\otimes n} \star \big( |\tau_N|(T) \eta^{\otimes n} \big) (z)\\
&+\frac{1}{N} \sum_{i_1,\dots, i_{|T|} = 1}^N \big| w_{N,T}(i_1,\dots, i_{|T|}) \big| \bigg[ \Lambda^{\otimes n} \star \bigg( \big( f_{N}^{i_1,\dots, i_{|T|}}(\cdot - w_{N;i_m}^{i_1,\dots, i_{|T|}}) - f_{N}^{i_1,\dots, i_{|T|}}(\cdot) \big)\eta^{\otimes n} \bigg) (z) \bigg].\\
\end{aligned}
\end{equation*}
Hence again by Lemma~\ref{lem:translation_estimate},
\begin{equation*}
\begin{aligned}
  \;& \big( \Lambda^{\otimes n} \star \big|\mathscr{R}_{N,T,m}\eta^{\otimes n} \big| \big)(z) \leq  2\,\Lambda^{\otimes n} \star \big( |\tau_N|(T) \eta^{\otimes n} \big) (z)\\
  &+\frac{1}{N} \sum_{i_1,\dots, i_{|T|} = 1}^N \big| w_{N,T}(i_1,\dots, i_{|T|}) \big| \big[ \exp \big( (1 + \alpha) c(w,|T|) \big) - 1 \big] \big[ \Lambda^{\otimes n} \star \big( f_{N}^{i_1,\dots, i_{|T|}} \eta^{\otimes n} \big) (z) \big]
\\
&\qquad =\big[ \exp \big( (1 + \alpha) c(w,|T|) \big) + 1 \big] \big[ \Lambda^{\otimes n} \star \big( |\tau_N|(T) \eta^{\otimes n} \big) (z) \big].
\end{aligned}
\end{equation*}
In conclusion
\begin{equation*}
\begin{aligned}
\| \mathscr{R}_{N,T,m} \|_{H^{-1 \otimes |T|}_\eta}^2
\leq \;& \big[ \exp \big( (1 + \alpha) c(w,|T|) \big)^2 - 1 \big] \int_{\R^{|T|}} \big[ \Lambda^{\otimes n} \star \big( |\tau_N|(T) \eta^{\otimes n} \big) (z) \big] \big[ \big( |\tau_N|(T) \eta^{\otimes n} \big) (z) \big] \;\rd z
\\
= \;& \big[ \exp \big( (2 + 2\alpha) c(w,|T|) \big) - 1 \big] \||\tau_N|(T)\|_{H^{-1 \otimes |T|}_\eta}^2.
\end{aligned}
\end{equation*}

\end{proof}
\subsection{Bounding the firing rate through exponential moments}
We present here another set of technical result which shows how to handle the weight function in our subsequent commutator estimates. 
%
\begin{lem} \label{lem:firing_rate_difference}
Consider weight function $\eta = \eta_\alpha$ and any signed measure $f \in \mathcal{M}(\R)$. The following estimate holds:
\begin{equation*}
\begin{aligned}
  \bigg| \int_{\R} K \star (\nu f) \;\rd x\bigg| = \bigg| \int_{\R} \nu f \;\rd x\bigg|  
  \leq C (\alpha) \|\nu\|_{W^{1,\infty}} \|f\|_{H^{-1}_\eta},
\end{aligned}
\end{equation*}
where $C(\alpha)$ only depends on $\alpha > 0$.
\end{lem}
\begin{proof}
  Only the inequality in the statement is not trivial. 
Choose now a non-negative, smooth function $\varphi$ with compact support $\supp \varphi \subset [-1,1]$, such that, $\varphi_{i} = \varphi(\cdot - i)$, $i \in \N$ form a partition of unity of $\R$ in the usual sense that
\begin{equation*}
\begin{aligned}
\sum_{i = -\infty}^{\infty} \varphi(x - i) \equiv 1, \quad \forall x \in \R.
\end{aligned}
\end{equation*}
It is easy to verify that
\begin{equation*}
\begin{aligned}
\int_{\R} \varphi \;\rd x = 1.
\end{aligned}
\end{equation*}
Then
\begin{equation*}
\begin{aligned}
\bigg| \int_{\R} \nu f \;\rd x\bigg|
= \;& \bigg| \int_{\R} (\nu/\eta) f \eta \;\rd x\bigg|
\leq \sum_{i = -\infty}^{\infty} \bigg| \int_{\R} (\nu/\eta) f \eta \,\varphi_i \;\rd x\bigg| = \sum_{i = -\infty}^{\infty} \bigg| \int_{\R} \varphi \star \Big(  (\nu/\eta) f \eta\,\varphi_i \Big) \;\rd x\bigg|
\\
\leq \;& C \sum_{i = -\infty}^{\infty} \bigg( \int_{\R} \Big| \varphi \star \Big( (\nu/\eta) f \eta\,\varphi_i \Big) \Big|^2 \;\rd x \bigg)^{\frac{1}{2}},
\end{aligned}
\end{equation*}
where in the last line we use that each integrand is supported in $[-2 + i ,2 + i]$.

From the smoothness of $\varphi$, its Fourier transform can be bounded by
\begin{equation*}
\begin{aligned}
\hat \varphi(\xi) \leq \frac{C}{\sqrt{1 + 4 \pi^2 \xi^2}} = C \hat K(\xi).
\end{aligned}
\end{equation*}
Hence we further have from Lemma~\ref{lem:stability_under_tensorization},
\begin{equation*}
\begin{aligned}
\bigg| \int_{\R} \nu f \;\rd x\bigg|
\leq \;& C \sum_{i = -\infty}^{\infty} \bigg( \int_{\R} \Big| K \star \Big(  (\nu/\eta) f \eta\,\varphi_i \Big) \Big|^2 \;\rd x \bigg)^{\frac{1}{2}},
\\
\leq \;& C \sum_{i = -\infty}^{\infty} \|(\nu/\eta)\,\varphi_i \|_{W^{1,\infty}} \bigg( \int_{\R} \big| K \star (f \eta) \big|^2 \;\rd x \bigg)^{\frac{1}{2}}
\\
\leq \;& C \bigg( \sum_{i = -\infty}^{\infty} \|\varphi_i /\eta \|_{W^{1,\infty}} \bigg) \|\nu\|_{W^{1,\infty}} \|f \eta\|_{H^{-1}},
\end{aligned}
\end{equation*}
where the constant $C$ is some universal constant which may change line by line.

Since each $\varphi_i$ is a translation of $\varphi$ and has support in $[-1+i,1+i]$, it is easy the check the uniform bound
\begin{equation*}
\begin{aligned}
\sum_{i = -\infty}^{\infty} \|\varphi_i /\eta \|_{W^{1,\infty}} \leq C (1 + \alpha) \sum_{i = -\infty}^{\infty} \exp(-\alpha |i|) < \infty,
\end{aligned}
\end{equation*}
where the constant only depends on the particular choice of $\varphi$, which concludes the proof.
\end{proof}
This lemma also admits the following tensorization. 
\begin{lem} \label{lem:firing_rate_difference_tensorized}
For $f \in \mathcal{M}(\R^k)\cap H^{-1 \otimes k}_\eta$,
\begin{equation*}
\begin{aligned}
\;& \int_{\R^{k-1}} \bigg( \int_{\R} K^{\otimes k} \star \big( (\nu_m/\eta_m) f \eta^{\otimes k} \big)(t,z) \;\rd z_m \bigg)^2 \prod_{n \neq m} \;\rd z_n
\\
\leq \;& C (\alpha)^2 \|\nu\|_{W^{1,\infty}}^2 \|f\|_{H^{-1 \otimes k}_\eta}^2,
\end{aligned}
\end{equation*}
where we recall the notations $\nu_m=\nu(z_m)$ and $\eta_m=\eta(z_m)$.
\end{lem}
\begin{proof}
Without any loss of generality, we may assume $m = k$ and define
\begin{equation*}
\begin{aligned}
g(z) = \;& \Big[ K^{\otimes (k-1)} \star_{1,\dots,(k-1)} (f \eta^{\otimes (k-1)}) \Big] (z)
\\
= \;& \int_{\R^{k-1}} {\textstyle \prod_{n=1}^{k-1} K(u_n - z_n)} f(u_1,\dots,u_{k-1},z_{k}) {\textstyle \prod_{n=1}^{k-1} \eta(u_n) \; \rd u_n}.
\end{aligned}
\end{equation*}
Then, from the previous Lemma,
\begin{equation*}
\begin{aligned}
\;& \int_{\R^{k-1}} \bigg( \int_{\R} K^{\otimes k} \star \big( (\nu_k/\eta_k) f \eta^{\otimes k} \big)(t,z) \;\rd z_k \bigg)^2 \prod_{n = 1}^{k-1} \;\rd z_n
\\
= \;& \int_{\R^{k-1}} \bigg( \int_{\R} \nu(z_k)\, g(t,z)\,\rd z_k \bigg)^2 \prod_{n = 1}^{k-1} \;\rd z_n
\\
\leq \;& \int_{\R^{k-1}} C (\alpha)^2 \|\nu\|_{W^{1,\infty}}^2 \int_{\R} \bigg( \big[ K \star_k (g \eta_k) \big](t,z_1,\dots,z_k) \bigg)^2 \;\rd z_k \prod_{n = 1}^{k-1} \;\rd z_n
\\
= \;& C (\alpha)^2 \|\nu\|_{W^{1,\infty}}^2 \int_{\R^k} \bigg( \big[ K^{\otimes k} \star (f \eta^{\otimes k}) \big] (t,z_1,\dots,z_k) \bigg)^2 \; \prod_{n = 1}^{k} \;\rd z_n,
\end{aligned}
\end{equation*}
which concludes.
\end{proof}
\section{The limiting observables from Vlasov equation} \label{sec:observables}
This section is centered on the limiting observables $\tau_\infty(T,w,f)$, $T \in \mathcal{T}$. 
We first show that Definition~\ref{defi:limiting_observables} is still correct  when the kernel and extended density are merely $w \in \mathcal{W}$ and $f \in L^\infty([0,t_*] \times [0,1]; \mathcal{M}_+(\R))$.   We also prove Proposition~\ref{prop:independent_compactness}, which shows the compactness can be attained not only at the level of weak-* topology of each limiting observable $\tau_\infty$, $T \in \mathcal{T}$, but also directly at the level of $w$ and $f$.

Contrary to the rest of the paper, this section owes much to the technical  framework developed in \cite{JaPoSo:21}, that it extends to our setting.
\subsection{Revisiting the definition of limiting observables}
A motivation behind introducing the Banach space $\mathcal{W}$ in its current form is due to its ability to operate as a $L^p \to L^p$ mapping.
\begin{lem} \label{lem:L_p_operator} Consider the following bounded linear operator
\begin{equation*}
\begin{aligned}
\mathcal{W} \times C([0,1]; B) \to \;& L^\infty( [0,1]; B)
\\
(w, \phi) \mapsto \;& \int_{[0,1]} \phi(\cdot, \zeta) w (\cdot, \rd \zeta)
\end{aligned}
\end{equation*}
where $B$ stands for any Banach space such as $L^p(\R)$. Then this operator can be uniquely extended to $\mathcal{W} \times L^\infty([0,1]; B) \to L^\infty( [0,1]; B)$ with
\begin{equation*}
\begin{aligned}
\left\| \int_{[0,1]} \phi(\cdot, \zeta) w (\cdot, \rd \zeta) \right\|_{L^p([0,1]; B)} \leq \|w\|_{\mathcal{W}} \|\phi\|_{L^p([0,1]; B)}.
\end{aligned}
\end{equation*}
\end{lem}

\begin{proof}
The cases for $p = 1$ and $p = \infty$ can be checked through a careful but straightforward density argument, for which we refer to Lemma~3.8 in \cite{JaPoSo:21}. 
Extending the result to $1 < p < \infty$ is an application of textbook result of interpolation between Banach spaces, which can be found in \cite{BeLo:76} for instance.
\end{proof}
The integrals appearing in Definition~\ref{defi:limiting_observables} can then be made rigorous by sequentially consider the integrations as operations $L^p \to L^p$. To assist such argument, we follow again \cite{JaPoSo:21} and introduce the following countable algebra, which, as we see later, contains all necessary information to reproduce the limiting observables $\tau_\infty(T,w,f)$, $T \in \mathcal{T}$.
\begin{defi}[A countable algebra] \label{defi:tree_algebra}
We denote by $\mathscr{T}$ the countable algebra of transforms over spaces of arbitrarily large dimensions which is built as follows: For each transform $F \in \mathscr{T}$ there exists $k \in \N$ (called the rank of $F$) so that $F$ maps each couple $(w,f)$ into a signed measure $F(w,f) \in L^\infty([0,1]; \mathcal{M}(\R^k))$.
The full algebra $\mathscr{T}$ is obtained in a recursive way according to the following three rules:
\begin{itemize}
\item[(i)] (Seed). The elementary $1$-rank transform $F_0: (w,f) \mapsto f$ belongs to the algebra $\mathscr{T}$.

\item[(ii)] (Graft). Let $F_1 \in \mathscr{T}$ and $F_2 \in \mathscr{T}$ be $k_1$ rank and $k_2$ rank transforms respectively. Then, the following $(k_1 + k_2)$-rank transform $(F_1 \otimes F_2)$ also belongs to $\mathscr{T}$:
\begin{equation*}
\begin{aligned}
\;& (F_1 \otimes F_2)(w,f):
\\
\;& (\xi,z_1,\dots,z_{k_1 + k_2}) \mapsto F_1(w,f)(\xi,z_1,\dots,z_{k_1}) F_2(w,f)(\xi, z_{k_1+1},\dots,z_{k_1 + k_2}).
\end{aligned}
\end{equation*}

\item[(iii)] (Grow). Let $F \in \mathscr{T}$ be a $k$-rank transform. Then, the following $k$-rank transform $F^*$ also belongs to $\mathscr{T}$:
\begin{equation*}
\begin{aligned}
\;& F^*(w,f): 
\\
\;& (\xi, z_1,\dots,z_k) \mapsto \int_{[0,1]} F(w,f) (\zeta, z_1,\dots,z_k) w(\xi,\rd \zeta).
\end{aligned}
\end{equation*}

\end{itemize}

\end{defi}


The following lemma shows that the transform of the countable algebra $\mathscr{T}$ are well-defined on $\mathcal{W}$.
\begin{lem} \label{lem:tree_algebra}
Consider any kernel $w \in \mathcal{W}$ and extended density $f \in L^\infty([0,1]; H^{-1}_\eta \cap \mathcal{M}_+(\R))$.
Then for each $F \in \mathscr{T}$, the signed measure $F(w,f)$ is well-defined and belongs to $L^\infty([0,1]; H^{-1 \otimes k}_\eta \cap \mathcal{M}_+(\R^k))$ for some $k \in \N$. Moreover, as $n \to \infty$,
\begin{equation*}
\begin{aligned}
F(w^{(n)}, f^{(n)}) \to F(w, f) \quad \text{ in } \quad L^2([0,1]; H^{-1 \otimes k}_\eta)
\end{aligned}
\end{equation*}
for any fixed $F \in \mathscr{T}$, any sequence $\{f^{(n)}\}_{n=1}^\infty$ uniformly bounded in $L^\infty([0,1]; H^{-1}_\eta \cap \mathcal{M}_+(\R))$, and any sequence $\{w^{(n)}\}_{n=1}^\infty$ uniformly bounded in $\mathcal{W}$, satisfying
\begin{equation*}
\begin{aligned}
f^{(n)} \to \;& f  && \text{ in } \quad L^\infty([0,1]; H^{-1}_\eta(\R)),
\\
w^{(n)}(\xi,\zeta) \to \;& w(\xi,\zeta) && \text{ in } \quad L^2_\xi H^{-1}_\zeta \cap L^2_\zeta H^{-1}_\xi.
\end{aligned}
\end{equation*}
\end{lem}
We note that since $\zeta\in [0,\ 1]$, we have that $L^2_\xi H^{-1}_\zeta \subset L^\infty_\xi \mathcal{M}_\zeta$ with compact embedding.

\begin{proof}
We use an induction argument based on the recursive rules in Definition~\ref{defi:tree_algebra}.
\begin{itemize}
\item[(i)] The seed element $F_0(w,f) = f$ is well-defined and belongs to $L^\infty([0,1]; H^{-1}_\eta \cap \mathcal{M}_+(\R))$.

\item[(ii)] Consider two elements $F_1(w,f)$ and $F_2(w,f)$ that are well-defined and satisfying
\begin{equation*}
\begin{aligned}
F_i(w,f) \in L^\infty([0,1]; H^{-1 \otimes k_i}_\eta \cap \mathcal{M}(\R^{k_i})), \quad i = 1,2.
\end{aligned}
\end{equation*}
Because both norms are stable under tensorization, for the combined element we have 
\begin{equation*}
\begin{aligned}
\|(F_1 \otimes F_2)(w,f)\|_{L^\infty([0,1]; B_1 \otimes B_2)} \leq \|F_1(w,f)\|_{L^\infty([0,1]; B_1)} \|F_2(w,f)\|_{L^\infty([0,1]; B_2)},
\end{aligned}
\end{equation*}
where we may choose either $B_1 = H^{-1 \otimes k_1}_\eta$, $B_2 = H^{-1 \otimes k_2}_\eta$, $B_1 \otimes B_2 = H^{-1 \otimes (k_1+k_2)}_\eta$ or $B_1 = \mathcal{M}(\R^{k_1})$, $B_2 = \mathcal{M}(\R^{k_2})$, $B_1 \otimes B_2 = \mathcal{M}(\R^{k_1+k_2})$.
Hence,
\begin{equation*}
\begin{aligned}
(F_1 \otimes F_2)(w,f) \in L^\infty([0,1]; H^{-1 \otimes (k_1+k_2)}_\eta \cap \mathcal{M}(\R^{k_1+k_2})).
\end{aligned}
\end{equation*}
\item[(iii)] Consider an element $F(w,f)$ that is well-defined and satisfies
\begin{equation*}
\begin{aligned}
F(w,f) \in L^\infty([0,1]; H^{-1 \otimes k}_\eta \cap \mathcal{M}(\R^{k})).
\end{aligned}
\end{equation*}
Applying Lemma~\ref{lem:L_p_operator} with $p = \infty$ with either $B = H^{-1 \otimes k}_\eta$ or $B = \mathcal{M}( \mathcal{\R}^k )$, for the grow element we have
\begin{equation*}
\begin{aligned}
\|F^*(w,f)\|_{L^\infty([0,1]; B)} \leq \|w\|_{\mathcal{W}} \|F(w,f)\|_{L^\infty([0,1]; B)}.
\end{aligned}
\end{equation*}
Hence,
\begin{equation*}
\begin{aligned}
F^*(w,f) \in L^\infty([0,1]; H^{-1 \otimes k}_\eta \cap \mathcal{M}_+(\R^k)).
\end{aligned}
\end{equation*}

\end{itemize}

Since $\mathscr{T}$ is generated by the three rules in Definition~\ref{defi:tree_algebra}, the above argument shows that any $F(w,f)$, $F \in \mathscr{T}$ is well-defined.

We can use a similar argument to prove the convergence $F(w^{(n)}, f^{(n)}) \to F(w, f)$ for any fixed $F \in \mathscr{T}$.
\begin{itemize}
\item[(i)] For the seed sequence, $F_0(w^{(n)}, f^{(n)}) = f^{(n)} \to f$ in $L^\infty([0,1]; H^{-1}_\eta(\R))$, hence the convergence also holds in $L^2([0,1]; H^{-1}_\eta(\R))$.

\item[(ii)]
Consider the two sequences $F_1(w^{(n)}, f^{(n)})$ and $F_2(w^{(n)}, f^{(n)})$ satisfying
\begin{equation*}
\begin{aligned}
F_i(w^{(n)}, f^{(n)}) \to F_i(w, f) \quad \text{ in } \quad L^2([0,1]; H^{-1 \otimes k_i}_\eta) , \quad i = 1,2.
\end{aligned}
\end{equation*}
Then by introducing the intermediary element $F_1(w^{(n)}, f^{(n)}) \otimes_z F_2(w, f)$ and by applying the triangular inequality, we have that
\begin{equation*}
\begin{aligned}
& \| (F_1 \otimes F_2)(w^{(n)}, f^{(n)}) - (F_1 \otimes F_2)(w, f)\|_{L^2([0,1]; H^{-1 \otimes (k_1+k_2)}_\eta)}\\
& \qquad\leq \| F_1(w^{(n)}, f^{(n)}) \|_{L^\infty([0,1]; H^{-1 \otimes k_1}_\eta)} \| F_2(w^{(n)}, f^{(n)}) - F_2(w, f) \|_{L^2([0,1]; H^{-1 \otimes k_2}_\eta)}
\\
& \qquad\qquad+ \| F_1(w^{(n)}, f^{(n)}) - F_1(w, f) \|_{L^2([0,1]; H^{-1 \otimes k_1}_\eta)} \| F_2(w, f) \|_{L^\infty([0,1]; H^{-1 \otimes k_2}_\eta)}.
\end{aligned}
\end{equation*}
As $n \to \infty$, we immediately have that
\begin{equation*}
\begin{aligned}
(F_1 \otimes F_2)(w^{(n)}, f^{(n)}) \to (F_1 \otimes F_2) (w, f) \quad \text{ in } \quad L^2([0,1]; H^{-1 \otimes (k_1+k_2)}_\eta).
\end{aligned}
\end{equation*}

\item[(iii)] (Grow). Consider a sequence $F(w^{(n)}, f^{(n)})$ satisfying
\begin{equation*}
\begin{aligned}
F(w^{(n)}, f^{(n)}) \to F(w, f) \quad \text{ in } \quad L^2([0,1]; H^{-1 \otimes k}_\eta).
\end{aligned}
\end{equation*}
The difference between the grow sequences is given by
\begin{equation*}
\begin{aligned}
& \| F^*(w^{(n)}, f^{(n)}) - F^*(w, f) \|_{L^2([0,1]; H^{-1 \otimes k}_\eta)}
\\
&\quad= \bigg\| \int_{[0,1]} F(w^{(n)}, f^{(n)})(\zeta,\cdot) w^{(n)}(\xi,\rd \zeta) - \int_{[0,1]} F(w, f)(\zeta,\cdot) w(\xi,\rd \zeta) \bigg\|_{L^2([0,1]; H^{-1 \otimes k}_\eta)}.
\end{aligned}
\end{equation*}
Introduce any $\phi_\varepsilon \in H^1([0,1]; H^{-1 \otimes k}_\eta)$ approximating $F(w, f)$ in $L^2([0,1]; H^{-1 \otimes k}_\eta)$ and the intermediary elements
\begin{equation*}
\begin{aligned}
\int_{[0,1]} \phi_\varepsilon(\zeta,\cdot) w^{(n)}(\xi,\rd \zeta), \quad \int_{[0,1]} \phi_\varepsilon(\zeta,\cdot) w(\xi,\rd \zeta),
\end{aligned}
\end{equation*}
then apply the triangular inequality and Lemma~\ref{lem:L_p_operator} with $p = 2$, $B = H^{-1 \otimes k}_\eta$, we have
\begin{equation*}
\begin{aligned}
& \| F^*(w^{(n)}, f^{(n)}) - F^*(w, f) \|_{L^2([0,1]; H^{-1 \otimes k}_\eta)}
\\
&\quad \leq  \|w^{(n)}\|_{\mathcal{W}} \|F(w^{(n)}, f^{(n)}) - \phi_\epsilon\|_{L^2([0,1]; H^{-1 \otimes k}_\eta)} + \|w\|_{\mathcal{W}} \|F(w, f) - \phi_\epsilon\|_{L^2([0,1]; H^{-1 \otimes k}_\eta)}
\\
& \qquad+ \|w^{(n)} - w\|_{L^2 H^{-1}} \|\phi_\epsilon\|_{H^1([0,1]; H^{-1 \otimes k}_\eta)}.
\end{aligned}
\end{equation*}
Letting $n \to \infty$ and $\varepsilon\to 0$, we conclude that
\begin{equation*}
\begin{aligned}
F^*(w^{(n)}, f^{(n)}) \to F^*(w, f) \quad \text{ in } \quad L^2([0,1]; H^{-1 \otimes k}_\eta).
\end{aligned}
\end{equation*}
\end{itemize}
\end{proof}
The following lemma shows that it is possible to recover the limiting observables $\tau_\infty(T,w,f)$, $T \in \mathcal{T}$ from $F(w,f)$, $F \in \mathscr{T}$.
\begin{lem} \label{lem:algebra_to_tree}
For any tree $T \in \mathcal{T}$, there exists a transform $F \in \mathscr{T}$ such that
\begin{equation*}
\begin{aligned}
F(w,f)(\zeta, z_1,\dots, z_{|T|}) = \bigg( \int_{[0,1]^{|T| - 1}} w_T(\xi_1,\dots,\xi_{|T|}) {\textstyle \prod_{m=1}^{|T|} f(z_m,\xi_m)} \;\rd \xi_2 \dots \rd \xi_{|T|} \bigg) \bigg|_{\xi_1=\zeta},
\end{aligned}
\end{equation*}
where the variable $\xi_1$ corresponding to the root of $T$ is not integrated. As a consequence
\[
\tau_\infty(T,w,f) = \int_{[0,1]} F(w,f)(\zeta, z_1,\dots, z_{|T|}) \;\rd \zeta.
\]
\end{lem}
%

\begin{proof}[Proof of Lemma~\ref{lem:algebra_to_tree}] For the tree $T_1 \in \mathcal{T}$ with only one node, the corresponding transform in $\mathscr{T}$ is the seed element $F_0$. It is easy to verify that
\begin{equation*}
\begin{aligned}
F_0(w,f)(\zeta, z_1) = f(z_1,\zeta), \quad \tau_\infty(T_0,w,f)(z_1) = \int_{[0,1]} f(z_1,\zeta)\;\rd \zeta.
\end{aligned}
\end{equation*}
For any tree $T \in \mathcal{T}$ with more than one node. Let $ i_1,\dots,i_k \in \{2,\dots, |T|\}$ be all the nodes that are directly connected to the root $1$, and let  $T_1,\dots,T_k$ be the subtrees of $T$ taking $i_1,\dots,i_k$ as their roots.
Suppose by induction that we have found corresponding transforms $F_1,\dots,F_k$ for $T_1,\dots,T_k$, then
\begin{equation*}
\begin{aligned}
&  \int_{[0,1]^{|T| - 1}} w_T(\xi_1,\dots,\xi_{|T|}) {\textstyle \prod_{m=1}^{|T|} f(z_m,\xi_m)} \;\rd \xi_2,\dots, \rd \xi_{|T|}
\\
&\quad = f(z_1,\xi_1) \prod_{l=1}^k \bigg( \int_{[0,1]^{|T_l|}} w(\xi_1, \xi_{i_l}) \prod_{(j,j') \in \mathcal{E}(T_l)} w(\xi_j,\xi_{j'}) \prod_{m \in T_l} f(z_m,\xi_m) \rd \xi_m \bigg)
\\
&\quad = f(z_1,\xi_1) \prod_{l=1}^k \int_{[0,1]} F_l(w,f)(\xi_{i_l}, z_{i_l}, \dots) w(\xi_1, \rd \xi_{i_l})=F(w,f)(\xi_1, z_1,\dots, z_{|T|}).
\end{aligned}
\end{equation*}
Up to an index permutation (so that if $i \in T_{l}, i' \in T_{l'}$, $l < l'$, then $i < i'$), it can be reformulated into the more straightforward form
\begin{equation*}
\begin{aligned}
F(w,f) = \bigg[ F_0 \otimes \bigotimes_{l=1}^k (F_l)^* \bigg] (w,f),
\end{aligned}
\end{equation*}
showing $F$ is obtained by making each $F_l$ grow (by rule (iii)) with depth $1$, then grafting (by rule (ii)) together all of them with another seed element $F_0$ (by rule (i)).
\end{proof}
\subsection{Compactness of the limiting observables}
We now turn to the proof of Proposition~\ref{prop:independent_compactness}. 
\begin{proof} [Proof of Proposition~\ref{prop:independent_compactness}]
To prove \eqref{eqn:approximate_independence}, let us define $S^{m}(N) \defeq \{1,\dots, N\}^{m}$ and
\begin{equation*}
\begin{aligned}
S^{m}_{\textnormal{diag}}(N) \defeq \big\{ (i_1,\dots,i_m) \in \{1,\dots, N\}^{m} : \exists j \neq k \text{ s.t. } i_j = i_k \big\}.
\end{aligned}
\end{equation*}
Recall that
\begin{equation*}
\begin{aligned}
\tilde w_N(\xi,\zeta) = \;& \sum_{i,j=1}^N N w_{i,j;N} \mathbbm{1}_{[\frac{i-1}{N},\frac{i}{N})}(\xi) \mathbbm{1}_{[\frac{j-1}{N},\frac{j}{N})}(\zeta),
\\
\tilde f_N(x,\xi) = \;& \sum_{i=1}^N f_{N}^{i}(x) \mathbbm{1}_{[\frac{i-1}{N},\frac{i}{N})}(\xi).
\end{aligned}
\end{equation*}
Since we have independence, it is straightforward that
\begin{equation*}
\begin{aligned}
\tau_\infty(T,\tilde w_N, \tilde f_N)(\rd z) = \int_{[0,1]^{|T|}} {\textstyle \prod_{(l,l') \in \mathcal{E}(T)} \tilde w_N (\xi_l, \xi_{l'}) \prod_{m=1}^{|T|} \tilde f_N(t,\rd z_m,\xi_m)} \;\rd \xi_1,\dots, \rd \xi_{|T|}
\\
= \frac{1}{N} \sum_{(i_1,\dots,i_{|T|}) \in S^{|T|}(N) } {\textstyle \prod_{(l,l') \in \mathcal{E}(T)} w_{i_l, i_{l'}; N} \prod_{m=1}^{|T|} f_{N}^{i_m}(\rd z_m)}.
\end{aligned}
\end{equation*}
On the other hand, again from independence,
\begin{equation*}
\begin{aligned}
\tau_N(T,w_N, f_N)(\rd z)
= \frac{1}{N} \sum_{(i_1,\dots,i_{|T|}) \in S^{|T|}(N) \setminus S^{|T|}_{\textnormal{diag}}(N)} {\textstyle \prod_{(l,l') \in \mathcal{E}(T)} w_{i_l, i_{l'}; N} \prod_{m=1}^{|T|} f_{N}^{i_m}(\rd z_m)},
\end{aligned}
\end{equation*}
where the terms involving repeating index are excluded from the summation, contrary to the case of $\tau_\infty$.

Therefore the difference is controlled by
\begin{equation*}
\begin{aligned}
\;& \tau_\infty(T,\tilde w_N, \tilde f_N)(\rd z) - \tau_N(T,w_N, f_N)(\rd z)
\\
= \;& \frac{1}{N} \sum_{(i_1,\dots,i_{|T|}) \in S^{|T|}_{\textnormal{diag}}(N)} {\textstyle \prod_{(l,l') \in \mathcal{E}(T)} w_{i_l, i_{l'}; N} \prod_{m=1}^{|T|} f_{N}^{i_m}(\rd z_m)},
\end{aligned}
\end{equation*}
whose (weighted) total variation norm is bounded by
\begin{equation*}
\begin{aligned}
& \int_{\R^{|T|}} { \textstyle \exp\big( a \sum_{m=1}^{|T|} |z_m|\big)} |\tau_\infty(T,\tilde w_N, \tilde f_N)(\rd z) - \tau_N(T,w_N, f_N)(\rd z)|
\\
&\ =  \int_{\R^{|T|}} { \textstyle \exp\big( a \sum_{m=1}^{|T|} |z_m|\big)} \bigg| \frac{1}{N} \sum_{(i_1,\dots,i_{|T|}) \in S^{|T|}_{\textnormal{diag}}(N)} {\textstyle \prod_{(l,l') \in \mathcal{E}(T)} w_{i_l, i_{l'}; N} \prod_{m=1}^{|T|} f_{N}^{i_m}(\rd z_m)} \bigg|
\\
&\ \leq  \frac{1}{N} \sum_{(i_1,\dots,i_{|T|}) \in S^{|T|}_{\textnormal{diag}}(N)} {\textstyle \prod_{(l,l') \in \mathcal{E}(T)} |w_{i_l, i_{l'}; N}}| \int_{\R^{|T|}} {\textstyle \prod_{m=1}^{|T|} \exp(a |z_m|) f_{N}^{i_m}(\rd z_m)}.
\end{aligned}
\end{equation*}
When $|T| = 1$, this term is zero as $S^{|T|}_{\textnormal{diag}}(N) = \varnothing$, while for $|T| \geq 2$ we use the following lemma.
\begin{lem} \label{lem:diagonal_estimate}
The following bound holds
\begin{equation*}
\begin{aligned}
\;& \frac{1}{N} \sum_{(i_1,\dots,i_{|T|}) \in S^{|T|}_{\textnormal{diag}}(N)} {\textstyle \prod_{(l,l') \in \mathcal{E}(T)} |w_{i_l, i_{l'}; N}}|
\\
\leq \;&  \max_{1\leq i,j \leq N} |w_{i,j;N}| \max \Big( \textstyle \max_{i} \sum_{j} |w_{i,j;N}| , \max_{j} \sum_{i} |w_{i,j;N}| \Big)^{|T|-2} |T|^2.
\end{aligned}
\end{equation*}
\end{lem}
Once we prove Lemma~\ref{lem:diagonal_estimate}, we immediately obtain \eqref{eqn:approximate_independence}, 
\begin{equation*}
\begin{aligned}
\;& \int_{\R^{|T|}} { \textstyle \exp\big( a \sum_{m=1}^{|T|} |z_m|\big)} |\tau_\infty(T,\tilde w_N, \tilde f_N)(\rd z) - \tau_N(T,w_N, f_N)(\rd z)|
\\
\leq \;& \max_{1\leq i,j \leq N} |w_{i,j;N}|\,\max \Big(\max_{i} \sum_{j} |w_{i,j;N}|,\; \max_{j} \sum_{i} |w_{i,j;N}| \Big)^{|T|-2}\, |T|^2\, M_a^{|T|}.
\end{aligned} 
\end{equation*}
\begin{proof}[Proof of Lemma~\ref{lem:diagonal_estimate}]
Let us consider
\begin{equation*}
\begin{aligned}
\;& \sum_{(i_1,\dots,i_{|T|}) \in S^{|T|}(N)} \mathbbm{1}_{\{i_{m} = i_{m'} = i\}} {\textstyle \prod_{(l,l') \in \mathcal{E}(T)} |w_{i_l, i_{l'}; N}}|
\end{aligned}
\end{equation*}
for any $1 \leq m,m' \leq |T|$ and $1 \leq i \leq N$. We introduce the path $P$ which is the set of indices $n$ on the unique path connecting $m$ and $m'$. We can immediately remove from the sum the indices not in $P$ as before, 
\begin{equation*}
\begin{aligned}
&\sum_{(i_1,\dots,i_{|T|}) \in S^{|T|}(N)} \mathbbm{1}_{\{i_{m} = i_{m'} = i\}} {\textstyle \prod_{(l,l') \in \mathcal{E}(T)} |w_{i_l, i_{l'}; N}}|
\\
&\ \leq \max \Big( \textstyle \max_{i} \sum_{j} |w_{i,j;N}|, \max_{j} \sum_{i} |w_{i,j;N}| \Big)^{|T| - |P|}\\
&\qquad\qquad\qquad\sum_{(i_{n_1},\ldots,i_{n_{|P|}}) \in S^{|P|}(N)} \mathbbm{1}_{\{i_{m} = i_{m'} = i\}} {\textstyle \prod_{(l,l') \in \mathcal{E}(P)} |w_{i_l, i_{l'}; N}}|,
\end{aligned}
\end{equation*}
where we denote $P=\{n_1,\ldots,n_{|P|}\}$ with $n_1=m$ and $n_{|P|}=m'$.

The path $P$ connecting $m$ and $m'$ naturally goes up in the tree first (to reach the parent vertex that is shared by $m$ and $m'$) and then down. Denote by $k$ the number of indices for which the path goes up (with possibly $k=1$ if $m$ is a parent of $m'$) and write
\begin{equation*}
\begin{aligned}
& \sum_{(i_{n_1},\dots,i_{n_{|P|}}) \in S^{|P|}(N)} \mathbbm{1}_{\{i_{m} = i_{m'} = i\}} {\textstyle \prod_{(l,l') \in \mathcal{E}(P)} |w_{i_l, i_{l'}; N}}|\\
  &\ = \sum_{1 \leq j_1, \dots, j_{|P|} \leq N} \mathbbm{1}_{\{j_1 = j_{|P|} = i\}}
    \prod_{n=1}^{k-1} |w_{j_{n+1}, j_n; N}| \prod_{n=k}^{|P| - 1} |w_{j_n, j_{n+1}; N}|
\\
&\ = \max_{1\leq i,j \leq N} |w_{i,j;N}| \max \Big( \textstyle \max_{i} \sum_{j} |w_{i,j;N}| , \max_{j} \sum_{i} |w_{i,j;N}| \Big)^{|P| - 2}.
\end{aligned}
\end{equation*}
Therefore
\begin{equation*}
\begin{aligned}
& \sum_{(i_1,\dots,i_{|T|}) \in S^{m}(N)} \mathbbm{1}\{i_{m} = i_{m'} = i\} {\textstyle \prod_{(l,l') \in \mathcal{E}(T)} |w_{i_l, i_{l'}; N}}|
\\
&\ \leq  \max_{1\leq i,j \leq N} |w_{i,j;N}| \max \Big( \textstyle \max_{i} \sum_{j} |w_{i,j;N}| , \max_{j} \sum_{i} |w_{i,j;N}| \Big)^{|T| - 2}.
\end{aligned}
\end{equation*}
As a consequence,
\begin{equation*}
\begin{aligned}
& \frac{1}{N} \sum_{(i_1,\dots,i_{|T|}) \in S^{m}_{\textnormal{diag}}(N)} {\textstyle \prod_{(l,l') \in \mathcal{E}(T)} |w_{i_l, i_{l'}; N}}|\\
&\ \leq \frac{1}{N} \sum_{i = 1}^N \sum_{1 \leq m,m' \leq |T|} \sum_{(i_1,\dots,i_{|T|}) \in S^{m}(N)} \mathbbm{1}_{\{i_{m} = i_{m'} = i\}} {\textstyle \prod_{(l,l') \in \mathcal{E}(T)} |w_{i_l, i_{l'}; N}}|
\\
&\ \leq  \max_{1\leq i,j \leq N} |w_{i,j;N}| \max \Big( \textstyle \max_{i} \sum_{j} |w_{i,j;N}| , \max_{j} \sum_{i} |w_{i,j;N}| \Big)^{|T|-2} |T|^2,
\end{aligned}
\end{equation*}
which concludes the proof.
\end{proof}

It remains to prove \eqref{eqn:assumption_4_compact}, for which we first invoke Corollary 4.9 in \cite{JaPoSo:21}.

\begin{lem} [Corollary 4.9 in \cite{JaPoSo:21}] \label{lem:compactness_rearrangement} Consider any sequence $g_n$ in $L^\infty([0,1])$. Then, there exists $\Phi:[0,1] \to [0,1]$, a.e. injective, measure preserving, such that the following estimate is verified
\begin{equation*}
\begin{aligned}
\int_{[0,1]} |(g_n \circ \Phi)(\xi) - (g_n \circ \Phi)(\xi + h)| \;\rd \xi \leq 2^n \|g_n\|_{L^\infty} 2^{-C \sqrt{\log \frac{1}{|h|}}}
\end{aligned}
\end{equation*}
for any $n \in \N$, $0 < |h| < 1$ and some universal constant $C$.

\end{lem}

This lemma tells us that, at the cost of a measure-preserving re-arrangement, a minimum regularity of $L^\infty$ functions on $[0,1]$ can be obtained. 
In order to apply Lemma~\ref{lem:compactness_rearrangement}, we need to first check the stability of the algebra $F(w,f)$, $F \in \mathscr{T}$ under measure preserving re-arrangements.
\begin{lem} \label{lem:stable_rearrangement}
Consider any $w \in \mathcal{W}$ and $f \in L^\infty([0,1]; \mathcal{M}_+(\R))$ and any a.e. injective, measure-preserving $\Phi: [0,1] \to [0,1]$. Define the push forward kernel and measure
\begin{equation*}
\begin{aligned}
w_\#(\xi, \rd \zeta) \defeq \Phi^{-1}_\# w(\Phi(\xi), \cdot)(\rd \zeta), \quad f_\#(\xi, \rd z)  \defeq f(\Phi(\xi), \rd z),
\end{aligned}
\end{equation*}
where $\Phi^{-1}$ is any a.e. defined left inverse of $\Phi$.
Then the algebra $F(w,f)$, $F \in \mathscr{T}$ is stable under $\Phi$ in the sense that
\begin{equation*}
\begin{aligned}
F(w_\#, f_\#)(\xi, \rd z_1, \dots, \rd z_k) = F(w, f)(\Phi(\xi), \rd z_1, \dots, \rd z_k)
\end{aligned}
\end{equation*}
for any transform $F \in \mathscr{T}$ and for a.e. $\xi \in [0,1]$. Moreover, $\tau_\infty(T,w_\#,f_\#) = \tau_\infty(T,w,f)$ for any $T \in \mathcal{T}$.

\end{lem}

\begin{proof}
The proof is again done by an induction argument based on the recursive rules defining $F(w,f)$, $F \in \mathscr{T}$.

\begin{itemize}
\item[(i)] For the seed element $F_0(w,f)$ the property is obvious.

\item[(ii)]
Consider two elements $F_1(w,f)$ and $F_2(w,f)$ stable under $\Phi$. Then the grafted element satisfies
\begin{equation*}
\begin{aligned}
\;& (F_1 \otimes F_2)(w_\#, f_\#)(\xi, \rd z_1, \dots, \rd z_{k_1+k_2})
\\
= \;& F_1(w_\#, f_\#)(\xi, \rd z_1, \dots, \rd z_{k_1}) F_2(w_\#, f_\#)(\xi, \rd z_{k_1+1}, \dots, \rd z_{k_1+k_2})
\\
= \;& F_1(w, f)(\Phi(\xi), \rd z_1, \dots, \rd z_{k_1}) F_2(w, f)(\Phi(\xi), \rd z_{k_1+1}, \dots, \rd z_{k_1+k_2})
\\
= \;& F(w, f)(\Phi(\xi), \rd z_1, \dots, \rd z_{k_1+k_2}),
\end{aligned}
\end{equation*}
which is the stated stability under $\Phi$.

\item[(iii)] Consider an element $F(w,f)$ stable under $\Phi$. Then the grow element satisfies the stability property
\begin{equation*}
\begin{aligned}
F^*(w_\#, f_\#) (\xi, \rd z_1, \dots, \rd z_k) = \;& \int_{\zeta \in [0,1]} F(w_\#, f_\#) (\zeta, \rd z_1, \dots, \rd z_k) w_\#(\xi, \rd \zeta)
\\
= \;& \int_{\zeta \in [0,1]} F(w, f) (\Phi(\zeta), \rd z_1, \dots, \rd z_k) w_\#(\xi, \rd \zeta)
\\
= \;& \int_{\zeta \in [0,1]} F(w, f) (\zeta, \rd z_1, \dots, \rd z_k) w(\Phi(\xi), \rd \zeta)
\\
= \;& F^*(w, f) (\Phi(\xi), \rd z_1, \dots, \rd z_k).
\end{aligned}
\end{equation*}
%
\end{itemize}

Finally, for any $T \in \mathcal{T}$, take $F \in \mathscr{T}$ as claimed in Lemma~\ref{lem:algebra_to_tree}. Then
\begin{equation*}
\begin{aligned}
\tau_\infty (T,w_\#,f_\#)(\rd z_1, \dots, \rd z_{|T|}) = \;& \int_{\xi \in [0,1]} F(w_\#, f_\#) (\xi, \rd z_1, \dots, \rd z_{|T|}) \;\rd \xi
\\
= \;& \int_{\xi \in [0,1]} F(w, f) (\xi, \rd z_1, \dots, \rd z_{|T|}) \;\rd \xi
\\
= \;& \tau_\infty (T,w,f)(\rd z_1, \dots, \rd z_{|T|}),
\end{aligned}
\end{equation*}
which finishes the proof.
\end{proof}

The next step is to is to derive the compactness of the algebra $F(w,f)$, $F \in \mathscr{T}$ and identify the limit, which we summarize here.
\begin{lem} \label{lem:algebra_limit}
Under the assumptions of Proposition~\ref{prop:independent_compactness}, there exists measure-preserving maps $\Phi_N: [0,1] \to [0,1]$ for the sequence of $N \to \infty$ and $w \in \mathcal{W}$, $f \in L^\infty([0,1]; \mathcal{M}_+(\R))$, such that convergence in the following strong-weak-* sense holds: For all $F \in \mathscr{T}$ and all $\varphi \in C_c(\R^k)$, where $k$ is the rank of $F$,
\begin{equation} \label{eqn:algebra_limit_full}
\begin{aligned}
\lim_{N \to \infty} \;& \int_{z \in \R^k} \varphi(z_1,\dots,z_k) F(\tilde w_N, \tilde f_N)(\Phi_N(\xi), \rd z_1,\dots, \rd z_k)
\\
= \;& \int_{z \in \R^k} \varphi(z_1,\dots,z_k) F(w, f)(\xi, \rd z_1,\dots, \rd z_k)
\end{aligned}
\end{equation}
in any $L^p_\xi([0,1])$, $1 \leq p < \infty$.

\end{lem}

\begin{proof}

Since the algebra is countable, we may index the elements as $\mathscr{T} = \{F_m: m \in \N\}$. For each $m \in \N$, let $k_m$ be the rank of $F_m$ and let $\{\varphi_{m,l}\}_{l \in \N}$ be any countable dense set of $C_c(\R^{k_m})$. Define the functions
\begin{equation*}
\begin{aligned}
g_{m,l}^N(\xi) \defeq \int_{z \in \R^k} \varphi_{m,l}(z_1,\dots,z_{k_m}) F_m(\tilde w_N, \tilde f_N)(\xi, \rd z_1,\dots, \rd z_{k_m}), \quad \forall m,l,N.
\end{aligned}
\end{equation*}
It is straightforward that $\sup_N \|g_{m,l}\|_{L^\infty([0,1])}<\infty$ from the bounds on $F_m(\tilde w_N, \tilde f_N)$ in the space $L^\infty([0,1]; \mathcal{M}(\R^{k_m}))$ that follow from Lemma~\ref{lemboundtauN} and the identification provided by Lemma~\ref{lem:algebra_to_tree}.

Thus, by Lemma~\ref{lem:compactness_rearrangement}, there exists $\Phi_N: [0,1] \to [0,1]$ for the sequence $N \to \infty$, so that the re-arrangements
\begin{equation*}
\begin{aligned}
\tilde g_{m,l}^N(\xi) = (g_{m,l}^N \circ \Phi_N) (\xi)= \int_{z \in \R^k} \varphi_{m,l}(z_1,\dots,z_{k_m}) F_m(\tilde w_N, \tilde f_N)(\Phi_N(\xi), \rd z_1,\dots, \rd z_{k_m})
\end{aligned}
\end{equation*}
fulfill the estimates
\begin{equation*}
\begin{aligned}
\int_{[0,1]} |\tilde g_{m,l}^N(\xi) - \tilde g_{m,l}^N(\xi + h)| \;\rd \xi \leq C_{m,l} 2^{-C \sqrt{\log \frac{1}{|h|}}}, \quad \forall 0 < |h| < 1
\end{aligned}
\end{equation*}
for some universal constant $C>0$ and $C_{m,l}>0$ depending on the two indexes only.

By the Fr\'echet-Kolmogorov theorem and using a diagonal extraction there exists some subsequence of $N$ (which we still denote $N$ for simplicity) and for all $m,l \in \N$, there exists $\tilde g_{m,l} \in L^\infty([0,1])$ such that as $N \to \infty$,
\begin{equation*}
\begin{aligned}
\tilde g_{m,l}^N \to \tilde g_{m,l} \; \text{ in any } L^p([0,1]), \; 1 \leq p < \infty.
\end{aligned}
\end{equation*}
Let us define, for any $N$ in the subsequence, any $F \in \mathscr{T}$ and $\varphi \in C_c(\R^{k})$, where $k$ is the rank of $F$,
\begin{equation*}
\begin{aligned}
\tilde g_{F,\varphi}^N \defeq \;& \int_{z \in \R^k} \varphi(z_1,\dots,z_k) F(\tilde w_N, \tilde f_N)(\Phi_N(\xi), \rd z_1,\dots, \rd z_k)
\\
= \;& \int_{z \in \R^k} \varphi(z_1,\dots,z_k) F(\tilde w_{N;\#}, \tilde f_{N;\#})(\xi, \rd z_1,\dots, \rd z_k),
\end{aligned}
\end{equation*}
where we again apply the following notation for the re-arrangement
\begin{equation*}
\begin{aligned}
\tilde w_{N;\#}(\xi, \rd \zeta) \defeq \Phi^{-1}_\# \tilde w_N(\Phi(\xi), \cdot)(\rd \zeta), \quad \tilde f_{N;\#}(\xi, \rd z) \defeq \tilde f_{N}(\Phi(\xi), \rd z).
\end{aligned}
\end{equation*}

By a density argument of $C_c(\R^k)$, we conclude that for any $F \in \mathscr{T}$ and $\varphi \in C_c(\R^{k})$, there exists $\tilde g_{F,\varphi} \in L^\infty([0,1])$ such that as $N \to \infty$,
\begin{equation} \label{eqn:algebra_compactness}
\begin{aligned}
\tilde g_{F,\varphi}^N \to \tilde g_{F,\varphi} \; \text{ in any } L^p([0,1]), \; 1 \leq p < \infty.
\end{aligned}
\end{equation}
It remains to identify $w \in \mathcal{W}$ and $f \in L^\infty([0,1]; \mathcal{M}_+(\R))$ for the limit.
Recall that we have defined the kernel space $\mathcal{W}$ as
\begin{equation*}
\begin{aligned}
\mathcal{W} \defeq \{w \in \mathcal{M}([0,1]^2) : w(\xi, \rd \zeta) \in L^\infty_\xi([0,1], \mathcal{M}_\zeta[0,1]), \;  w(\rd \xi, \zeta) \in L^\infty_\zeta([0,1], \mathcal{M}_\xi[0,1]) \},
\end{aligned}
\end{equation*}
where $L^\infty_\xi([0,1], \mathcal{M}_\zeta[0,1])$ denotes the topological dual of $L^1_\xi([0,1], C_\zeta[0,1])$.

Hence there exists a subsequence (which we still index by $N$) and $w \in \mathcal{W}$, $f \in L^\infty([0,1]; \mathcal{M}_+(\R))$, such that
\begin{equation*}
\begin{aligned}
\tilde w_{N;\#} \overset{\ast}{\rightharpoonup} w, \quad \tilde f_{N;\#} \overset{\ast}{\rightharpoonup} f.
\end{aligned}
\end{equation*}
By passing to the limit we can immediately obtain the exponential moment bound
\begin{equation*}
\begin{aligned}
\esssup_{\xi \in [0,1]} \int_{\R} { \textstyle \exp\big( a |x|\big)} f (\xi, \rd x) \leq \;& M_a.
\end{aligned}
\end{equation*}
Let us define, for any $F \in \mathscr{T}$ and $\varphi \in C_c(\R^{k})$,
\begin{equation*}
\begin{aligned}
g_{F,\varphi}(\xi) \defeq \int_{z \in \R^k} \varphi(z_1,\dots,z_k) F(w, f)(\xi, \rd z_1,\dots, \rd z_k).
\end{aligned}
\end{equation*}
It is straightforward that $g_{F,\varphi} \in L^\infty([0,1])$ and \eqref{eqn:algebra_limit_full} can be simply restated as
\begin{equation} \label{eqn:algebra_limit}
\begin{aligned}
\tilde g_{F,\varphi} = g_{F,\varphi}.
\end{aligned}
\end{equation}
We apply another induction argument based on the recursive rules.
\begin{itemize}
\item[(i)] For the seed element $F_0(w,f) = f$, it is straightforward that for any $\psi \in C([0,1])$, $\phi \in C_c(\R)$,
\begin{equation*}
\begin{aligned}
\int_{[0,1]} \psi(\xi) g_{F_0,\varphi}(\xi) \;\rd \xi
= \;& \int_{[0,1]} \psi(\xi) \int_{z \in \R} \varphi(z) f(\xi, \rd z) \;\rd \xi
\\
\overset{\ast}{=} \;& \lim_{N \to \infty} \int_{[0,1]} \psi(\xi) \int_{z \in \R} \varphi(z) \tilde f_{N; \#}(\xi, \rd z) \;\rd \xi
\\
= \;& \int_{[0,1]} \psi(\xi) \tilde g_{F_0,\varphi}(\xi) \;\rd \xi,
\end{aligned}
\end{equation*}
where the equality $\overset{\ast}{=}$ is due to the weak-* convergence $\tilde f_{N;\#} \overset{\ast}{\rightharpoonup} f$. Hence, the identity \eqref{eqn:algebra_limit} holds for $F_0$.

\item[(ii)]
Consider two elements $F_1,F_2 \in \mathscr{T}$ satisfying \eqref{eqn:algebra_limit}. Then for any $\phi_1 \in C_c(\R^{k_1})$, $\phi_2 \in C_c(\R^{k_2})$,
\begin{equation*}
\begin{aligned}
& g_{(F_1 \otimes F_2),(\varphi_1 \otimes \varphi_2)}(\xi)
\\
&\quad=  \int_{z \in \R^{k_1+k_2}} (\varphi_1 \otimes \varphi_2)(z_1,\dots,z_{k_1+k_2}) (F_1 \otimes F_2)(w, f)(\xi, \rd z_1,\dots, \rd z_{k_1+k_2})
\\
&\quad= g_{F_1,\varphi_1}(\xi) g_{F_2,\varphi_2}(\xi),
\end{aligned}
\end{equation*}
hence $g_{(F_1 \otimes F_2),(\varphi_1 \otimes \varphi_2)} = g_{F_1,\varphi_1} g_{F_2,\varphi_2}$.

By a similar argument, $\tilde g_{(F_1 \otimes F_2),(\varphi_1 \otimes \varphi_2)}^N = \tilde g_{F_1,\varphi_1}^N \tilde g_{F_2,\varphi_2}^N$ for all $N$.
Passing to the limit (in any $L^p$, $1 \leq p < \infty$) as $N \to \infty$ we obtain $\tilde g_{(F_1 \otimes F_2),(\varphi_1 \otimes \varphi_2)} = \tilde g_{F_1,\varphi_1} \tilde g_{F_2,\varphi_2}$. Therefore, one can conclude
\begin{equation*}
\begin{aligned}
g_{(F_1 \otimes F_2),(\varphi_1 \otimes \varphi_2)} = \tilde g_{(F_1 \otimes F_2),(\varphi_1 \otimes \varphi_2)},
\end{aligned}
\end{equation*}
which is \eqref{eqn:algebra_limit} for $F = (F_1 \otimes F_2)$ when $\varphi \in C_c(\R^{k_1+k_2})$ is in the tensorized form $\varphi = \varphi_1 \otimes \varphi_2$.

Finally, any $\varphi \in C_c(\R^{k_1+k_2})$ can be approximated by a sum of tensorized functions 
so that we derive \eqref{eqn:algebra_limit} for $F = (F_1 \otimes F_2)$ with any arbitrary $\varphi \in C_c(\R^{k_1+k_2})$.

\item[(iii)] Consider an element $F \in \mathscr{T}$ satisfying \eqref{eqn:algebra_limit}. Then for any $\psi \in C([0,1])$, $\phi \in C_c(\R^k)$,
\begin{equation*}
\begin{aligned}
& \int_{[0,1]} \psi(\xi) g_{F^*,\varphi}(\xi) \;\rd \xi
\\
&\quad=  \int_{[0,1]} \psi(\xi) \int_{z \in \R^k} \varphi(z_1,\dots,z_k) F^*(w, f)(\xi, \rd z_1,\dots, \rd z_k) \;\rd \xi
\\
&\quad= \int_{\xi \in [0,1]} \psi(\xi) \int_{z \in \R^k} \varphi(z_1,\dots,z_k) \int_{\zeta \in [0,1]} F(w,f) (\zeta, \rd z_1,\dots, \rd z_k) w(\rd \xi, \zeta) \;\rd \zeta
\\
&\quad = \int_{\xi \in [0,1]} \psi(\xi) g_{F,\varphi}(\zeta) w(\rd \xi, \zeta) \;\rd \zeta.
\end{aligned}
\end{equation*}
By a similar argument, for all $N$.
\begin{equation*}
\begin{aligned}
\;& \int_{[0,1]} \psi(\xi) \tilde g_{F^*,\varphi}^N(\xi) \;\rd \xi = \int_{\xi \in [0,1]} \psi(\xi) \tilde g_{F,\varphi}^N(\zeta) \tilde w_{N;\#}(\rd \xi, \zeta) \;\rd \zeta.
\end{aligned}
\end{equation*}
Next, by the convergence
\begin{equation*}
\begin{aligned}
\psi(\xi) \tilde g_{F,\varphi}^N(\zeta) \to \;& \psi(\xi) g_{F,\varphi}(\zeta) \; \text{ in }\; L^1_{\zeta}([0,1], C_\xi[0,1]),
\\
\tilde w_{N;\#}(\rd \xi, \zeta) \;\rd \zeta \overset{\ast}{\rightharpoonup} \;& w(\rd \xi, \zeta) \;\rd \zeta \; \text{ in }\; L^\infty_{\zeta}([0,1], \mathcal{M}_\xi[0,1]),
\end{aligned}
\end{equation*}
we obtain that
\begin{equation*}
\begin{aligned}
\lim_{N \to \infty} \int_{\xi \in [0,1]} \psi(\xi) \tilde g_{F,\varphi}^N(\zeta) \tilde w_{N;\#}(\rd \xi, \zeta) \;\rd \zeta = \int_{\xi \in [0,1]} \psi(\xi) g_{F,\varphi}(\zeta) w(\rd \xi, \zeta) \;\rd \zeta.
\end{aligned}
\end{equation*}
Hence $g_{F^*,\varphi} = \tilde g_{F^*,\varphi}$, which is \eqref{eqn:algebra_limit} for $F^*$.

\end{itemize}

\end{proof}

We may now conclude the proof of Proposition~\ref{prop:independent_compactness}.
For any $T \in \mathcal{T}$, there exists $F \in \mathscr{T}$ such that
\begin{equation*}
\begin{aligned}
\tau_\infty(T,\tilde w_{N;\#},\tilde f_{N;\#}) = \;& \int_{[0,1]} F(\tilde w_N, \tilde f_N)(\Phi_N(\xi), \rd z_1,\dots, \rd z_{|T|}) \;\rd \xi,
\\
\tau_\infty(T,w,f) = \;& \int_{[0,1]} F(w, f)(\xi, \rd z_1,\dots, \rd z_{|T|}) \;\rd \xi.
\end{aligned}
\end{equation*}
For any $\varphi \in C_c(\R^{|T|})$, by Lemma~\ref{lem:algebra_limit},
\begin{equation*}
\begin{aligned}
\;& \lim_{N \to \infty} \int_{z \in \R^{|T|}} \tau_\infty(T,\tilde w_{N;\#},\tilde f_{N;\#}) (\rd z_1,\dots, \rd z_{|T|})
\\
= \;& \lim_{N \to \infty} \int_{[0,1]} \int_{z \in \R^{|T|}} \varphi(z_1,\dots,z_{|T|}) F(\tilde w_N, \tilde f_N)(\Phi_N(\xi), \rd z_1,\dots, \rd z_{|T|}) \;\rd \xi
\\
= \;& \int_{[0,1]} \int_{z \in \R^{|T|}} \varphi(z_1,\dots,z_{|T|}) F(w, f)(\xi, \rd z_1,\dots, \rd z_{|T|}) \;\rd \xi
\\
= \;& \int_{z \in \R^{|T|}} \tau_\infty(T,w,f) (\rd z_1,\dots, \rd z_{|T|}).
\end{aligned}
\end{equation*}
Since $\varphi \in C_c(\R^{|T|})$ is arbitrary we conclude \eqref{eqn:assumption_4_compact}, restated here:
\begin{equation*}
\begin{aligned}
\tau_\infty(T,\tilde w_N, \tilde f_N) \overset{\ast}{\rightharpoonup} \tau_\infty(T,w,f) \in \mathcal{M}(\R^{|T|}), \quad \forall T \in \mathcal{T}.
\end{aligned}
\end{equation*}

\end{proof}
\section{Proofs of the quantitative results} \label{sec:technical_proofs}
\subsection{The hierarchy of equations} \label{sec:hierarchy}
The subsection provides the main proofs of Proposition~\ref{prop:hierarchy_of_equations}, \ref{prop:Hilbert_a_priori_bound} and \ref{prop:hierarchy_of_equations_limit}, which derive the hierarchy of equations from the Liouville equation \eqref{eqn:IF_Liouville_PDE} and the Vlasov equation \eqref{eqn:Vlasov_transport}-\eqref{eqn:Vlasov_mean_field}.

We begin with the proof of Proposition~\ref{prop:hierarchy_of_equations}, showing that the observables corresponding to the laws of $(X^{1;N}_0,\dots, X^{N;N}_0)$ solving \eqref{eqn:IF_multi_agent} satisfy the extended BBGKY hierarchy \eqref{eqn:hierarchy_equation}-\eqref{eqn:hierarchy_equation_remainder}.
\begin{proof} [Proof of Proposition~\ref{prop:hierarchy_of_equations}]

Since the coefficients are bounded Lipschitz, the well-posedness of the SDE system \eqref{eqn:IF_multi_agent} and the Liouville-type equation \eqref{eqn:IF_Liouville_PDE} are classical results. 
For simplicity of the presentation, we avoid using weak formulations but only present a formal calculation. 

Consider any distinct indexes $i_1,\dots,i_k \in \{1,\dots,N\}$. It is easy to verify the following identity deriving the marginal laws from the full joint law,
\begin{equation*}
\begin{aligned}
f_{N}^{i_1,\dots, i_k}(t,z_1,\dots,z_{k}) \defeq \;& \law (X^{i_1;N}_t,\dots, X^{i_{k};N}_t)
\\
=\;&  \bigg(\int_{\R^{N - k}} f_{N}(t,x_1,\dots,x_N) \textstyle{\prod_{i \neq i_1,\dots, i_k} \rd x_i}\bigg)\bigg|_{\forall l = 1,\dots, k, \; x_{i_l} = z_l}.
\end{aligned}
\end{equation*}
By integrating Liouville equation \eqref{eqn:IF_Liouville_PDE} along spatial directions $i \notin \{i_1,\dots,i_{k}\}$ and calculate the summation $i \in \{i_1,\dots,i_{k}\}$ and $i \notin \{i_1,\dots,i_{k}\}$ separately, we obtain equations for the marginals,
\begin{equation} \label{eqn:hierarchy_equation_marginal_integrate}
\begin{aligned}
& \partial_t f_{N}^{i_1,\dots, i_{k}}(t,z_1,\dots,z_{k})
\\
&\quad =  \sum_{m = 1}^{k} \Bigg\{ \bigg[ - \partial_{z_m}(\mu(z_m) f_{N}^{i_1,\dots, i_{k}}(t,z)) + \frac{\sigma^2}{2} \partial_{z_m}^2 f_{N}^{i_1,\dots, i_{k}}(t,z)
\\
&\qquad - \nu(z_m) f_{N}^{i_1,\dots, i_{k}}(t,z) + \delta_0(z_m) \bigg( \int_{\R} \nu(u_m) f_{N}^{i_1,\dots, i_{k}}(t,u - {w_{N;i_m}^{i_1,\dots, i_{k}}}) \Big) \;\rd u_m \bigg)\bigg|_{\forall n \neq m,\, u_n = z_n} \bigg] \Bigg\}
\\
&\qquad\  + \sum_{i \neq i_1,\dots,i_{k}} \int_{\R} \nu(z_{k+1}) \bigg( f_{N}^{i_1,\dots, i_{k}, i}(t,z - {w_{N;i}^{i_1,\dots, i_{k}, i}}) - f_{N}^{i_1,\dots, i_{k}, i}(t,z) \bigg) \;\rd z_{k+1}.
\end{aligned}
\end{equation}
We can reformulate the last line as
\begin{equation*}
\begin{aligned}
& \sum_{i \neq i_1,\dots,i_{k}} \int_{\R} \nu(z_{k+1}) \bigg( f_{N}^{i_1,\dots, i_{k}, i}(t,z - {w_{N;i}^{i_1,\dots, i_{k}, i}}) - f_{N}^{i_1,\dots, i_{k}, i}(t,z) \bigg) \;\rd z_{k+1}
\\
&\quad= \sum_{i \neq i_1,\dots,i_{k}} \int_{\R} \nu(z_{k+1}) \bigg( \int_0^1 \sum_{m=1}^{k} - w_{i_m,i} \partial_{z_m} f_{N}^{i_1,\dots, i_{k}, i}(t,z - r {w_{N;i}^{i_1,\dots, i_{k}, i}}) \;\rd r \bigg) \;\rd z_{k+1}
\\
&\quad = \sum_{m=1}^{k} - \partial_{z_m} \bigg[ \sum_{i \neq i_1,\dots,i_{k}} w_{i_m,i;N} \int_{\R} \nu(z_{k+1}) \bigg( \int_0^1 f_{N}^{i_1,\dots, i_{k}, i}(t,z - r {w_{N;i}^{i_1,\dots, i_{k}, i}}) \;\rd r \bigg) \;\rd z_{k+1} \bigg],
\end{aligned}
\end{equation*}
changing it into an additional advection term $\partial_{z_m}[\dots]$ to the equation. 

Introduce the simple identity
\begin{equation*}
\begin{aligned}
f_{N}^{i_1,\dots, i_{k}}(u - {w_{N;i_m}^{i_1,\dots, i_{k}}}) = f_{N}^{i_1,\dots, i_{k}}(u) - \big\{f_{N}^{i_1,\dots, i_{k}}(u) - f_{N}^{i_1,\dots, i_{k}}(u - {w_{N;i_m}^{i_1,\dots, i_{k}}})\big\},
\end{aligned}
\end{equation*}
and proceed to do the same for $f_{N}^{i_1,\dots, i_{k}, i}(z - r {w_{N;i}^{i_1,\dots, i_{k}, i}})$, so that
the marginal equations \eqref{eqn:hierarchy_equation_marginal_integrate} now read
\begin{equation} \label{eqn:hierarchy_equation_marginal}
\begin{aligned}
& \partial_t f_{N}^{i_1,\dots, i_{k}}(z_1,\dots,z_{k})
\\
&\quad = \sum_{m = 1}^{k} \Bigg\{ \bigg[ - \partial_{z_m}(\mu(z_m) f_{N}^{i_1,\dots, i_{k}}(z)) + \frac{\sigma^2}{2} \partial_{z_m}^2 f_{N}^{i_1,\dots, i_{k}}(z) - \nu(z_m) f_{N}^{i_1,\dots, i_{k}}(z)
\\
&\qquad + \delta_0(z_m) \bigg( \int_{\R} \nu(u_m) \Big( f_{N}^{i_1,\dots, i_{k}}(u) - \big\{f_{N}^{i_1,\dots, i_{k}}(u) - f_{N}^{i_1,\dots, i_{k}}(u - {w_{N;i_m}^{i_1,\dots, i_{k}}})\big\} \Big) \;\rd u_m \bigg)\bigg|_{\forall n \neq m,\, u_n = z_n} \bigg]
\\
& \qquad- \partial_{z_m} \bigg[ \sum_{i \neq i_1,\dots,i_{k}} w_{i_m,i;N} \int_{\R} \nu(z_{k+1}) \bigg( \int_0^1 f_{N}^{i_1,\dots, i_{k}, i}(z) \\
&\qquad\qquad\qquad\qquad\qquad- \big\{f_{N}^{i_1,\dots, i_{k}, i}(z) - f_{N}^{i_1,\dots, i_{k}, i}(z - r {w_{N;i}^{i_1,\dots, i_{k}, i}})\big\} \;\rd r \bigg) \;\rd z_{k+1} \bigg] \Bigg\},
\end{aligned}
\end{equation}
where we omit variable $t$ for simplicity.

By taking the time derivative to the definition of observables~\eqref{eqn:hierarchy}, restated here 
\begin{equation*}
\begin{aligned}
\tau_N (T,w_N,f_N)(t,z) \defeq \;& \frac{1}{N} \sum_{i_1,\dots, i_{|T|} = 1}^N w_{N,T}(i_1,\dots, i_{|T|}) f_{N}^{i_1,\dots, i_{|T|}}(t,z_1,\dots,z_{|T|})
\end{aligned}
\end{equation*}
and substituting the right hand side $\partial_t f_{N}^{i_1,\dots, i_{|T|}}$ by the marginal equation \eqref{eqn:hierarchy_equation_marginal} with $k = |T|$, we obtain that
\begin{equation*}
\begin{aligned}
& \partial_t \bigg( \frac{1}{N} \sum_{i_1,\dots, i_{|T|} = 1}^N w_{N,T}(i_1,\dots, i_{|T|}) f_{N}^{i_1,\dots, i_{|T|}}(z_1,\dots,z_{|T|}) \bigg)= \frac{1}{N} \sum_{i_1,\dots, i_{|T|} = 1}^N w_{N,T}(i_1,\dots, i_{|T|}) \sum_{m = 1}^{|T|} \Bigg\{ \\
&\ \bigg[ - \partial_{z_m}(\mu(z_m) f_{N}^{i_1,\dots, i_{|T|}}(z)) + \frac{\sigma^2}{2} \partial_{z_m}^2 f_{N}^{i_1,\dots, i_{|T|}}(z) - \nu(z_m) f_{N}^{i_1,\dots, i_{|T|}}(z)\\
&\ + \delta_0(z_m) \bigg( \int_{\R} \nu(u_m) \Big( f_{N}^{i_1,\dots, i_{|T|}}(u) - \big\{f_{N}^{i_1,\dots, i_{|T|}}(u)- f_{N}^{i_1,\dots, i_{|T|}}(u - {w_{N;i_m}^{i_1,\dots, i_{|T|}}})\big\} \Big) \;\rd u_m \bigg)\bigg|_{\forall n \neq m,\, u_n = z_n} \bigg]
\\
& - \partial_{z_m} \bigg[ \sum_{i \neq i_1,\dots,i_{k}} \!\!\!\! w_{i_m,i;N} \!\! \int_{\R} \! \nu(z_{|T|+1}) \! \bigg( \int_0^1 f_{N}^{i_1,\dots, i_{|T|}, i}(z) - \big\{f_{N}^{i_1,\dots, i_{|T|}, i}(z) \\
&\qquad\qquad\qquad\qquad\qquad\qquad\qquad - f_{N}^{i_1,\dots, i_{|T|}, i}(z - r {w_{N;i}^{i_1,\dots, i_{|T|}, i}})\big\} \rd r \bigg) \rd z_{|T|+1} \bigg] \Bigg\}.
\end{aligned}
\end{equation*}
Noticing the identity $w_{N,T + j}(i_1,\dots, i_{|T| + 1}) = w_{N,T}(i_1,\dots, i_{|T|}) w_{i_j, i_{|T|+1}}$,
we see that all the marginals, except the two terms of form
$\{f_N^{\dots}(\cdot) - f_N^{\dots}(\cdot - w)\}$, are expressed in the right way so they can be rewritten as observables, obtaining \eqref{eqn:hierarchy_equation} as the approximate hierarchy and \eqref{eqn:hierarchy_equation_remainder} as the explicit form of the remainders.
\end{proof}

We now turn to the proof of Proposition~\ref{prop:Hilbert_a_priori_bound}. It is worth noting that the main Gronwall estimate could also be written in the probabilistic language of It\^o calculus. However, we prefer to keep an approach and notation similar to the rest of the proofs presented.

\begin{proof} [Proof of Proposition~\ref{prop:Hilbert_a_priori_bound}]

To simplify the argument, we only present a formal calculation where the tensorized weight $\eta^{\otimes |T|}$ is directly used as the test function, while, strictly speaking, the valid test functions for distributional solutions should have compact support. Given that the remaining coefficients are bounded Lipschitz and all terms in the subsequent calculation are non-negative, passing the limit to justify the use of unbounded weight on the dual side poses no problems.

The weighted total variation $\||\tau_N|(T) \eta^{\otimes |T|}\|_{\mathcal{M}(\R^{|T|})}$ can be decomposed as
\begin{equation*}
\begin{aligned}
\||\tau|(T)(t,\cdot) \eta^{\otimes |T|}\|_{\mathcal{M}(\R^{|T|})} = \;& \int_{\R^{|T|}} \frac{1}{N} \sum_{i_1,\dots, i_{|T|} = 1}^N \big| w_{N,T}(i_1,\dots, i_{|T|}) \big| f_{N}^{i_1,\dots, i_{|T|}}(t,z) \eta^{\otimes |T|}(z) \;\rd z
\\
= \;& \frac{1}{N} \sum_{i_1,\dots, i_{|T|} = 1}^N \big| w_{N,T}(i_1,\dots, i_{|T|}) \big| \int_{\R^{|T|}} f_{N}^{i_1,\dots, i_{|T|}}(t,z) \eta^{\otimes |T|}(z) \;\rd z.
\end{aligned}
\end{equation*}
For any distinct indexes $i_1,\dots,i_k$, we have
\begin{equation*}
\begin{aligned}
\int_{\R^{|T|}} f_{N}^{i_1,\dots, i_{|T|}}(t,z) \eta^{\otimes |T|}(z) \;\rd z = \int_{\R^N} f_{N}(t,x) {\textstyle \prod_{l = 1}^{|T|} \eta(x_{i_l})} \;\rd x.
\end{aligned}
\end{equation*}
%
The forthcoming estimate is not exclusive to our specific choice $\eta = \eta_\alpha$, but for any weight function adhering to the form
\begin{equation*}
\begin{aligned}
\eta(x) = \exp( h(x)), \quad \forall x \in \R
\end{aligned}
\end{equation*}
such that $\|h'\|_{L^\infty}$, $\|h''\|_{L^\infty}$ are bounded and $h(0) \leq h(x)$.
Our choice of $\eta = \eta_\alpha$ is clearly included by choosing
$h(x) = \sqrt{1 + \alpha^2 x^2}$, resulting in $\|h'\|_{L^\infty} \leq \alpha$ and $\|h''\|_{L^\infty} \leq \alpha^2$.
The following inequalities are immediate results by chain rule and fundamental theorem of calculus.
\begin{lem} \label{lem:weight_inequalities} 
For any weight function of form $\eta(x) = \exp( h(x))$ such that $\|h'\|_{L^\infty}$, $\|h''\|_{L^\infty}$ are bounded and $h(0) \leq h(x)$, one has that
\begin{equation*}
\begin{aligned}
|\eta'/\eta|(x) \leq \|h'\|_{L^\infty}, \quad |\eta''/\eta|(x) \leq \|h''\|_{L^\infty} + \|h'\|_{L^\infty}^2,
\end{aligned}
\end{equation*}
and
\begin{equation*}
\begin{aligned}
\eta(x + y) - \eta(x) 
\leq \|h'\|_{L^\infty} |y| \exp(\|h'\|_{L^\infty} |y|) \eta(x).
\end{aligned}
\end{equation*}
The last inequality can be extended to the tensorized case $\eta^{\otimes k}(x) = \prod_{l = 1}^{k} \eta(x_{i_l})$ as
\begin{equation*}
\begin{aligned}
\eta^{\otimes k}(x + y) - \eta^{\otimes k}(x) 
\leq \|h'\|_{L^\infty} \|y\|_{\ell^1} \exp(\|h'\|_{L^\infty} \|y\|_{\ell^1}) \eta^{\otimes k}(x).
\end{aligned}
\end{equation*}

\end{lem}

We are now ready to prove Proposition~\ref{prop:Hilbert_a_priori_bound} under the more general assumption that $\eta(x) = \exp( h(x))$.
Since $f_N$ solves \eqref{eqn:IF_Liouville_PDE} in the distributional sense, it is easy to verify that
\begin{equation*}
\begin{aligned}
& \int_{\R^N} f_{N}(t,x) {\textstyle \prod_{l = 1}^{|T|} \eta(x_{i_l})} \;\rd x
= \int_{\R^N} f_{N}(0,x) {\textstyle \prod_{l = 1}^{|T|} \eta(x_{i_l})} \;\rd x
\\
&\quad + \int_0^t \int_{\R^N} f_{N}(s,x) \Bigg[ \sum_{m=1}^{|T|} \bigg( \mu(x_{i_m}) (\eta'/\eta) (x_{i_m}) + \frac{1}{2}\sigma^2 (\eta''/\eta) (x_{i_m}) \bigg) {\textstyle \prod_{l = 1}^{|T|} \eta(x_{i_l})}
\\
&\quad + \sum_{j = i_1,\dots, i_{|T|}} \nu(x_j) \bigg( \frac{\eta(0)}{\eta(x_{j})} {\textstyle \prod_{l = 1}^{|T|} \eta(x_{i_l} + w_{i_l,j;N})} - {\textstyle \prod_{l = 1}^{|T|} \eta(x_{i_l})} \bigg)
\\
&\quad + \sum_{j \neq i_1,\dots, i_{|T|}} \nu(x_j) \bigg( {\textstyle \prod_{l = 1}^{|T|} \eta(x_{i_l} + w_{i_l,j;N})} - {\textstyle \prod_{l = 1}^{|T|} \eta(x_{i_l})} \bigg)
\Bigg] \;\rd x \rd s.
\end{aligned}
\end{equation*}
By Lemma~\ref{lem:weight_inequalities}, we have that
\begin{equation*}
\begin{aligned}
& \int_{\R^N} f_{N}(t,x) {\textstyle \prod_{l = 1}^{|T|} \eta(x_{i_l})} \;\rd x
\leq \int_{\R^N} f_{N}(0,x) {\textstyle \prod_{l = 1}^{|T|} \eta(x_{i_l})} \;\rd x
\\
&\qquad + \int_0^t \int_{\R^N} f_{N}(s,x) \Bigg[ \sum_{m=1}^{|T|} \bigg( \|\mu\|_{L^\infty} \|h'\|_{L^\infty} + \frac{1}{2}\sigma^2 (\|h''\|_{L^\infty} + \|h'\|_{L^\infty}^2) \bigg) {\textstyle \prod_{l = 1}^{|T|} \eta(x_{i_l})}
\\
&\qquad + \sum_{j = 1}^{N} \|\nu\|_{L^\infty} \; \|h'\|_{L^\infty} {\textstyle \sum_{m=1}^{|T|} |w_{i_m,j;N}|} \exp\Big(\|h'\|_{L^\infty} \; {\textstyle \max_{j} \sum_{i} |w_{i,j;N}|}\Big) {\textstyle \prod_{l = 1}^{|T|} \eta(x_{i_l})} \Bigg] \;\rd x \rd s
\\
&\quad= \int_{\R^N} f_{N}(0,x) {\textstyle \prod_{l = 1}^{|T|} \eta(x_{i_l})} \;\rd x
 + \Bigg[ \sum_{m=1}^{|T|} \bigg( \|\mu\|_{L^\infty} \|h'\|_{L^\infty} + \frac{1}{2}\sigma^2 (\|h''\|_{L^\infty} + \|h'\|_{L^\infty}^2) \bigg)
\\
& \qquad+ \sum_{j = 1}^{N} \sum_{m=1}^{|T|} |w_{i_m,j;N}| \; \|\nu\|_{L^\infty} \|h'\|_{L^\infty} \exp\Big(\|h'\|_{L^\infty} \; {\textstyle \max_{j} \sum_{i} |w_{i,j;N}|}\Big) \Bigg] \int_0^t \int_{\R^N} f_{N}(s,x) {\textstyle \prod_{l = 1}^{|T|} \eta(x_{i_l})} \;\rd x \rd s,
\end{aligned}
\end{equation*}
where the summations of $j = i_1,\dots, i_{|T|}$ and $j \neq i_1,\dots, i_{|T|}$ are combined together by the simple fact that $h(0) \leq h(x_j)$, hence ${\eta(0)}/{\eta(x_{j})} \leq 1$.

Furthermore, we have that
\begin{equation*}
\begin{aligned}
\sum_{j = 1}^{N} \sum_{m=1}^{|T|} |w_{i_m,j;N}| \leq |T| \; {\textstyle \max_{i} \sum_{j} |w_{i,j;N}|}.
\end{aligned}
\end{equation*}
Hence by choosing
\begin{equation*}
\begin{aligned}
C_\mathcal{W} = \;& \textstyle \max\left(\max_{i} \sum_{j} |w_{i,j;N}| ,\ \max_{j} \sum_{i} |w_{i,j;N}|\right),
\\
A_\eta = \;& \Big( \|\mu\|_{L^\infty} \|h'\|_{L^\infty} + \frac{1}{2}\sigma^2 (\|h''\|_{L^\infty} + \|h'\|_{L^\infty}^2) + \|\nu\|_{L^\infty} \|h'\|_{L^\infty} C_{\mathcal{W}} \exp(\|h'\|_{L^\infty} C_{\mathcal{W}}) \Big),
\end{aligned}
\end{equation*}
we conclude that
\begin{equation*}
\begin{aligned}
& \int_{\R^N} f_{N}(t,x) {\textstyle \prod_{l = 1}^{|T|} \eta(x_{i_l})} \;\rd x \\
&\quad \leq \int_{\R^N} f_{N}(0,x) {\textstyle \prod_{l = 1}^{|T|} \eta(x_{i_l})} \;\rd x + \int_0^t |T| A_\eta \int_{\R^N} f_{N}(s,x) {\textstyle \prod_{l = 1}^{|T|} \eta(x_{i_l})} \;\rd x \;\rd s.
\end{aligned}
\end{equation*}
By Gronwall lemma, this implies that
\begin{equation*}
\begin{aligned}
\;& \int_{\R^N} f_{N}(t,x) {\textstyle \prod_{l = 1}^{|T|} \eta(x_{i_l})} \;\rd x \leq \exp \big( |T| A_\eta t \big) \int_{\R^N} f_{N}(0,x) {\textstyle \prod_{l = 1}^{|T|} \eta(x_{i_l})} \;\rd x.
\end{aligned}
\end{equation*}
Taking the summation over $i_1,\dots, i_{|T|}$, we have that
\begin{equation*}
\begin{aligned}
\||\tau|(T) \eta^{\otimes |T|}(t,\cdot)\|_{\mathcal{M}(\R^{|T|})} \leq \;& \exp \big( |T| A_\eta t \big) \||\tau|(T) \eta^{\otimes |T|}(0,\cdot)\|_{\mathcal{M}(\R^{|T|})}
\\
\leq \;& C_\eta \big( M_\eta \exp(A_\eta t_*) \big)^{|T|}
\end{aligned}
\end{equation*}
Finally, by applying Lemma~\ref{lem:convolutional_inequality} to the left hand side, we immediately obtain \eqref{eqn:Hilbert_energy_bound_scale}, restated here,
\begin{equation*}
\begin{aligned}
\| |\tau_N| (T) (t,\cdot) \|_{H^{-1\otimes |T|}_\eta} \leq C_\eta(T) \big(\|K\|_{L^2(\R)} \exp(A_\eta t_*) \big)^{|T|}
\end{aligned}
\end{equation*}
for all $T \in \mathcal{T}, t \in [0,t_*]$.

\end{proof}

Finally, we give the proof of Proposition~\ref{prop:hierarchy_of_equations_limit}.

\begin{proof} [Proof of Proposition~\ref{prop:hierarchy_of_equations_limit}]

We show the well-posedness of Vlasov equation \eqref{eqn:Vlasov_transport}-\eqref{eqn:Vlasov_mean_field} by a classical fixed point argument. Let us first define the mapping $f \mapsto \mathcal{L} f$ as the solution of
\begin{equation*}
\begin{aligned}
\partial_t \mathcal{L} f(t,\xi,x) + \partial_x \Big(\mu^*_{f}(t,\xi,x) \mathcal{L} f(t,\xi,x) \Big) - \frac{\sigma^2}{2} \partial_{xx} \Big( \mathcal{L} f(t,\xi,x) \Big)
\\
+ \nu(x) \mathcal{L} f(t,\xi,x) - \delta_0(x) J_f(t,\xi) = 0
\end{aligned}
\end{equation*}
If $f$ is given, then $J_f$ and $\mu^*_{f}$ are determined, making the above identity a linear equation with respect to $\mathcal{L} f$.
We are going to see that if $f \in L^\infty([0,t_*] \times [0,1] ; H^{-1}_\eta \cap \mathcal{M}_+(\R) )$, then $\mathcal{L} f$ belongs to the same space.

By multiplying  the equation by the weight function $\eta$ and applying Leibniz formula, we obtain that
\begin{equation*}
\begin{aligned}
& \partial_t \mathcal{L} f(t,\xi,x) \eta(x)
\\
& \quad=- \partial_x \Big(\mu^*_{f}(t,\xi,x) \mathcal{L} f(t,\xi,x) \eta(x) \Big) + \frac{\sigma^2}{2} \partial_{xx} \Big( \mathcal{L} f(t,\xi,x) \eta(x) \Big) - \nu(x) \mathcal{L} f(t,\xi,x) \eta(x)\\
&\qquad + \delta_0(x) \eta(0) J_f(t,\xi) + \mu^*_{f}(t,\xi,x) (\eta'/\eta)(x) \mathcal{L} f(t,\xi,x) \eta(x)\\
&\qquad + \frac{\sigma^2}{2} \bigg[ - \partial_x \Big( 2(\eta'/\eta)(x) \mathcal{L} f(t,\xi,x) \eta(x) \Big) + (\eta''/\eta)(x) \mathcal{L} f(t,\xi,x) \eta(x) \bigg].
\end{aligned}
\end{equation*}

We start the a priori estimate of the linear mapping $\mathcal{L}$ by the total mass.
It is straightforward to verify that
\begin{equation} \label{eqn:a_priori_invariance_1}
\begin{aligned}
\| \mathcal{L}f(t,\cdot,\xi)\|_{\mathcal{M}(\R)} \leq \;& \| f(0,\cdot,\xi)\|_{\mathcal{M}(\R)} + \int_0^t J_f(s,\xi) \;\rd s
\\
\leq \;& \| f(0,\cdot,\xi)\|_{\mathcal{M}(\R)} + \int_0^t \|\nu\|_{L^\infty} \| f(s,\cdot,\xi) \|_{\mathcal{M}(\R)} \;\rd s.
\end{aligned}
\end{equation}
Note that by choosing $t_1 = 1/(2\|\nu\|_{L^\infty})$, we have that
\begin{equation*}
\begin{aligned}
\sup_{t \in [0,t_1]} \| f(t,\cdot,\xi)\|_{\mathcal{M}(\R)} \leq 2 \| f(0,\cdot,\xi)\|_{\mathcal{M}(\R)} \implies \sup_{t \in [0,t_1]} \| \mathcal{L}f(t,\cdot,\xi)\|_{\mathcal{M}(\R)} \leq 2 \| f(0,\cdot,\xi)\|_{\mathcal{M}(\R)}.
\end{aligned}
\end{equation*}
Next, consider the $\eta$-weighted total moment,
\begin{equation*}
\begin{aligned}
& \| \mathcal{L}f(t,\cdot,\xi) \eta\|_{\mathcal{M}(\R)}
\\
&\quad \leq \| f(0,\cdot,\xi) \eta\|_{\mathcal{M}(\R)} + \int_0^t \Big\{\eta(0)J_f(s,\xi) + \Big[ \|\mu_f^*(s,\cdot,\xi)\|_{L^\infty} \|\eta'/\eta\|_{L^\infty} + \frac{\sigma^2}{2} \|\eta''/\eta\|_{L^\infty} \Big] \\
&\hspace{330pt} \|\mathcal{L}f(s,\cdot,\xi) \eta\|_{\mathcal{M}(\R)}\Big\} \;\rd s
\\
&\quad \leq  \| f(0,\cdot,\xi) \eta\|_{\mathcal{M}(\R)} + \int_0^t \Big\{\eta(0)\|\nu\|_{L^\infty} \| f(s,\cdot,\xi) \|_{\mathcal{M}(\R)}
\\
& \qquad+ \Big[ \big( \|\mu\|_{L^\infty} + \|w\|_{\mathcal{W}} \|\nu\|_{L^\infty} \| f(s,\cdot,\cdot) \|_{L^\infty_\xi \mathcal{M}_x} \big) \|\eta'/\eta\|_{L^\infty} + \frac{\sigma^2}{2} \|\eta''/\eta\|_{L^\infty} \Big] \| \mathcal{L}f(s,\cdot,\xi) \eta\|_{\mathcal{M}(\R)}\Big\} \;\rd s.
\end{aligned}
\end{equation*}
By taking the supremum over $\xi \in [0,1]$, we have, for $t \in [0,t_1]$,
\begin{equation} \label{eqn:a_priori_invariance_2}
\begin{aligned}
& \| \mathcal{L}f(t,\cdot,\cdot) \eta\|_{L^\infty_\xi \mathcal{M}_x}
\leq \| f(0,\cdot,\cdot) \eta\|_{L^\infty_\xi \mathcal{M}_x} +  \int_0^t \eta(0)\|\nu\|_{L^\infty} \| f \|_{L^\infty_{t,\xi} \mathcal{M}_x}
\\
& \qquad + \Big[ \big( \|\mu\|_{L^\infty} + \|w\|_{\mathcal{W}} \|\nu\|_{L^\infty} \| f \|_{L^\infty_{t,\xi} \mathcal{M}_x} \big) \|\eta'/\eta\|_{L^\infty} + \frac{\sigma^2}{2} \|\eta''/\eta\|_{L^\infty} \Big] \| \mathcal{L}f(s,\cdot,\cdot) \eta\|_{L^\infty_\xi \mathcal{M}_x} \;\rd s
\\
& \quad \leq \bigg( \| f(0,\cdot,\cdot) \eta\|_{L^\infty_\xi \mathcal{M}_x} +  \int_0^t \eta(0)\|\nu\|_{L^\infty} \| f \|_{L^\infty_{t,\xi} \mathcal{M}_x} \;\rd s\bigg)
\\
& \qquad\qquad\exp \bigg( \Big[ \big( \|\mu\|_{L^\infty} + \|w\|_{\mathcal{W}} \|\nu\|_{L^\infty} \| f \|_{L^\infty_{t,\xi} \mathcal{M}_x} \big) \|\eta'/\eta\|_{L^\infty} + \frac{\sigma^2}{2} \|\eta''/\eta\|_{L^\infty} \Big] t \bigg),
\end{aligned}
\end{equation}
where the $L^\infty_t$ should be understood as the supremum over $t \in [0,t_1]$.

We construct the invariance set and show $\mathcal{L}$-contractivity on the set by the following procedure:
For any $ R > R_0 \defeq \| f(0,\cdot,\cdot) \eta\|_{L^\infty_\xi \mathcal{M}_x}$, and any $t_* > 0$, denote
\begin{equation*}
\begin{aligned}
E_{R;t} \defeq \{ f \in \mathcal{M}_+ :  \sup_{s \in [0,t]} \|f (s,\cdot,\cdot) \eta\|_{L^\infty_{\xi} \mathcal{M}_x} < R\}.
\end{aligned}
\end{equation*}
By taking sufficiently small $t_2$, for example
\begin{equation*}
\begin{aligned}
t_2 \leq \min\bigg( \frac{1}{2 \|\nu\|_{L^\infty}} , \frac{R - R_0}{2 \eta(0) \|\nu\|_{L^\infty} R} , \frac{\log \frac{2R}{R+R_0}}{\big( \|\mu\|_{L^\infty} + \|w\|_{\mathcal{W}} \|\nu\|_{L^\infty} R \big) \|\eta'/\eta\|_{L^\infty} + \frac{\sigma^2}{2} \|\eta''/\eta\|_{L^\infty}} \bigg),
\end{aligned}
\end{equation*}
we can make $E_{R;t_2}$ an invariance set, i.e. $\mathcal{L}(E_{R;t_2}) \subset E_{R;t_2}$.

To show that $f \mapsto \mathcal{L} f$ is contracting in the $H^{-1}_\eta$-sense, we consider the following energy estimate: Along each fiber $\xi \in [0,1]$,
\begin{equation*}
\begin{aligned}
& \frac{\rd}{\rd t} \bigg( \frac{1}{2} \int_{\R} \Big[ \Lambda \star \big((\mathcal{L} f - \mathcal{L} g) \eta\big) \Big] (\mathcal{L} f - \mathcal{L} g) \eta \;\rd x \bigg) = \int_{\R} \Big[ \Lambda \star \big( (\mathcal{L} f - \mathcal{L} g) \eta \big) \Big] \partial_t (\mathcal{L} f - \mathcal{L} g) \eta \;\rd x
\\
&\quad= \int_{\R} - \frac{\sigma^2}{2} \Big[ \Lambda \star \partial_x \big( (\mathcal{L} f - \mathcal{L} g) \eta\big) \Big] \Big[ \partial_x \big( (\mathcal{L} f - \mathcal{L} g) \eta\big) \Big]
\\
&\qquad + \Big[ \Lambda \star \partial_x \big( (\mathcal{L} f - \mathcal{L} g) \eta\big) \Big] \bigg[ \mu^*_{f} (\mathcal{L} f - \mathcal{L} g) \eta + (\mu^*_{f} - \mu^*_{g}) (\mathcal{L} g) \eta + \sigma^2 (\eta'/\eta) (\mathcal{L} f - \mathcal{L} g) \eta \bigg]
\\
&\qquad + \Big[ \Lambda \star \big( (\mathcal{L} f - \mathcal{L} g) \eta\big) \Big] \bigg[ - \nu (\mathcal{L} f - \mathcal{L} g) \eta + \delta_0 \eta(0) (J_f - J_g)
\\
& \qquad + \mu^*_{f} (\eta'/\eta) (\mathcal{L} f - \mathcal{L} g) \eta + (\mu^*_{f} - \mu^*_{g}) (\eta'/\eta) (\mathcal{L} g) \eta + \frac{\sigma^2}{2} (\eta''/\eta) (\mathcal{L} f - \mathcal{L} g) \eta \bigg] \;\rd x.
\end{aligned}
\end{equation*}
Apply Cauchy-Schwartz inequality, we obtain that
\begin{equation*}
\begin{aligned}
& \frac{\rd}{\rd t} \bigg( \int_{\R} \Big[ \Lambda \star \big((\mathcal{L} f - \mathcal{L} g) \eta\big) \Big] (\mathcal{L} f - \mathcal{L} g) \eta \;\rd x \bigg)
\\
&\quad \leq \frac{4}{\sigma^2} \|\mu^*_{f} (\mathcal{L} f - \mathcal{L} g)\|_{H^{-1}_\eta}^2 + \frac{4}{\sigma^2} \|(\mu^*_{f} - \mu^*_{g}) (\mathcal{L} g)\|_{H^{-1}_\eta}^2 + 4 \| (\eta'/\eta) (\mathcal{L} f - \mathcal{L} g)\|_{H^{-1}_\eta}^2 
\\
& \qquad+ \Big( 4 + \frac{\sigma^2}{2} \Big) \|(\mathcal{L} f - \mathcal{L} g)\|_{H^{-1}_\eta}^2 + \|\nu (\mathcal{L} f - \mathcal{L} g)\|_{H^{-1}_\eta}^2 + \|\delta_0 (J_f - J_g)\|_{H^{-1}_\eta}^2
\\
& \qquad+ \|\mu^*_{f} (\eta'/\eta) (\mathcal{L} f - \mathcal{L} g)\|_{H^{-1}_\eta}^2 + \|(\mu^*_{f} - \mu^*_{g}) (\eta'/\eta) (\mathcal{L} g)\|_{H^{-1}_\eta}^2 + \frac{\sigma^2}{2} \|(\eta''/\eta) (\mathcal{L} f - \mathcal{L} g)\|_{H^{-1}_\eta}^2.
\end{aligned}
\end{equation*}
Applying Lemma~\ref{lem:commutator_inequality}, we further have that
\begin{equation*}
\begin{aligned}
& \frac{\rd}{\rd t} \| (\mathcal{L} f - \mathcal{L} g)\|_{H^{-1}_\eta}^2 \leq  \bigg( \frac{16}{\sigma^2} \|\mu^*_{f}\|_{W^{1,\infty}}^2 + 16 \|\eta'/\eta\|_{W^{1,\infty}}^2 + \Big( 4 + \frac{\sigma^2}{2} \Big)
\\
&\qquad + 4 \|\nu\|_{W^{1,\infty}}^2 + 4 \|\mu^*_{f}\|_{W^{1,\infty}}^2 \|\eta'/\eta\|_{W^{1,\infty}}^2 + 2 \sigma^2 \|\eta''/\eta\|_{W^{1,\infty}}^2 \bigg) \| (\mathcal{L} f - \mathcal{L} g)\|_{H^{-1}_\eta}^2
\\
&\qquad + \bigg( \frac{4}{\sigma^2} \| (\mathcal{L} g)\|_{H^{-1}_\eta}^2 + 4\|\eta'/\eta\|_{W^{1,\infty}}^2 \| (\mathcal{L} g)\|_{H^{-1}_\eta}^2 \bigg) |\mu^*_{f} - \mu^*_{g}|^2 + \|\delta_0\|_{H^{-1}_\eta}^2 |J_f - J_g|^2.
\end{aligned}
\end{equation*}
Now let us consider the integration over $\xi \in [0,1]$. Firstly, using that $w \in \mathcal{W}$ combined with classical interpolation,
\begin{equation*}
\begin{aligned}
& \int_{[0,1]} |\mu^*_{f}(t,\xi,x) - \mu^*_{g}(t,\xi,x)|^2 \;\rd \xi = \int_{[0,1]} \bigg( \int_{[0,1]} w(\xi, \zeta) \big( J_f(t,\zeta) - J_g(t,\zeta) \big) \;\rd \zeta \bigg)^2 \;\rd \xi
\\
&\qquad \leq  \|w\|_{\mathcal{W}}^2 \|J_f(t,\cdot) - J_g(t,\cdot)\|_{L^2_\xi}^2.
\end{aligned}
\end{equation*}
Secondly, by Lemma~\ref{lem:firing_rate_difference},
\begin{equation*}
\begin{aligned}
\big| J_f(t,\xi) - J_g(t,\xi) \big| \leq \bigg| \int_{\R} \nu(x) \big( f(t,\xi,x) - g(t,\xi,x) \big) \;\rd x \bigg| \leq C (\alpha) \|\nu\|_{W^{1,\infty}} \|f(t,\cdot,\xi) - g(t,\cdot,\xi)\|_{H^{-1}_\eta}.
\end{aligned}
\end{equation*}
Hence, we have that
\begin{equation*}
\begin{aligned}
& \|J_f(t,\cdot) - J_g(t,\cdot)\|_{L^2_\xi}^2 = \int_{[0,1]} \big| J_f(t,\xi) - J_g(t,\xi) \big|^2 \;\rd \xi\\
&\quad \leq  C(\alpha)^2 \|\nu\|_{W^{1,\infty}}^2 \int_{[0,1]} \|f(t,\cdot,\xi) - g(t,\cdot,\xi)\|_{H^{-1}_\eta}^2 \;\rd \xi
 =  C(\alpha)^2 \|\nu\|_{W^{1,\infty}}^2 \|f - g\|_{L^2_\xi (H^{-1}_\eta)_x}^2.
\end{aligned}
\end{equation*}
Therefore, by integrating over $\xi \in [0,1]$,
\begin{equation} \label{eqn:Vlasov_Gronwall}
\begin{aligned}
\| (\mathcal{L} f - \mathcal{L} g)(t,\cdot,\cdot)\|_{L^2_\xi (H^{-1}_\eta)_x}^2
\leq \;& \int_0^t M_0 \| (\mathcal{L} f - \mathcal{L} g) (s,\cdot,\cdot)\|_{L^2_\xi (H^{-1}_\eta)_x}^2 +  M_1 \|(f - g) (s,\cdot,\cdot)\|_{L^2_\xi (H^{-1}_\eta)_x}^2 \;\rd s
\\
\leq \;& \exp(M_0 t) \int_0^t M_1 \|(f - g) (s,\cdot,\cdot)\|_{L^2_\xi (H^{-1}_\eta)_x}^2 \;\rd s
\end{aligned}
\end{equation}
where $M_0,M_1$ are required to satisfy that
\begin{equation*}
\begin{aligned}
M_0 \geq \;&\sup_{t \in [0,t_2]} \bigg( \frac{16}{\sigma^2} \|\mu^*_{f}\|_{L^\infty_\xi W^{1,\infty}_x}^2 + 16 \|\eta'/\eta\|_{W^{1,\infty}}^2 + \Big( 4 + \frac{\sigma^2}{2} \Big)
\\
\;&\quad + 4 \|\nu\|_{W^{1,\infty}}^2 + 4 \|\mu^*_{f}\|_{L^\infty_\xi W^{1,\infty}_x}^2 \|\eta'/\eta\|_{W^{1,\infty}}^2 + 2 \sigma^2 \|\eta''/\eta\|_{W^{1,\infty}}^2 \bigg)
\\
M_1 \geq \;& \sup_{t \in [0,t_2]} \bigg[ \bigg( \frac{4}{\sigma^2} \| (\mathcal{L} g)\|_{L^\infty_\xi (H^{-1}_\eta)_x}^2 + 4\|\eta'/\eta\|_{W^{1,\infty}}^2 \| (\mathcal{L} g)\|_{L^\infty_\xi (H^{-1}_\eta)_x}^2 \bigg) \|w\|_{\mathcal{W}}^2 + \|\delta_0\|_{H^{-1}_\eta}^2 \bigg] C(\alpha)^2 \|\nu\|_{W^{1,\infty}}^2.
\end{aligned}
\end{equation*}
In addition, by $w \in \mathcal{W}$ and Lemma~\ref{lem:firing_rate_difference}, we can derive
\begin{equation*}
\begin{aligned}
\|\mu^*_{f}\|_{L^\infty_\xi W^{1,\infty}_x} \leq \;& \|\mu\|_{W^{1,\infty}_x} + \sup_{\xi \in [0,1]} \bigg| \int_0^1 w(\xi, \zeta) J_f(t,\zeta) \;\rd \zeta \bigg|
\\
\leq \;& \|\mu\|_{W^{1,\infty}_x} + \|w\|_{\mathcal{W}} \; \|J_f(t,\cdot)\|_{L^\infty}
\\
\leq \;& \|\mu\|_{W^{1,\infty}_x} + \|w\|_{\mathcal{W}} \; C (\alpha) \|\nu\|_{W^{1,\infty}} \|f\|_{L^\infty_\xi (H^{-1}_\eta)_x}.
\end{aligned}
\end{equation*}
When $f,g,\mathcal{L}f,\mathcal{L}g \in E_{R;t_2}$, by Lemma~\ref{lem:convolutional_inequality}, we have that
\begin{equation*}
\begin{aligned}
\|f\|_{L^\infty_\xi (H^{-1}_\eta)_x} \leq \frac{R}{2}, \quad \|(\mathcal{L}g)\|_{L^\infty_\xi (H^{-1}_\eta)_x} \leq \frac{R}{2},
\end{aligned}
\end{equation*}
for $t \in [0,t_2]$.

Hence $M_0,\; M_1$ in \eqref{eqn:Vlasov_Gronwall} can be chosen such that they only depend on $R$ and the regularity of the various fixed coefficients in the system. By choosing sufficiently small $t_* > 0$, for example,
\begin{equation*}
\begin{aligned}
t_* \leq \max\left(t_2,\ \frac{1}{3M_1},\ \frac{\log 2}{M_0}\right),
\end{aligned}
\end{equation*}
by \eqref{eqn:Vlasov_Gronwall} we conclude that $\mathcal{L}$ is contracting on the set $\mathcal{L}(E_{R;t_*})$ for the $L^2_\xi (H^{-1}_\eta)_x$ norm.
Repeating the argument allows extending the weak solution to any finite time interval as usual, since the a priori estimates \eqref{eqn:a_priori_invariance_1} and \eqref{eqn:a_priori_invariance_2} do not blow up in finite time.

\bigskip

We now turn to the derivation of the limiting hierarchy.
Taking the derivative of $\tau_\infty(T) = \tau_\infty(T,w,f)$ in Definition~\ref{defi:limiting_observables}, we first obtain 
\begin{equation*}
\begin{aligned}
&\partial_t \tau_\infty(T,w,f)(t,z) =  \sum_{m=1}^{|T|} \bigg[ -\partial_{z_m}\Big( \mu(z_m) \tau_\infty(T)(t,z)\Big) + \frac{\sigma^2}{2} \partial_{z_m}^2 \tau_\infty(T)(t,z)
\\
&\qquad -\nu(z_m) \tau_\infty(T)(t,z) + \delta_0(z_m) \bigg( \int_{\R} \nu(u_m) \tau_\infty(T)(t,u) \bigg) \bigg|_{\forall n \neq m, u_n=z_n}
\\
& \qquad - \partial_{z_m} \bigg( \int_{[0,1]^{|T|}} w_T(\xi_1,\dots,\xi_{|T|}) f^{\otimes |T|}(t,z_1,\xi_1,\dots,z_{|T|},\xi_{|T|})
\\
& \qquad \bigg( \int_0^1 w(\xi_m,\xi_{|T|+1}) \int_{\R} \nu(z_{|T|+1}) f(t,z_{|T|+1},\xi_{|T|+1}) \;\rd z_{|T|+1} \rd \xi_{|T|+1} \bigg) \;\rd \xi_1,\dots,\xi_{|T|} \bigg) \bigg].
\end{aligned}
\end{equation*}
The last term can be rewritten by using the observables with one more leaf, resulting the limiting hierarchy \eqref{eqn:limit_hierarchy_equation}, restated here:
\begin{equation*}
\begin{aligned}
& \partial_t \tau_\infty (T)(t,z)
\\
&\quad=  \sum_{m = 1}^{|T|} \Bigg\{ \bigg[ - \partial_{z_m}(\mu(z_m) \tau_\infty (T)(t,z)) + \frac{\sigma^2}{2} \partial_{z_m}^2 \tau_\infty (T)(t,z)
\\
&\qquad - \nu(z_m) \tau_\infty (T)(t,z) + \delta_0(z_m) \bigg( \int_{\R} \nu(u_m) \tau_\infty (T)(t,u) \;\rd u_m \bigg)\bigg|_{\forall n \neq m,\, u_n = z_n} \bigg]
\\
&\qquad - \partial_{z_m} \bigg[ \int_{\R} \nu(z_{|T|+1}) \tau_\infty (T+m)(t,z) \;\rd z_{|T|+1} \bigg] \Bigg\}.
\end{aligned}
\end{equation*}
\end{proof}

\subsection{Quantitative stability} \label{subsec:quantitative}
This subsection focuses on the proof of the main quantitative estimate of the article. The technical Lemma~\ref{lem:differential_inequality_recursion} about recursive differential inequalities is given separately in the next subsection.

\begin{proof}  [Proof of Theorem~\ref{thm:stable_power_norm}]


For simplicity, let us recall the notation 
\begin{equation*}
\begin{aligned}
\nu_m = 1 \otimes \dots \otimes \nu \otimes \dots \otimes,
\end{aligned}
\end{equation*}
where $\nu$ appears in the $m$-th coordinate, i.e. $\nu_m(z) = \nu(z_m)$. The same convention applies to $\mu$ and $\eta$.

Define the difference $\Delta_N (T) (t,z) \defeq \tau_N (T) (t,z) - \tau_\infty (T) (t,z)$. 
By subtracting \eqref{eqn:limit_hierarchy_equation} from \eqref{eqn:hierarchy_equation}, one has that 
\begin{equation*}
\begin{aligned}
& \partial_t \Delta_N (T)(t,z) \\
&\quad = \sum_{m = 1}^{|T|} \Bigg\{ \bigg[ - \partial_{z_m}(\mu(z_m) \Delta_N (T)(t,z)) + \frac{\sigma^2}{2} \partial_{z_m}^2 \Delta_N (T)(t,z) \\
&\qquad - \nu(z_m) \Delta_N (T)(t,z) + \delta_0(z_m) \bigg( \int_{\R} \nu(u_m) \Big( \Delta_N (T)(t,u) + \mathscr{R}_{N,T,m} (t,u) \Big) \;\rd u_m \bigg)\bigg|_{\forall n \neq m,\, u_n = z_n} \bigg] \\
&\qquad - \partial_{z_m} \bigg[ \int_{\R} \nu(z_{|T|+1}) \Big( \Delta_N (T+m)(t,z) + \mathscr{\tilde R}_{N,T+m,|T|+1} (t,z) \Big) \;\rd z_{|T|+1} \bigg] \Bigg\}, \quad \forall T \in \mathcal{T}.
\end{aligned}
\end{equation*}
We highlight that, for any fixed $N < \infty$, the above equalities and later inequalities involving $\Delta_N (T)$ can be understood as recursive relations that holds on all $T \in \mathcal{T}$.
At a first glance, one may think that the approximate hierarchy \eqref{eqn:hierarchy_equation} is only defined for observables $\tau_N(T)$ with $|T| \leq N$.
Nevertheless, by our formal definition that $f_{N}^{i_1,\dots, i_k} \equiv 0$ if there are duplicated indices among $i_1,\dots,i_k$, it is easy to verify that for any tree $T$ such that $|T| > N$,
\begin{equation*}
\begin{aligned}
\tau_N (T,w_N,f_N)(t,z) \defeq \frac{1}{N} \sum_{i_1,\dots, i_{|T|} = 1}^N w_{N,T}(i_1,\dots, i_{|T|}) f_{N}^{i_1,\dots, i_{|T|}}(t, z_1,\dots,z_{|T|})
\equiv 0
\end{aligned}
\end{equation*}
as in each marginal there must be duplicated indices. By a similar discussion, we see that $\mathscr{R}_{N,T,m} \equiv 0$ and $\mathscr{\tilde R}_{N,T+m,|T|+1} \equiv 0$ when $|T| > N$. With these formal definition, it is then straightforward to show that approximate hierarchy \eqref{eqn:hierarchy_equation} holds for all $T \in \mathcal{T}$.

By multiplying by the weight function $\eta^{\otimes |T|}$ and integrating, we obtain that
\begin{equation*}
\begin{aligned}
& \Big( \partial_t \Delta_N (T)(t,z) \Big) \eta^{\otimes |T|}(z) =  \sum_{m = 1}^{|T|} \Bigg\{ - \partial_{z_m}\Big(\mu_m \Delta_N (T) \eta^{\otimes |T|}\Big)(t,z) + \frac{\sigma^2}{2} \partial_{z_m}^2 \Big( \Delta_N (T) \eta^{\otimes |T|}\Big)(t,z) \\
&\qquad - \Big(\nu_m \Delta_N (T) \eta^{\otimes |T|}\Big)(t,z) + \Big( \mu_m (\eta_m'/\eta_m) \Delta_N (T) \eta^{\otimes |T|} \Big)(t,z)\\
&\qquad + \delta_0(z_m) \eta(z_m) \bigg( \int_{\R} \Big( (\nu_m/\eta_m) ( \Delta_N (T) + \mathscr{R}_{N,T,m} ) \eta^{\otimes |T|} \Big)(t,u) \;\rd u_m \bigg)\bigg|_{\forall n \neq m,\, u_n = z_n}
\\
&\qquad - \partial_{z_m} \bigg[ \int_{\R} \Big( (\nu_{|T|+1}/\eta_{|T|+1}) ( \Delta_N (T+m) + \mathscr{\tilde R}_{N,T+m,|T|+1} ) \eta^{\otimes |T|+1} \Big)(t,z) \;\rd z_{|T|+1} \bigg] \\
&\qquad + \frac{\sigma^2}{2} \bigg[ \partial_{z_m} \Big( -2(\eta_m'/\eta_m) \Delta_N (T) \eta^{\otimes |T|}\Big) +  (\eta_m''/\eta_m) \Delta_N (T) \eta^{\otimes |T|}
\bigg](t,z) \Bigg\}.
\end{aligned}
\end{equation*}
Substituting $\big(\partial_t \Delta_N (T)\big) \eta^{\otimes |T|}$ in the right hand side of
\begin{equation*}
\begin{aligned}
\;& \frac{\rd}{\rd t} \bigg( \frac{1}{2} \int_{\R^{|T|}} \Big( K^{\otimes |T|} \star \big(\Delta_N (T)\eta^{\otimes |T|}\big) (t,z)\Big)^2 \; \rd z \bigg)
\\
= \;& \int_{\R^{|T|}} \bigg(K^{\otimes |T|} \star \big(\Delta_N (T)\eta^{\otimes |T|}\big) (t,z)\bigg) \bigg(K^{\otimes |T|} \star \big( \partial_t \Delta_N (T)\eta^{\otimes |T|} \big) (t,z)\bigg) \; \rd z,
\end{aligned}
\end{equation*}
yields the extensive expression
\begin{equation*}
\begin{aligned}
& \frac{\rd}{\rd t} \bigg( \frac{1}{2} \int_{\R^{|T|}} \Big( K^{\otimes |T|} \star \big(\Delta_N (T)\eta^{\otimes |T|}\big) (t,z)\Big)^2 \; \rd z \bigg)
\\
&\quad = \int_{\R^{|T|}} \sum_{m = 1}^{|T|} \Bigg\{ - \frac{\sigma^2}{2} \bigg[ \partial_{z_m} K^{\otimes |T|} \star \big( \Delta_N (T)\eta^{\otimes |T|} \big) (t,z) \bigg]^2
\\
&\qquad + \bigg[ K^{\otimes |T|} \star \big( \Delta_N (T)\eta^{\otimes |T|} \big) (t,z) \bigg] \bigg[- K^{\otimes |T|} \star \big( \nu_m \Delta_N (T)\eta^{\otimes |T|} \big)(t,z)
\\
&\qquad + K(z_m)\eta(0) \bigg( \int_{\R} K^{\otimes |T|} \star \big( (\nu_m/\eta_m) \Delta_N (T) \eta^{\otimes |T|} \big)(t,u) \;\rd u_m \bigg)\bigg|_{\forall n \neq m,\, u_n = z_n}
\\
&\qquad + K(z_m)\eta(0) \bigg( \int_{\R} K^{\otimes |T|} \star \big( (\nu_m/\eta_m) \mathscr{R}_{N,T,m} \eta^{\otimes |T|} \big) (t,u) \;\rd u_m \bigg)\bigg|_{\forall n \neq m,\, u_n = z_n}
\\
&\qquad + K^{\otimes |T|} \star \big( \mu_m (\eta_m'/\eta_m) \Delta_N (T) \eta^{\otimes |T|} \big)(t,z) + \frac{\sigma^2}{2} K^{\otimes |T|} \star \big( (\eta_m''/\eta_m) \Delta_N (T) \eta^{\otimes |T|} \big)(t,z)
\bigg]
\\
&\qquad + \bigg[ \partial_{z_m} K^{\otimes |T|} \star \big( \Delta_N (T) \eta^{\otimes |T|} \big) (t,z) \bigg] \bigg[ K^{\otimes |T|} \star \big( \mu_m \Delta_N (T) \eta^{\otimes |T|} \big)(t,z)
\\
&\qquad + \int_{\R} K^{\otimes |T|+1} \star \big( (\nu_{|T|+1}/\eta_{|T|+1}) \Delta_N (T+m) \eta^{\otimes |T|+1} \big) (t,z) \;\rd z_{|T|+1}
\\
&\qquad + \int_{\R} K^{\otimes |T|+1} \star \big( (\nu_{|T|+1}/\eta_{|T|+1}) \mathscr{\tilde R}_{N,T+m,|T|+1} \eta^{\otimes |T|+1} \big) (t,z) \;\rd z_{|T|+1}
\\
&\qquad + \frac{\sigma^2}{2} K^{\otimes |T|} \star \big( 2(\eta_m'/\eta_m) \Delta_N (T) \eta^{\otimes |T|} \big)(t,z)
\bigg] \Bigg\}
 \; \rd z.
\end{aligned}
\end{equation*}
We then apply Cauchy-Schwartz inequality to obtain,
\begin{equation} \label{eqn:energy_Cauchy_Schwartz}
\begin{aligned}
& \frac{\rd}{\rd t} \bigg( \frac{1}{2} \|\Delta_N (T)\|_{H^{-1 \otimes |T|}_\eta}^2 \bigg)
\leq \sum_{m = 1}^{|T|} \Bigg\{ \Big(2 + \frac{\sigma^2}{4}\Big) \|\Delta_N (T)\|_{H^{-1 \otimes |T|}_\eta}^2 + \frac{1}{2} \|\nu_m \Delta_N (T)\|_{H^{-1 \otimes |T|}_\eta}^2 
\\
&\quad + \frac{1}{2}\|\mu_m (\eta_m'/\eta_m) \Delta_N (T)\|_{H^{-1 \otimes |T|}_\eta}^2 + \frac{\sigma^2}{4} \|(\eta_m''/\eta_m)\Delta_N (T)\|_{H^{-1 \otimes |T|}_\eta}^2
\\
&\quad + \frac{1}{2} \|K\|_{L^2}^2 \eta(0)^2 \int_{\R^{|T|-1}} \bigg( \int_{\R} K^{\otimes |T|} \star \big( (\nu_m/\eta_m) \Delta_N (T) \eta^{\otimes |T|} \big)(t,z) \;\rd z_m \bigg)^2 \prod_{n \neq m} \;\rd z_n
\\
&\quad + \frac{1}{2} \|K\|_{L^2}^2 \eta(0)^2 \int_{\R^{|T|-1}} \bigg( \int_{\R} K^{\otimes |T|} \star \big( (\nu_m/\eta_m) \mathscr{R}_{N,T,m} \eta^{\otimes |T|} \big)(t,z) \;\rd z_m \bigg)^2 \prod_{n \neq m} \;\rd z_n
\\
&\quad + \frac{2}{\sigma^2} \|\mu_m \Delta_N (T)\|_{H^{-1 \otimes |T|}_\eta}^2 + \frac{\sigma^2}{2} \|2(\eta_m'/\eta_m) \Delta_N (T)\|_{H^{-1 \otimes |T|}_\eta}^2
\\
&\quad + \frac{2}{\sigma^2} \int_{\R^{|T|}} \bigg( \int_{\R} K^{\otimes |T|+1} \star \big( (\nu_{|T|+1}/\eta_{|T|+1}) \Delta_N (T+m) \eta^{\otimes |T|+1} \big) (t,z) \;\rd z_{|T|+1} \bigg)^2 \prod_{n=1}^{|T|} \;\rd z_n
\\
&\quad + \frac{2}{\sigma^2} \int_{\R^{|T|}} \bigg( \int_{\R} K^{\otimes |T|+1} \star \big( (\nu_{|T|+1}/\eta_{|T|+1}) \mathscr{\tilde R}_{N,T+m,|T|+1} \eta^{\otimes |T|+1} \big) (t,z) \;\rd z_{|T|+1} \bigg)^2 \prod_{n=1}^{|T|} \;\rd z_n \Bigg\}.
\end{aligned}
\end{equation}
This is where the proper choice of weak distance becomes critical as we need to bound the various terms in the right-hand side by the norm $ \|\Delta_N (T)\|_{H^{-1 \otimes |T|}_\eta}^2$. The commutator estimate in Lemma~\ref{lem:commutator_inequality} can directly bound all the terms with an explicit $H^{-1 \otimes |T|}_\eta$-norms as the coefficients $\mu,\nu$ are $W^{1,\infty}$ and $\eta$ is smooth. For example
\[
\|\nu_m \Delta_N (T)\|_{H^{-1 \otimes |T|}_\eta}^2\leq 4\,\|\nu\|_{W^{1,\infty}(\R)}^2\,\|\Delta_N (T)\|_{H^{-1 \otimes |T|}_\eta}^2.
\]
This leads to the simplified expression for some constant $\tilde C_0$,
\begin{equation} \label{eqn:energy_Cauchy_Schwartz2}
\begin{aligned}
& \frac{\rd}{\rd t} \bigg( \frac{1}{2} \|\Delta_N (T)\|_{H^{-1 \otimes |T|}_\eta}^2 \bigg)
\leq \sum_{m = 1}^{|T|} \Bigg\{ \tilde C_0\, \|\Delta_N (T)\|_{H^{-1 \otimes |T|}_\eta}^2 
\\
&\quad + \frac{1}{2} \|K\|_{L^2}^2 \eta(0)^2 \int_{\R^{|T|-1}} \bigg( \int_{\R} K^{\otimes |T|} \star \big( (\nu_m/\eta_m) \Delta_N (T) \eta^{\otimes |T|} \big)(t,z) \;\rd z_m \bigg)^2 \prod_{n \neq m} \;\rd z_n
\\
&\quad + \frac{1}{2} \|K\|_{L^2}^2 \eta(0)^2 \int_{\R^{|T|-1}} \bigg( \int_{\R} K^{\otimes |T|} \star \big( (\nu_m/\eta_m) \mathscr{R}_{N,T,m} \eta^{\otimes |T|} \big)(t,z) \;\rd z_m \bigg)^2 \prod_{n \neq m} \;\rd z_n
\\
&\quad + \frac{2}{\sigma^2} \int_{\R^{|T|}} \bigg( \int_{\R} K^{\otimes |T|+1} \star \big( (\nu_{|T|+1}/\eta_{|T|+1}) \Delta_N (T+m) \eta^{\otimes |T|+1} \big) (t,z) \;\rd z_{|T|+1} \bigg)^2 \prod_{n=1}^{|T|} \;\rd z_n
\\
&\quad + \frac{2}{\sigma^2} \int_{\R^{|T|}} \bigg( \int_{\R} K^{\otimes |T|+1} \star \big( (\nu_{|T|+1}/\eta_{|T|+1}) \mathscr{\tilde R}_{N,T+m,|T|+1} \eta^{\otimes |T|+1} \big) (t,z) \;\rd z_{|T|+1} \bigg)^2 \prod_{n=1}^{|T|} \;\rd z_n \Bigg\}.
\end{aligned}
\end{equation}
The remaining integrals terms in \eqref{eqn:energy_Cauchy_Schwartz} can be bounded by first applying Lemma~\ref{lem:firing_rate_difference_tensorized} followed by Proposition~\ref{prop:translation_estimate_weighted}.
For example, consider the first remainder term and write by Lemma~\ref{lem:firing_rate_difference_tensorized},
\begin{equation*}
\begin{aligned}
& \int_{\R^{|T|-1}} \bigg( \int_{\R} K^{\otimes |T|} \star \big( (\nu_m/\eta_m) \mathscr{R}_{N,T,m} \eta^{\otimes |T|} \big)(t,z) \;\rd z_m \bigg)^2 \prod_{n \neq m} \;\rd z_n
\\
&\qquad \leq  C (\alpha)^2 \|\nu\|_{W^{1,\infty}}^2 \|\mathscr{R}_{N,T,m}\|_{H^{-1 \otimes |T|}_\eta}^2.
\end{aligned}
\end{equation*}
Next, apply Proposition~\ref{prop:translation_estimate_weighted} to the right hand side to conclude that
\begin{equation*}
\begin{aligned}
& \int_{\R^{|T|-1}} \bigg( \int_{\R} K^{\otimes |T|} \star \big( (\nu_m/\eta_m) \mathscr{R}_{N,T,m} \eta^{\otimes |T|} \big)(t,z) \;\rd z_m \bigg)^2 \prod_{n \neq m} \;\rd z_n
\\
&\qquad\leq  C (\alpha)^2 \|\nu\|_{W^{1,\infty}}^2 \big[ \exp \big( (2 + 2\alpha) c(w,|T|) \big) - 1 \big] \||\tau_N|(T)\|_{H^{-1 \otimes |T|}_\eta}^2.
\end{aligned}
\end{equation*}
The method applies for the other integrals terms in \eqref{eqn:energy_Cauchy_Schwartz2}, which yields
\begin{equation*}
\begin{aligned}
& \int_{\R^{|T|-1}} \bigg( \int_{\R} K^{\otimes |T|} \star \big( (\nu_m/\eta_m) \Delta_N (T) \eta^{\otimes |T|} \big)(t,z) \;\rd z_m \bigg)^2 \prod_{n \neq m} \;\rd z_n
\\
&\qquad \leq  C (\alpha)^2 \|\nu\|_{W^{1,\infty}}^2 \|\Delta_N (T)\|_{H^{-1 \otimes |T|}_\eta}^2,
\\
& \int_{\R^{|T|}} \bigg( \int_{\R} K^{\otimes |T|+1} \star \big( (\nu_{|T|+1}/\eta_{|T|+1}) \Delta_N (T+m) \eta^{\otimes |T|+1} \big) (t,z) \;\rd z_{|T|+1} \bigg)^2 \prod_{n=1}^{|T|} \;\rd z_n
\\
&\qquad \leq C (\alpha)^2 \|\nu\|_{W^{1,\infty}}^2 \|\Delta_N (T+m)\|_{H^{-1 \otimes (|T|+1)}_\eta}^2,
\end{aligned}
\end{equation*}
together with
\begin{equation*}
\begin{aligned}
& \int_{\R^{|T|}} \bigg( \int_{\R} K^{\otimes |T|+1} \star \big( (\nu_{|T|+1}/\eta_{|T|+1}) \mathscr{\tilde R}_{N,T+m,|T|+1} \eta^{\otimes |T|+1} \big) (t,z) \;\rd z_{|T|+1} \bigg)^2 \prod_{n=1}^{|T|} \;\rd z_n
\\
&\qquad \leq  C (\alpha)^2 \|\nu\|_{W^{1,\infty}}^2 \big[ \exp \big( (2 + 2\alpha) c(w,|T|) \big) - 1 \big] \||\tau_N|(T)\|_{H^{-1 \otimes |T|}_\eta}^2.
\end{aligned}
\end{equation*}
Inserting those bounds into the energy estimate \eqref{eqn:energy_Cauchy_Schwartz2}, we obtain a recursive differential inequality: for all $T \in \mathcal{T}$,
\begin{equation} \label{eqn:differential_inequality_hierarchy}
\begin{aligned}
& \frac{\rd}{\rd t} \|\Delta_N (T) (t,\cdot) \|_{H^{-1 \otimes |T|}_\eta}^2
\leq  \sum_{m = 1}^{|T|} \Bigg\{ \tilde C_0 \| \Delta_N (T) (t,\cdot) \|_{H^{-1 \otimes |T|}_\eta}^2 + \tilde C_1 \| \Delta_N (T+m) (t,\cdot) \|_{H^{-1 \otimes (|T|+1)}_\eta}^2
\\
& \qquad + \varepsilon_0(T) \| |\tau_N| (T) (t,\cdot) \|_{H^{-1 \otimes |T|}_\eta}^2 + \varepsilon_1(T) \| |\tau_N| (T+m) (t,\cdot) \|_{H^{-1 \otimes (|T|+1)}_\eta}^2 \Bigg\},
\end{aligned}
\end{equation}
where we can even provide the explicit expressions for the constants
\begin{equation*}
\begin{aligned}
\;& \tilde C_0 = 4 + \frac{\sigma^2}{2} + 4\bigg(\|\nu\|_{W^{1,\infty}}^2 + \|\mu(\eta'/\eta)\|_{W^{1,\infty}}^2 + \frac{\sigma^2}{2} \|\eta''/\eta\|_{W^{1,\infty}}^2 + \frac{4}{\sigma^2} \|\mu\|_{W^{1,\infty}}^2 + 2\sigma^2 \|(\eta'/\eta)\|_{W^{1,\infty}}^2 \bigg)
\\
\;& \quad\quad + \|K\|_{L^2}^2 \eta(0)^2 C (\alpha)^2 \|\nu\|_{W^{1,\infty}}^2,
\\
\;& \tilde C_1 = \frac{4C(\alpha)^2}{\sigma^2} \|\nu\|_{W^{1,\infty}}^2,
\\
\;& \varepsilon_0(T) = \|K\|_{L^2}^2 \eta(0)^2 C (\alpha)^2 \|\nu\|_{W^{1,\infty}}^2 \big[ \exp \big( (2 + 2\alpha) c(w,|T|) \big) - 1 \big],
\\
\;& \varepsilon_1(T) = \frac{4 C (\alpha)^2}{\sigma^2} \|\nu\|_{W^{1,\infty}}^2 \big[ \exp \big( (2 + 2\alpha) c(w,|T|) \big) - 1 \big].
\end{aligned}
\end{equation*}
We can now restrict the recursion relations by truncating them at any given depth $n \geq 1$, meaning that we only consider the inequalities \eqref{eqn:differential_inequality_hierarchy} for all $T \in \mathcal{T}$ such that $|T| \leq n - 1$.
In such a case, since
\begin{equation*}
\begin{aligned}
c(w,|T|) \leq |T| \big(\max_{i,j}|w_{i,j;N}| \big) \leq n \bar w_N,
\end{aligned}
\end{equation*}
the coefficients $\varepsilon_0$, $\varepsilon_1$ can take the vanishing expression
\begin{equation*}
\begin{aligned}
\;& \varepsilon_0(n) = \|K\|_{L^2}^2 \eta(0)^2 C (\alpha)^2 \|\nu\|_{W^{1,\infty}}^2 \big[ \exp \big( (2 + 2\alpha) n \bar w_N \big) - 1 \big],
\\
\;& \varepsilon_1(n) = \frac{4 C (\alpha)^2}{\sigma^2} \|\nu\|_{W^{1,\infty}}^2 \big[ \exp \big( (2 + 2\alpha) n \bar w_N \big) - 1 \big].
\end{aligned}
\end{equation*}
For a fixed depth $n \geq 1$, $\varepsilon_0(n)$ and $\varepsilon_1(n)$ now vanish as $\bar w_N \to 0$. 

Let us now rescale the energy inequality through some $\lambda^{|T|}$ factor: For all $T \in \mathcal{T}$ such that $|T| \leq n - 1$,
\begin{equation} \label{eqn:recursive_energy_estimate}
\begin{aligned}
& \frac{\rd}{\rd t} \lambda^{|T|} \|\Delta_N (T) (t,\cdot) \|_{H^{-1 \otimes |T|}_\eta}^2
\\
&\quad \leq  \sum_{m = 1}^{|T|} \Bigg\{ \tilde C_0 \lambda^{|T|} \| \Delta_N (T) (t,\cdot) \|_{H^{-1 \otimes |T|}_\eta}^2 + (\tilde C_1/\lambda) \lambda^{|T| + 1} \| \Delta_N (T+m) (t,\cdot) \|_{H^{-1 \otimes (|T|+1)}_\eta}^2
\\
&\qquad + \varepsilon_0(n) \lambda^{|T|} \| |\tau_N| (T) (t,\cdot) \|_{H^{-1 \otimes |T|}_\eta}^2 + (\varepsilon_1(n)/\lambda) \lambda^{|T| + 1} \| |\tau_N| (T+m) (t,\cdot) \|_{H^{-1 \otimes (|T|+1)}_\eta}^2 \Bigg\}.
\end{aligned}
\end{equation}
We also recall the a priori bound \eqref{eqn:hierarchy_boundedness_1} for $\tau_N,\tau_\infty$ assumed in Theorem~\ref{thm:stable_power_norm}:
\begin{equation*} 
\sup_{t\leq t_*} \ \max_{|T| \leq \max(n,\ |T_*|)} \lambda^{\frac{|T|}{2}}\, \left(\||\tau_N|(T,w_N,f_N)(t,\cdot)\|_{H^{-1\otimes |T|}_\eta} 
+\|\tau_\infty(T)(t,\cdot)\|_{H^{-1\otimes |T|}_\eta}\right) \leq C_{\lambda;\eta}, 
\end{equation*}
where $T_* \in \mathcal{T}$ is the tree index in the final estimate \eqref{eqn:stable_power_norm}.
By a triangle inequality, this implies the following uniform bound of $\Delta_N$,
\begin{equation} \label{eqn:uniform_energy_bound}
\begin{aligned}
\sup_{t\leq t_*} \ \max_{|T| \leq \max(n,\ |T_*|)} \lambda^{|T|} \|\Delta_N (T) (t,\cdot) \|_{H^{-1\otimes |T|}_\eta}^2 \leq C_{\lambda;\eta}^2.
\end{aligned}
\end{equation}

Denote
\begin{equation*}
\begin{aligned}
M_k(t) = \;& \max_{|T| \leq k} \lambda^{|T|} \|\Delta_N (T) (t,\cdot) \|_{H^{-1\otimes |T|}_\eta}^2,
\\
C = \;& \tilde C_0 + \tilde C_1/\lambda,
\\
\varepsilon = \;& \big[ \varepsilon_0(n) + \varepsilon_1(n)/\lambda \big] C_{\lambda;\eta}^2,
\\
L = \;& C_{\lambda;\eta}^2,
\\
n = \;& n, \quad n' = |T_*|,
\end{aligned}
\end{equation*}
so that \eqref{eqn:recursive_energy_estimate} and \eqref{eqn:uniform_energy_bound} can be summarized as follows,
\begin{subequations}
\begin{align} \label{eqn:differential_inequality_recursion}
\frac{\rd}{\rd t} M_k(t)
\leq \;& k \Big( C M_{k+1}(t) + \varepsilon \Big), && \forall 1 \leq k \leq n-1,
\\ \label{eqn:bound_inequality_recursion}
M_k(t) \leq \;& L, && \forall 1 \leq k \leq \max(n,\ n'), \; t \in [0,t_*].
\end{align}
\end{subequations}
We now invoke the following result.
\begin{lem} \label{lem:differential_inequality_recursion}
Consider a sequence of non-negative functions $(M_k(t))_{k=1}^\infty$ on $t \in [0,t_*]$ that satisfies the inequalities
\eqref{eqn:differential_inequality_recursion}-\eqref{eqn:bound_inequality_recursion} with $\big[ \varepsilon/CL + (2\theta)^n \big] \leq 1$.
Then
\begin{equation} \label{eqn:convergence_inequality_hierarchy}
\begin{aligned}
\max_{1 \leq k \leq \max(n , \ n')} \big[ \theta^k M_k(t) \big] 
\leq \;& L (Ct/\theta + 2) \,\max\left( \big[\varepsilon/CL + (2\theta)^n \big],\ \max_{1 \leq k \leq n - 1} \big[ \theta^k M_k(0) \big] / L \right)^{\frac{1}{p^{(Ct/\theta + 1)}}},
\end{aligned}
\end{equation}
holds for any $1 < p < \infty$, $0 < \theta < 2^{-p'}$ where $1/p + 1/p' = 1$, and any $t \in [0,t_*]$.
\end{lem}
Assume for the time being that Lemma~\ref{lem:differential_inequality_recursion} holds and apply it to \eqref{eqn:recursive_energy_estimate} and \eqref{eqn:uniform_energy_bound}. 
Choose $p = 2$, $\theta = 1/8$ and substitute $\varepsilon, C, L$ by its explicit expression to find that
\begin{equation*}
\begin{aligned}
\varepsilon/CL = \frac{\varepsilon_0(n) + \varepsilon_1(n)/\lambda}{\tilde C_0 + \tilde C_1/\lambda} = C_1 \big[ \exp \big( (2 + 2\alpha) \bar w n \big) - 1 \big],
\end{aligned}
\end{equation*}
where $C_1$ depends only on $\lambda$, the $W^{1,\infty}$-regularity of coefficients $\mu$, $\nu$ and constant $\sigma > 0$ in \eqref{eqn:IF_Liouville_PDE}, but neither on $\bar w_N$ nor on $n$. Choosing $C_0 = C/\theta$, and as $\bar w_N\to 0$ as $N\to\infty$, we deduce that for $N$ large enough
\begin{equation*}
\begin{aligned}
\bar{\varepsilon} = \varepsilon/CL + (2\theta)^n = C_1 \big[ \exp \big( (2 + 2\alpha) n \bar w_N \big) - 1 \big] + (1/4)^n \leq 1.
\end{aligned}
\end{equation*}
The conclusion of Lemma~\ref{lem:differential_inequality_recursion} hence holds, showing that
\begin{equation*}
\begin{aligned}
& \max_{|T| \leq \max(n , \ |T_*|)} (\lambda / 8)^{|T|} \| \tau_N(T,w_N,f_N)(t,\cdot) - \tau_\infty(T)(t,\cdot) \|_{H^{-1\otimes |T|}_\eta}^2
\\
&\quad \leq  C_{\lambda;\eta}^2\, \Big( C_0 t + 2 \Big)\,\max \left( \bar{\varepsilon},\  \max_{|T| \leq n-1} (\lambda / 8)^{|T|} \| \tau_N(T,w_N,f_N)(0,\cdot) - \tau_\infty(T)(0,\cdot) \|_{H^{-1\otimes |T|}_\eta}^2 / C_{\lambda;\eta}^2 \right)^{\frac{1}{2^{(C_0 t + 1)}}}.
\end{aligned}
\end{equation*}
This can be further simplified to \eqref{eqn:stable_power_norm} by relaxing the maximum on the left hand side as $T = T_*$, taking the maximum on the right hand side over $|T| \leq \max(n,\ |T_*|)$, and choosing $C_2$ in $\eqref{eqn:stable_power_norm}$ as $C_2 = \max\big(C_0 t + 2, \ 2^{(C_0 t + 1)} \big)$.

\end{proof}

\subsection{Proof of Lemma~\ref{lem:differential_inequality_recursion} } \label{subsec:differential_inequality_hierarchy}
\begin{proof} [Proof of Lemma~\ref{lem:differential_inequality_recursion}]

Let us restate here the recursive differential inequality \eqref{eqn:differential_inequality_recursion},
\begin{equation*}
\begin{aligned}
\frac{\rd}{\rd t} M_k(t)
\leq \;& k \Big( C M_{k+1}(t) + \varepsilon \Big), && \forall 1 \leq k \leq n-1,
\end{aligned}
\end{equation*}
which directly yields
\begin{equation*}
\begin{aligned}
\frac{\rd}{\rd t} \Big( M_k(t) + (\varepsilon/C) \Big)
\leq k C \Big( M_{k+1}(t) + (\varepsilon/C) \Big), && \forall 1 \leq k \leq n-1.
\end{aligned}
\end{equation*}
For any $1 \leq k \leq n - 1$ and $t \in [0,t_*]$, by inductively integrating the inequalities in time, we obtain that
\begin{equation*}
\begin{aligned}
\Big( M_k(t) + (\varepsilon/C) \Big) \leq \;&
C^{n - k} \int_s^t \begin{pmatrix} n-1 \\ k-1 \end{pmatrix} \frac{(t - r)^{n-k-1}}{n-k-1} \Big( M_n(r) + (\varepsilon/C) \Big) \;\rd r
\\
\;& + \sum_{l = k}^{n-1} C^{l - k} \begin{pmatrix} l-1 \\ k-1 \end{pmatrix} (t - s)^{l-k} \Big( M_l(s) + (\varepsilon/C) \Big),
\end{aligned}
\end{equation*}
We estimate the increase on $M_k$ within time steps of size
\begin{equation*}
\begin{aligned}
t-s = \theta / C.
\end{aligned}
\end{equation*}
First, we bound the constant terms,
\begin{equation*}
\begin{aligned}
& C^{n - k} \int_s^t \begin{pmatrix} n-1 \\ k-1 \end{pmatrix} \frac{(t - r)^{n-k-1}}{n-k-1} (\varepsilon/C) \;\rd r
+ \sum_{l = k}^{n-1} C^{l - k} \begin{pmatrix} l-1 \\ k-1 \end{pmatrix} (t - s)^{l-k} (\varepsilon/C)
\\
&\quad = (\varepsilon/C) \bigg\{ C^{n - k} \begin{pmatrix} n-1 \\ k-1 \end{pmatrix} (t - s)^{n-k-1}
+ \sum_{l = k}^{n-1} C^{l - k} \begin{pmatrix} l-1 \\ k-1 \end{pmatrix} (t - s)^{l-k} \bigg\}
\\
&\quad = (\varepsilon/C) \sum_{l = k}^n C^{l - k} \begin{pmatrix} l-1 \\ k-1 \end{pmatrix} (t - s)^{l-k}
\\
&\quad \leq \theta^{-k} (\varepsilon/C) \sum_{l = k}^n \begin{pmatrix} l-1 \\ k-1 \end{pmatrix} \theta^l,
\end{aligned}
\end{equation*}
where the last inequality uses our choice of time step $(t-s) \leq \theta / C$.

On the other hand, for $\theta \leq 1/2$,
\begin{equation*}
\begin{aligned}
\sum_{l = k}^{\infty} \begin{pmatrix} l-1 \\ k-1 \end{pmatrix} \theta^l = \frac{1}{(\theta^{-1} - 1)^{k}} \leq 1.
\end{aligned}
\end{equation*}
Hence,
\[
C^{n - k} \int_s^t \begin{pmatrix} n-1 \\ k-1 \end{pmatrix} \frac{(t - r)^{n-k-1}}{n-k-1} (\varepsilon/C) \;\rd r
+ \sum_{l = k}^{n-1} C^{l - k} \begin{pmatrix} l-1 \\ k-1 \end{pmatrix} (t - s)^{l-k} (\varepsilon/C)\leq \theta^{-k} \,\frac{\varepsilon}{C}.
\]
We now turn to the terms involving $M_l(s)$ and $M_n(r)$ (with $s \leq r \leq t$).
For $M_n(r)$ we have no choice but to take
\begin{equation*}
\begin{aligned}
M_n(r) \leq L
\end{aligned}
\end{equation*}
But for $M_l(s)$, $k \leq l \leq n-1$, we have
\begin{equation*}
\begin{aligned}
M_l(s) \leq \min\left(L,\ \max_{1 \leq m \leq n-1} \big[ \theta^m M_m(s) \big] \theta^{-l}\right),
\end{aligned}
\end{equation*}
together with any geometric average between the two terms. Choose $\frac{1}{p} + \frac{1}{p'} = 1$ so that
\begin{equation*}
\begin{aligned}
M_l(s) \leq \;& L^{\frac{1}{p'}} \Big( \max_{1 \leq m \leq n-1} \big[ \theta^m M_m(s) \big] \theta^{-l} \Big)^{\frac{1}{p}}
\\
= \;& L^{\frac{1}{p'}} \max_{1 \leq m \leq n-1} \big[ \theta^m M_m(s) \big]^{\frac{1}{p}} \big( \theta^{\frac{1}{p}} \big)^{-l}.
\end{aligned}
\end{equation*}
Then we may write
\begin{equation*}
\begin{aligned}
& C^{n - k} \int_s^t \begin{pmatrix} n-1 \\ k-1 \end{pmatrix} \frac{(t - r)^{n-k-1}}{n-k-1} M_n(r) \;\rd r
+ \sum_{l = k}^{n-1} C^{l - k} \begin{pmatrix} l-1 \\ k-1 \end{pmatrix} (t - s)^{l-k} M_l(s),
\\
&\quad\leq C^{n - k} \begin{pmatrix} n-1 \\ k-1 \end{pmatrix} (t - s)^{n-k} L
+ \sum_{l = k}^{n-1} C^{l - k} \begin{pmatrix} l-1 \\ k-1 \end{pmatrix} (t - s)^{l-k} L^{\frac{1}{p'}} \max_{1 \leq m \leq n-1} \big[ \theta^m M_m(s) \big]^{\frac{1}{p}} \big( \theta^{\frac{1}{p}} \big)^{-l}
\\
&\quad \leq  \theta^{-k} \bigg\{ L \begin{pmatrix} n-1 \\ k-1 \end{pmatrix} \theta^n + L^{\frac{1}{p'}} \max_{1 \leq m \leq n-1} \big[ \theta^m M_m(s) \big]^{\frac{1}{p}} \sum_{l = k}^{n-1} \begin{pmatrix} l-1 \\ k-1 \end{pmatrix} \big( \theta^{\frac{1}{p'}} \big)^{-l} \bigg\},
\end{aligned}
\end{equation*}
where again use our choice of time step $(t-s) \leq \theta / C$ in the last inequality.

Observe that
\begin{equation*}
\begin{aligned}
\begin{pmatrix} n-1 \\ k-1 \end{pmatrix} \theta^n \leq 2^{n-1} \theta^n, \quad \sum_{l = k}^{n-1} \begin{pmatrix} l-1 \\ k-1 \end{pmatrix} \big( \theta^{\frac{1}{p'}} \big)^{-l} \leq \frac{1}{(\theta^{-\frac{1}{p'}} - 1)^{k}} \leq 1
\end{aligned}
\end{equation*}
when choosing $\theta^{\frac{1}{p'}} \leq 1/2$, so that
\[\begin{split}
&C^{n - k} \int_s^t \begin{pmatrix} n-1 \\ k-1 \end{pmatrix} \frac{(t - r)^{n-k-1}}{n-k-1} M_n(r) \;\rd r
+ \sum_{l = k}^{n-1} C^{l - k} \begin{pmatrix} l-1 \\ k-1 \end{pmatrix} (t - s)^{l-k} M_l(s)\\
&\qquad\leq \theta^{-k}\,\left(L (2\theta)^n + L^{\frac{1}{p'}} \max_{1 \leq m \leq n - 1} \big[ \theta^m M_m(s) \big]^{\frac{1}{p}} \right).
\end{split}\]

Combining those bounds, provided that $\theta^{\frac{1}{p'}} \leq 1/2$, we have that, for all $1 \leq k \leq n-1$,
\begin{equation*}
\begin{aligned}
M_k(t) \leq \;& \theta^{-k} \bigg\{(\varepsilon/C) + L (2\theta)^n + L^{\frac{1}{p'}} \max_{1 \leq m \leq n - 1} \big[ \theta^m M_m(s) \big]^{\frac{1}{p}} \bigg\}.
\end{aligned}
\end{equation*}
On the other hand, for $n \leq k \leq \max(n,\ n')$, we simply have $M_k(t) \leq L$. As $\theta^{-k + n} \geq 1$, 
\begin{equation*}
\begin{aligned}
M_k(t) \leq L \leq \theta^{-k} \bigg\{L (2\theta)^n \bigg\},
\end{aligned}
\end{equation*}
and we can combine the two cases to obtain that 
\begin{equation*}
\begin{aligned}
\max_{1 \leq k \leq \max(n,\ n')} \big[ \theta^k M_k(t) \big] \leq \;& (\varepsilon/C) + L (2\theta)^n + L^{\frac{1}{p'}} \max_{1 \leq k \leq n-1} \big[ \theta^k M_k(s) \big]^{\frac{1}{p}}.
\end{aligned}
\end{equation*}
If $t\leq \theta/C$ we are done but otherwise we need to sum up the various bounds. Denote $t_j=j\,\theta/C$ and write that By the fact that , we have that
\begin{equation*}
\begin{aligned}
& \max_{1 \leq k \leq \max(n,\ n')} \big[ \theta^k M_k(t_j) \big]\\
&\quad \leq  (\varepsilon/C) + L (2\theta)^n + L^{\frac{1}{p'}} \max_{1 \leq k \leq n-1} \big[ \theta^k M_k(t_{j-1}) \big]^{\frac{1}{p}}\\
&\quad \leq (\varepsilon/C) + L (2\theta)^n + L^{\frac{1}{p'}} \bigg\{ (\varepsilon/C) + L (2\theta)^n + L^{\frac{1}{p'}} \max_{1 \leq k \leq n-1} \big[ \theta^k M_k(t_{j-2}) \big]^{\frac{1}{p}} \bigg\}^{\frac{1}{p}}\\
&\quad \leq (\varepsilon/C) + L (2\theta)^n + L^{\frac{1}{p'}} \bigg\{ (\varepsilon/C) + L (2\theta)^n \bigg\}^{\frac{1}{p}} + L^{1-\frac{1}{p^2}} \max_{1 \leq k \leq n-1} \big[ \theta^k M_k(t_{j-2}) \big]^{\frac{1}{p^2}}\\ &\qquad \dots \\
&\quad \leq  \sum_{i=0}^{j-1} L^{1-\frac{1}{p^i}} \bigg\{ (\varepsilon/C) + L (2\theta)^n \bigg\}^{\frac{1}{p^i}} \; + \; L^{1-\frac{1}{p^j}} \max_{1 \leq k \leq n-1} \big[ \theta^k M_k(0) \big]^{\frac{1}{p^j}},
\end{aligned}
\end{equation*}
where we use that $(a+b)^{1/p} \leq a^{1/p} + b^{1/p}$ by concavity.

For any $t\geq 0$, we hence have with $j(t)=\bigg\lfloor \frac{C t}{\theta} \bigg\rfloor + 1$,
\begin{equation} \label{eqn:convergence_inequality_hierarchy_tight}
\begin{aligned}
\max_{1 \leq k \leq \max(n, \ n')} \big[ \theta^k M_k(t) \big] \leq \;& \sum_{i=0}^{j(t)-1} L^{1-\frac{1}{p^i}} \bigg\{ \varepsilon/C + L (2\theta)^n \bigg\}^{\frac{1}{p^i}} \; + \; L^{1-\frac{1}{p^{j(t)}}} \max_{1 \leq k \leq n-1} \big[ \theta^k M_k(0) \big]^{\frac{1}{p^{j(t)}}}.
\end{aligned}
\end{equation}
Finally, by the assumption that $\big[ \varepsilon/CL + (2\theta)^n \big] \leq 1$,
\begin{equation*}
\begin{aligned}
\forall i \leq j, \quad L^{1-\frac{1}{p^i}} \bigg\{ \varepsilon/C + L (2\theta)^n \bigg\}^{\frac{1}{p^i}} = L \bigg\{ \varepsilon/CL + (2\theta)^n \bigg\}^{\frac{1}{p^i}} \leq L \bigg\{ \varepsilon/CL + (2\theta)^n \bigg\}^{\frac{1}{p^j}}.
\end{aligned}
\end{equation*}
Hence we can replace every $i$ and every $j(t)$ in \eqref{eqn:convergence_inequality_hierarchy_tight} by $(Ct/\theta + 1)$, which gives the looser bound \eqref{eqn:convergence_inequality_hierarchy}, restated here
\begin{equation*}
\begin{aligned}
\max_{1 \leq k \leq \max(n, \ n')} \big[ \theta^k M_k(t) \big] 
\leq \;& L (Ct/\theta + 2) \max\bigg( \big[\varepsilon/CL + (2\theta)^n \big] , \sup_{1 \leq k \leq n-1} \big[ \theta^k M_k(0) \big] / L \bigg)^{\frac{1}{p^{(Ct/\theta + 1)}}}.
\end{aligned}
\end{equation*}

\end{proof}

\nocite{*}
\bibliography{IF}{} 
\bibliographystyle{siam} 

\end{document}